\documentclass[11pt]{article}

\usepackage[T1]{fontenc}
\usepackage[utf8]{inputenc}
\usepackage{lmodern}
\usepackage[a4paper,margin=1.08in]{geometry}
\usepackage{amsmath,amssymb,amsthm,mathtools,mathrsfs}
\usepackage{booktabs}
\usepackage{microtype}
\usepackage{etoolbox}
\usepackage{hyperref}
\hypersetup{
  hidelinks,
  pdftitle={Quadratic variations of rough generalized Hermite processes: hidden long memory and a universal Gaussian boundary field},
  pdfauthor={Obayda Julien Assaad},
  pdfsubject={Limit theorems for fractionally filtered generalized Hermite processes},
  pdfkeywords={generalized Hermite process, quadratic variation, Rosenblatt process, Wiener chaos, zero roughness}
}

\newtheorem{theorem}{Theorem}[section]
\newtheorem{corollary}[theorem]{Corollary}
\newtheorem{proposition}[theorem]{Proposition}
\newtheorem{lemma}[theorem]{Lemma}
\theoremstyle{definition}

\theoremstyle{remark}
\newtheorem{remark}[theorem]{Remark}
\AfterEndEnvironment{proof}{\par}

\newcommand{\R}{\mathbb R}
\newcommand{\E}{\mathbb E}
\newcommand{\dd}{\,\mathrm d}
\newcommand{\one}{\boldsymbol 1}
\newcommand{\1}{\mathbf 1}
\newcommand{\Sym}{\operatorname{Sym}}
\newcommand{\per}{\operatorname{per}}
\newcommand{\Var}{\operatorname{Var}}
\newcommand{\Cov}{\operatorname{Cov}}

\renewcommand{\P}{\mathbb P}

\title{Quadratic variations of rough generalized Hermite processes:\\
hidden long memory and a universal Gaussian boundary field}
\author{Obayda Julien Assaad\\[0.35em]
\normalsize Independent researcher, Chisinau, Moldova\\[-0.1em]
\small\href{mailto:obayda.assaad@gmail.com}%
{obayda.assaad@gmail.com}}
\date{}

\begin{document}
\maketitle

\begin{abstract}
We study centered quadratic variations of fractionally filtered generalized
Hermite processes. For every chaos order $q\ge2$, every admissible
anisotropic monomial kernel $g_{\boldsymbol\gamma}$, and every fixed
$0<h\le1/2$, we prove that $N^{2h+\theta-1}V_N$ converges to a Rosenblatt
process with self-similarity exponent $1-\theta$, without subtracting any
chaos projection. Thus the visible roughness $h$ determines the
normalization, whereas the latent singularity exponent $\theta$ determines
the memory of the limit. The proof uses a sharp paired-filter threshold and
charged diagram power counting.

At $h=0$, the bare kernel has logarithmically divergent energy.
Variance-normalized $h$-regularization and a hard endpoint cutoff have the
same residue and converge in finite-dimensional distributions to the
universal Gaussian field $G_t=(\zeta_t-\zeta_0)/\sqrt2$. Although $G$ has no
stochastically continuous modification, the iterated boundary-first
quadratic energies converge to $\sqrt3B$. In the isotropic case,
simultaneous limits $h_N\downarrow0$ are governed by
$\lambda_N=h_NN^{1/2-\theta}$: $\lambda_N\to0$ yields a Brownian limit,
$\lambda_N\to\lambda\in(0,\infty)$ yields an independent
Brownian--Rosenblatt sum, and $\lambda_N\to\infty$ yields a Rosenblatt limit
after division by $\lambda_N$. We also obtain an explicit polynomial rate
in fixed Malliavin--Sobolev norms and functional convergence in
little-H\"older spaces.
\end{abstract}



\section{Introduction}

Generalized Hermite processes are non-Gaussian self-similar processes with
stationary increments living in a fixed Wiener chaos.  They arise from the
classical noncentral-limit mechanism of Taqqu \cite{Taqqu1975,Taqqu1979};
their integral representations and the Rosenblatt member of the class are
developed in \cite{Tudor2008,PipirasTaqqu2010,Tudor2013,BaiTaqqu2014}.
Fractional filtering changes their self-similarity index without destroying
this finite-chaos representation; rough versions and related generalized
Wiener--Hermite integrals were developed in
\cite{BaiTaqqu2014,AssaadDiezTudor}.  Starting with a long-memory exponent
$H_0>1/2$, we apply a filter of order $\beta=h-H_0<0$ and observe a process
of exponent $0<h\le1/2$.  Its covariance has the scaling of a rough process
up to and including the Brownian regularity threshold.
The first question is whether a nonlinear local energy still detects the
long memory hidden by this covariance-level change of regularity.

At the singular endpoint $h=0$, where the bare kernel leaves $L^2$, a
second question is whether a canonical renormalized object exists and how
the zero-roughness limit interacts with the high-frequency limit.

Every variance-normalized member of the class has covariance
\[
 \E[X_tX_s]
 =\frac12\bigl(t^{2h}+s^{2h}-|t-s|^{2h}\bigr),
\]
independently of the chaos order $q$, the latent exponent $H_0$, and the
anisotropic kernel parameters.  The covariance therefore sees only $h$.
The centered quadratic energy sees more: its normalization involves both
$h$ and the dominant singularity exponent $\theta$, whereas the
self-similarity exponent of its limiting field is $1-\theta$ and is
independent of the filter.  This separation is the hidden-memory mechanism
studied below.

The variation theory of self-similar processes provides the immediate
background; see Tudor's monograph \cite{Tudor2013}.  For ordinary Hermite
processes of arbitrary order, the chaos expansion of the raw quadratic
variation is asymptotically governed by the maximal nontrivial contraction:
after normalization, the surviving term belongs to the second Wiener chaos
and has a Rosenblatt limit.  This reproduction phenomenon, its
Malliavin-calculus proof, and its consequences for the self-similarity
parameter were established in
\cite{ChronopoulouTudorViens2009,TudorViens2009,CTV}.  For the Rosenblatt
process, replacing first increments by longer finite-difference or wavelet
filters does not remove that leading second-chaos contribution
\cite{LongerFilters}.  Quantitative nonnormal approximation for Hermite
power variations and central/noncentral limits for weighted power variations
of fractional Brownian motion were obtained in
\cite{BretonNourdin2008,NourdinNualartTudor2010}; the anisotropic analogue for
Hermite variations of the fractional Brownian sheet was treated in
\cite{ReveillacStauchTudor2012}.  These results explain both the robustness of
Hermite limits and the sensitivity of the limiting chaos to the memory
exponents.

A complementary multiscale line replaces increments by wavelet coefficients
and squares them.  For the Rosenblatt process the wavelet statistic still has
a non-Gaussian limit even when the wavelet has several vanishing moments
\cite{BardetTudor2010}.  For nonlinear transforms of long-memory Gaussian
sequences, the large-scale wavelet coefficients and their scalograms can
select different chaotic orders according to the Hermite rank and the joint
scale--sample regime
\cite{ClauselEtAl2012,ClauselEtAl2013}.  In the pure Hermite-polynomial case
the normalized scalogram again reproduces a Rosenblatt limit
\cite{ClauselWavelet2014}, whereas quadratic variations of sums of two
consecutive Hermite processes can be Gaussian or Rosenblatt according to the
dependence and parameter regime \cite{ClauselSum2014}.  Related quadratic
variations of the fractional-colored stochastic heat equation display their
own Gaussian/non-Gaussian phase transition \cite{TorresTudorViens2014}.

Recent work changes the statistic more radically: suitably selected
increments, modified wavelet coefficients, and modified weighted power
variations recover asymptotic normality for Hermite processes
\cite{AyacheTudor2024,LoosveldtTudor2025,AyacheLoosveldtTudor2026}.
Here, by contrast, the statistic remains the unmodified quadratic variation
built from consecutive first increments.  It is the underlying generalized
Hermite process that is moved below Brownian regularity by a singular
fractional filter.  The questions are therefore different: whether the
second-chaos reproduction mechanism survives for every admissible
anisotropy in the symmetrized monomial class, how its scale separates the visible exponent $h$ from the hidden
memory exponent $\theta$, and what replaces it when the kernel reaches the
non-$L^2$ boundary $h=0$.

The classical central and noncentral limit theories for nonlinear Gaussian
functionals go back to Breuer--Major, Dobrushin--Major and Taqqu
\cite{BreuerMajor1983,DobrushinMajor1979,Taqqu1979}.  The present statistic
starts instead from a non-Gaussian process already lying in a fixed chaos.
The result controls the complete quadratic variation without subtracting
lower-chaos projections by hand, covers every admissible anisotropy in
the symmetrized monomial class at every order, identifies the complete family
$M^2$ of surviving second-chaos sectors, and supplies an explicit polynomial
rate in fixed Malliavin--Sobolev norms together with functional convergence
in little-H\"older spaces.  These features are proved in one normalization.

The paired argument below is valid throughout $\beta\in(-1,0)$ and in
particular includes the endpoint $h=1/2$.  Below $1/2$ the problem changes
character.  When
$\beta\le-1/2$, treating the two endpoint singularities separately imposes
the unusable condition $a<1+\beta$ on each Riesz edge of exponent $a$.  A
quadratic variation, however, always contains the filters in pairs.  Their
absolute autocorrelation behaves like $|z-s|^{2\beta+1}$ near
$s\in\{-1,0,1\}$ when $\beta<-1/2$, has a logarithmic singularity when
$\beta=-1/2$, and is locally bounded when $\beta>-1/2$.  In all three
cases this gives the doubled budget
\[
 a<\min\{1,2(1+\beta)\}.
\]

The apparent anisotropic restriction disappears when one keeps the port
labels throughout the closure.  Writing
$\theta_j=-2\gamma_j-1$ and assigning charge $\theta_j/2$ to the
$j$th port, every closed diagram satisfies an exact cut identity.  It makes
all compact collision inequalities strict under admissibility alone.
After the filter tails are represented by two additional roots, the same
charge count controls every macroscopic partition face.  Consequently the
fixed-positive-roughness theorem holds for every admissible anisotropic
exponent vector in the monomial class throughout $0<h\le1/2$; the Brownian threshold is included.

The leading projection is the contraction of order $q-1$, but its
normalization depends on the observed roughness:
\[
 N^{2h+\theta-1}V_N(t),
 \qquad \theta=-2\gamma_*-1.
\]
It converges to a Rosenblatt process of exponent $1-\theta$, which depends on
the hidden kernel rather than on $h$.  In the isotropic model
$\theta=2(1-H_0)/q$, so the second-chaos limit retains a quantitative trace
of the parent-chaos order even though all variance-normalized processes with
the same $h$ have identical covariance.  This is a structural separation of
visible roughness from hidden memory, not an estimation statement.

The fixed-parameter theorem does not extend continuously to $h=0$.  The
normalizing Riesz energy has a simple pole, and the bare multiple-integral
kernel leaves $L^2$.  We compute its residue and show that both
variance-normalized $h$-regularization and an endpoint cutoff select the same
Gaussian cylindrical field.  More precisely, the pole $1/(2h)$ and the cutoff divergence
$\log(1/\varepsilon)$ have the same residue.  This is the exact analogue of
the dimensional-regularization/cutoff dictionary.  Thus zero roughness is
not a missing value of the Rosenblatt theorem: it changes both the chaos and
the available path topology.

The two orders of limits, $h\downarrow0$ and $N\to\infty$, do not commute.
In the isotropic model their competition is measured exactly by
\[
 \lambda_N=h_NN^{1/2-\theta}.
\]
The active second-chaos projection has size $\lambda_N$, while the
remaining even chaoses produce a Brownian motion of variance parameter
$3$.  At the critical scale the two limits coexist and are independent.
The proof requires uniform control of translated endpoint roots on four
macroscopic blocks; it cannot be obtained by inserting $h_N$ into a
fixed-$h$ estimate.

The common algebraic skeleton, the three distinct analytic mechanisms, and
the precise role of every appendix are described in
Section~\ref{sec:proof-strategy}.  In particular, the boundary and crossover
arguments require additional uniform diagram estimates developed for the
present problem and are not formal consequences of the
fixed-positive-roughness theorem.

\section{Model and notation}

We use the convention
$\E[I_q(f)I_q(g)]=q!\langle f,g\rangle$ for symmetric kernels and the
standard Malliavin--Sobolev notation of
\cite{Nualart2006,NourdinPeccati2012}.  For multiple-integral product
formulae, contractions, cumulants, and diagram expansions we also refer to
\cite{PeccatiTaqqu2011}.

Throughout,
\[
 \langle x\rangle=(1+|x|^2)^{1/2}.
\]
We denote by $\mathcal C_r$ the $r$th Wiener chaos and by $d_{\rm W}$ the
$1$-Wasserstein distance.  The notation $B(a,b)$ denotes Euler's beta
function, whereas $B=(B_t)_{t\ge0}$ denotes a standard Brownian motion when
no arguments are displayed.  The Skorokhod space $D([0,T])$ is endowed with
the $J_1$ topology.

Fix $q\ge2$ and exponents
\begin{equation}
-1<\gamma_q\le\cdots\le\gamma_1<-\frac12,
\qquad
-\frac{q+1}{2}<\alpha:=\sum_{j=1}^q\gamma_j<-\frac q2.
\label{eq:admissible}
\end{equation}
Put
\begin{equation}
H_0=\alpha+\frac q2+1\in\left(\frac12,1\right),
\qquad
\beta=h-H_0\in(-1,0),
\qquad 0<h\le\frac12.
\label{eq:parameters}
\end{equation}
Let
\[
g_{\boldsymbol\gamma}(x_1,\ldots,x_q)
=\frac1{q!}\sum_{\sigma\in\mathfrak S_q}
\prod_{j=1}^q(x_j)_+^{\gamma_{\sigma(j)}}.
\]
On an isonormal Gaussian space over $L^2(\R)$ define
\begin{equation}
L_t^{(\beta)}(\mathbf x)=A_{\boldsymbol\gamma,\beta}
\int_\R\big[(t-u)_+^\beta-(-u)_+^\beta\big]
g_{\boldsymbol\gamma}(u\one-\mathbf x)\dd u,
\qquad
X_t=I_q(L_t^{(\beta)}),
\label{eq:process}
\end{equation}
where $A_{\boldsymbol\gamma,\beta}>0$ is fixed by
$\E[X_1^2]=1$.  Proposition~\ref{prop:processconstruction} below proves
that the kernel is square integrable and that the process is
$h$-self-similar with stationary increments.  In particular,
\begin{equation}
\E[(X_{t+r}-X_t)^2]=|r|^{2h}.
\label{eq:incrementvariance}
\end{equation}
At $h=1/2$ this is Brownian covariance.  The process is nevertheless not a
Brownian motion: each nonzero $X_t$ remains in the $q$th Wiener chaos with
$q\ge2$ and is therefore non-Gaussian.

For $t\ge0$ set
\begin{equation}
V_N^{(h)}(t)=\sum_{i=0}^{\lfloor Nt\rfloor-1}
\left[(X_{(i+1)/N}-X_{i/N})^2-N^{-2h}\right].
\label{eq:variation}
\end{equation}
Write
\[
\gamma_*=\gamma_1,
\qquad M=\{j:\gamma_j=\gamma_*\},
\qquad \mu=|M|,
\qquad \theta=-2\gamma_*-1.
\]
We use $M^2=M\times M$.
Since $\alpha\le q\gamma_*$,
$H_0\le1-q\theta/2$, and therefore
$0<\theta\le2(1-H_0)/q<1/q$.  The limiting second-chaos kernel is
\begin{equation}
K_t^{(\gamma_*)}(x,y)
=\int_0^t(s-x)_+^{\gamma_*}(s-y)_+^{\gamma_*}\dd s.
\label{eq:rosenblattkernel}
\end{equation}
The process $I_2(K_t^{(\gamma_*)})$ is, up to a deterministic
normalization, a Rosenblatt process with self-similarity exponent
$H_R=2\gamma_*+2=1-\theta$.

\subsection{Constants}

For $a\in(0,1)$ define the Riesz energy
\begin{equation}
\Lambda_a(\beta)=\int_{\R^2}\varphi_\beta(x)\varphi_\beta(y)
|x-y|^{-a}\dd x\dd y,
\qquad
\varphi_\beta(x)=(1-x)_+^\beta-(-x)_+^\beta.
\label{eq:energy}
\end{equation}
The integral is understood as the positive spectral energy; for the
parameters used below it is absolutely convergent after pairing the two
filters.  Fourier inversion gives
\begin{equation}
\Lambda_a(\beta)=
\frac{2\Gamma(1-a)\sin(\pi a/2)\Gamma(\beta+1)^2}
{\Gamma(3+2\beta-a)
 \sin\!\big(\pi(2+2\beta-a)/2\big)},
\qquad 2\beta<a<2(1+\beta).
\label{eq:energyexplicit}
\end{equation}
Indeed, with the convention
$\widehat f(\xi)=\int_\R e^{-i\xi x}f(x)\dd x$, one has in the sense of
tempered distributions
\[
|\widehat{\varphi_\beta}(\xi)|^2
=2\Gamma(\beta+1)^2(1-\cos\xi)|\xi|^{-2\beta-2}
\]
and
\[
\widehat{|\,\cdot\,|^{-a}}(\xi)
=2\Gamma(1-a)\sin(\pi a/2)|\xi|^{a-1}.
\]
Put $\nu=2+2\beta-a\in(0,2)$.  Parseval's identity and
\[
\int_\R(1-\cos\xi)|\xi|^{-1-\nu}\dd\xi
=\frac{\pi}{\Gamma(1+\nu)\sin(\pi\nu/2)}
\]
give \eqref{eq:energyexplicit}.  One may justify the distributional
calculation by multiplying both Fourier factors by
$e^{-\varepsilon\xi^2}\1_{\{|\xi|\le R\}}$, first sending $R$ to infinity
and then $\varepsilon$ to zero.  Lemma~\ref{lem:roughconvolution}, proved
below, supplies spatial domination.  As a direct consistency check,
$\varphi_0=\1_{[0,1)}$ and
\[
 \Lambda_a(0)=\frac{2}{(1-a)(2-a)}
 =\int_0^1\!\int_0^1|x-y|^{-a}\dd x\dd y.
\]
For the two energies used here,
\begin{align}
\Lambda_c(\beta)
&=\frac{2\Gamma(1-c)\sin(\pi c/2)\Gamma(\beta+1)^2}
{\Gamma(1+2h+\theta)\sin(\pi(2h+\theta)/2)},
\label{eq:Lambdacexplicit}\\
\Lambda_d(\beta)
&=\frac{2\Gamma(1-d)\sin(\pi d/2)\Gamma(\beta+1)^2}
{\Gamma(1+2h)\sin(\pi h)}.
\label{eq:Lambdadexplicit}
\end{align}

Put
\[
W_{jk}=B(\gamma_j+1,-\gamma_j-\gamma_k-1),
\qquad \mathcal B_0=\per(W),
\]
and let $\mathcal B_*$ be the permanent of the minor obtained by deleting a
row and a column indexed by elements of $M$.  Permuting rows and columns
carrying the common maximal exponent shows that this permanent is independent
of the selected pair of maximal labels.  Finally set
\begin{equation}
d=2(1-H_0),
\qquad c=d-\theta.
\label{eq:cd}
\end{equation}
The inequality $H_0\le1-q\theta/2$ gives $d\ge q\theta$, hence
$0<c=d-\theta<d<1$; since $h>0$,
\(c,d<2(1+\beta)\).  Hence both energies in
\eqref{eq:constant} below are finite and strictly positive.

\begin{proposition}[Construction and normalization]
\label{prop:processconstruction}
For every $t\ge0$, the kernel $L_t^{(\beta)}$ in
\eqref{eq:process} belongs to $L_s^2(\R^q)$.  With
\begin{equation}
A_{\boldsymbol\gamma,\beta}^2
=\frac1{\mathcal B_0\Lambda_d(\beta)},
\label{eq:Aexplicit}
\end{equation}
the process $X_t=I_q(L_t^{(\beta)})$ is centered, has stationary
increments, is $h$-self-similar, and satisfies
\begin{equation}
\E[(X_t-X_s)^2]=|t-s|^{2h}.
\label{eq:processcovariance}
\end{equation}
It admits a continuous modification.
\end{proposition}

\begin{proof}
For $u\ne v$, spatial contraction of the two symmetrized monomial kernels
gives
\begin{equation}
\int_{\R^q}g_{\boldsymbol\gamma}(u\one-\mathbf x)
g_{\boldsymbol\gamma}(v\one-\mathbf x)\dd\mathbf x
=\frac{\mathcal B_0}{q!}|u-v|^{-d}.
\label{eq:gfullcontractionmodel}
\end{equation}
Indeed, each relative permutation occurs exactly $q!$ times among the
$(q!)^2$ pairs of permutations.  Exchanging the two permutations exchanges
the regions $u>v$ and $v>u$ and leaves the double sum unchanged; hence both
orientations give the same permanent constant.  If
\(
\varphi_{\beta,t}(u)=(t-u)_+^\beta-(-u)_+^\beta,
\)
then scaling in the absolutely convergent energy gives
\begin{align*}
q!\|L_t^{(\beta)}\|_2^2
&=A_{\boldsymbol\gamma,\beta}^2\mathcal B_0
\int_{\R^2}\varphi_{\beta,t}(u)\varphi_{\beta,t}(v)
|u-v|^{-d}\dd u\dd v\\
&=A_{\boldsymbol\gamma,\beta}^2\mathcal B_0
\Lambda_d(\beta)t^{2+2\beta-d}
=t^{2h}.
\end{align*}
Here $0<d<1$ and $d<2(1+\beta)=d+2h$; absolute convergence follows from
Lemma~\ref{lem:roughconvolution} below (or directly by its endpoint and tail
decomposition).  This proves both square integrability and
\eqref{eq:Aexplicit}.

Translation of the temporal and spatial variables shows that
$X_{t+r}-X_r$ has the same law as $X_t$.  Moreover, homogeneity gives
\[
L_{ct}^{(\beta)}(c\mathbf x)
=c^{\alpha+\beta+1}L_t^{(\beta)}(\mathbf x).
\]
White-noise scaling contributes $c^{q/2}$ to the multiple integral, and
$\alpha+\beta+1+q/2=h$; hence $X$ is $h$-self-similar.  Formula
\eqref{eq:processcovariance} follows.  Finally, finite-chaos
hypercontractivity gives, for every $p\ge2$,
\(
\E|X_t-X_s|^p\le C_p|t-s|^{ph}.
\)
Choosing $p>1/h$ and applying Kolmogorov's criterion yields a continuous
modification.
\end{proof}

\section{Main results}

\begin{theorem}[Fixed positive roughness: strong noncentral limit]
\label{thm:main}
Assume \eqref{eq:admissible} and \eqref{eq:parameters}.  Put
\[
 Z_N^{(h)}(t)=N^{2h+\theta-1}V_N^{(h)}(t),
 \qquad
 Z(t)=C_{\boldsymbol\gamma,h}I_2(K_t^{(\gamma_*)}).
\]
For every $\ell\ge1$ and every $t_1,\ldots,t_\ell\ge0$,
\begin{equation}
 \bigl(Z_N^{(h)}(t_1),\ldots,Z_N^{(h)}(t_\ell)\bigr)
 \longrightarrow
 \bigl(Z(t_1),\ldots,Z(t_\ell)\bigr)
 \quad\text{in }L^2(\Omega;\R^\ell),
 \label{eq:mainlimit}
\end{equation}
where
\begin{equation}
 C_{\boldsymbol\gamma,h}
 =\frac{\mu^2\mathcal B_*}{\mathcal B_0}
 \frac{\Lambda_c(\beta)}{\Lambda_d(\beta)}>0.
 \label{eq:constant}
\end{equation}
The only branches of valuation zero are the second-chaos sectors obtained
by contracting $q-1$ variables and leaving two labels in $M^2$ free.

The convergence is quantitative in every fixed Malliavin--Sobolev norm.
With the explicit $\kappa>0$ and
$\ell_N\in\{1,(\log N)^{1/2}\}$ defined in
\eqref{eq:kglobal}--\eqref{eq:logfactor}, for every $T>0$, $k\ge0$, and
$p\ge2$,
\begin{equation}
 \sup_{0\le t\le T}
 \|Z_N^{(h)}(t)-Z(t)\|_{\mathbb D^{k,p}}
 \le C_{T,k,p}\ell_NN^{-\kappa}.
 \label{eq:malliavinrate}
\end{equation}
Let $\overline Z_N^{(h)}$ be the polygonal interpolation on
$N^{-1}\mathbb Z_+$, and let $c^\eta([0,T];B)$ be the closure of smooth
$B$-valued functions in the $C^\eta$ norm.  For every $p>1$,
$0<\eta<1-\theta$, and $0\le j\le k$,
\begin{equation}
 \E\!\left[
 \left\|D^j(\overline Z_N^{(h)}-Z)\right\|_
 {c^\eta([0,T];L^2(\R^j))}^{p}\right]\longrightarrow0,
 \qquad D^0Y:=Y.
 \label{eq:functional}
\end{equation}
In particular, the convergence holds in law in the separable space
$c^\eta([0,T])$.
\end{theorem}

\begin{corollary}[Entire isotropic rough-to-Brownian range]
\label{cor:isotropic}
If $\gamma_1=\cdots=\gamma_q$, then
the parameters satisfy
\begin{equation}
 H_0=1-\frac{q\theta}{2},
 \qquad H_R=1-\theta=1-\frac{2(1-H_0)}q.
 \label{eq:isotropicrelations}
\end{equation}
Theorem~\ref{thm:main} applies throughout $0<h\le1/2$.
For $q=2$ in particular,
\begin{equation}
 N^{2h-H_0}V_N^{(h)}(t)
 \longrightarrow
 \frac{2}{B(H_0/2,1-H_0)}
 \frac{\Lambda_{1-H_0}(h-H_0)}
      {\Lambda_{2(1-H_0)}(h-H_0)}
 I_2\!\left(K_t^{(H_0/2-1)}\right),
 \qquad 0<h\le\frac12.
 \label{eq:q2}
\end{equation}
The process on the right is an explicitly scaled Rosenblatt process of
self-similarity exponent $H_0$; no convention-dependent normalization is
hidden in \eqref{eq:q2}.
\end{corollary}

\begin{remark}[Consistency with the Hermite reproduction property]
At the unfiltered Hermite point $h=H_0$, the identity
$H_R=1-2(1-H_0)/q$ is exactly the reproduction exponent of the quadratic
variation of an order-$q$ Hermite process.  For $q=2$ it reduces to
$H_R=H_0$, the Rosenblatt-to-Rosenblatt reproduction property.  Although
that point lies outside the rough window considered here, the limiting
contraction has the correct classical specialization.
\end{remark}

\begin{theorem}[Universal Gaussian boundary at zero roughness]
\label{thm:zero-main}
Assume only \eqref{eq:admissible}, and put $\beta_0=-H_0$.  For every
$t>0$, the bare boundary kernel $\widetilde L_{t,0}$ in
\eqref{eq:bare-boundary-kernel} does not belong to $L^2(\R^q)$; its
squared energy diverges logarithmically.

Let $X^{(h)}$ be the variance-normalized $h$-regularized family in
\eqref{eq:analytic-family}, and let $X_\varepsilon$ be the logarithmically
normalized hard-cutoff family in \eqref{eq:cutoff-process}.  For every
finite set $F\subset[0,\infty)$,
\begin{equation}
 (X_t^{(h)})_{t\in F}
 \xrightarrow[h\downarrow0]{\rm law}(G_t)_{t\in F},
 \qquad
 (X_{\varepsilon,t})_{t\in F}
 \xrightarrow[\varepsilon\downarrow0]{\rm law}(G_t)_{t\in F},
 \label{eq:zero-main-fdd}
\end{equation}
where $G_0=0$ and
$G_t=(\zeta_t-\zeta_0)/\sqrt2$ for $t>0$, with
$(\zeta_t)_{t\ge0}$ independent standard Gaussian variables.  For every
fixed $t>0$,
\begin{equation}
 d_{\rm W}(X_t^{(h)},N(0,1))=O_t(\sqrt h),
 \qquad
 d_{\rm W}(X_{\varepsilon,t},N(0,1))
 =O_t\!\left((\log(1/\varepsilon))^{-1/2}\right).
 \label{eq:zero-main-rates}
\end{equation}
The two regularizations have the same endpoint residue, according to the
dictionary \eqref{eq:boundary-dictionary}.  The field $G$ has no
stochastically continuous modification, and neither selection converges in
probability on the original isonormal space.  Nevertheless, for every
$T>0$, the boundary-first quadratic variations in
\eqref{eq:boundary-qv} and \eqref{eq:boundary-cutoff-qv} satisfy
\begin{equation}
 \lim_{N\to\infty}\lim_{h\downarrow0}Q_{N,h}
 \Longrightarrow\sqrt3B,
 \qquad
 \lim_{N\to\infty}\lim_{\varepsilon\downarrow0}Q_{N,\varepsilon}
 \Longrightarrow\sqrt3B
 \quad\text{in }D([0,T]).
 \label{eq:zero-main-qv}
\end{equation}
\end{theorem}

\begin{theorem}[Simultaneous Gaussian--Rosenblatt transition]
\label{thm:crossover}
Assume $\gamma_1=\cdots=\gamma_q=\gamma$ and let $h_N\downarrow0$.  Put
\begin{equation}
 \mathcal Q_N(t)=N^{2h_N-1/2}V_N^{(h_N)}(t),
 \qquad
 \lambda_N=h_NN^{1/2-\theta}.
 \label{eq:phase-parameter}
\end{equation}
Then, in $D([0,T])$,
\begin{align}
 \lambda_N\to0
 &\quad\Longrightarrow\quad
 \mathcal Q_N\Longrightarrow\sqrt3B,
 \label{eq:crossover-brownian}\\
 \lambda_N\to\lambda\in(0,\infty)
 &\quad\Longrightarrow\quad
 \mathcal Q_N\Longrightarrow
 \sqrt3B+\lambda\overline C_{q,\theta}I_2(K^{(\gamma)}),
 \label{eq:crossover-mixed}\\
 \lambda_N\to\infty
 &\quad\Longrightarrow\quad
 \lambda_N^{-1}\mathcal Q_N\Longrightarrow
 \overline C_{q,\theta}I_2(K^{(\gamma)}),
 \label{eq:crossover-rosenblatt}
\end{align}
where $\overline C_{q,\theta}>0$ is explicit in
\eqref{eq:crossover-constant}.  The Brownian and Rosenblatt processes in
the mixed limit are independent.  If $(\lambda_N)$ is bounded, polygonal
interpolations converge in $C^\eta([0,T])$ for every $\eta<1/2$.
\end{theorem}

Theorem~\ref{thm:zero-main} is not an endpoint substitution in
Theorem~\ref{thm:main}: the bare kernel has left $L^2$, the limiting chaos
changes, and the path topology collapses.  Theorem~\ref{thm:crossover}
quantifies precisely the noncommutation of the roughness and mesh limits.
Together the three theorems form one quadratic-energy phase diagram.

\section{Strategy of the proofs and paper-specific technical inputs}
\label{sec:proof-strategy}

The three theorems use the same diagrammatic language, but they are not
obtained from one another by a formal passage to the limit.  At fixed
$h>0$, the problem is to retain enough integrability after taking absolute
values and then identify the unique surviving Wiener-chaos projection.  At
$h=0$, one must isolate and renormalize the unique logarithmically divergent
collision.  In the simultaneous limit, the constants in the fixed-$h$
estimates are no longer uniform, and a separate root--regular decomposition
is needed.  We give the logical chain here before entering the estimates.

\paragraph{Fixed positive roughness.}
The algebraic starting point is Lemma~\ref{lem:stationaryreduction}.  The
Wiener product formula gives the exact finite orthogonal expansion
\eqref{eq:exactchaosdecomposition} into chaoses of orders
$2,4,\ldots,2q$.  Proposition~\ref{prop:branchgrammar} then enumerates all
monomial contraction branches, and Proposition~\ref{prop:maxcontraction}
identifies the maximal contraction of order $q-1$, which is the candidate
surviving projection on the forced scale.

The analytic obstruction is already present before the branches are summed.
Below Brownian regularity, estimating a single singular filter loses too
much at its endpoints.  Quadratic variation always produces two filters,
however, and Lemmas~\ref{lem:autocorrelation} and
\ref{lem:roughconvolution} show that their absolute autocorrelation has the
sharp doubled budget
\[
 a<\min\{1,2(1+\beta)\}.
\]
This controls one paired convolution but not an arbitrary anisotropic
closure.  For the latter, each port carrying label $j$ is assigned charge
$c_j=\theta_j/2$.  Since an edge has weight equal to the sum of its endpoint
charges, every vertex set $S$ satisfies the exact cut identity
\[
 w_D(S)=\frac d2|S|-B_D(S).
\]
Lemma~\ref{lem:charged-collective-closure} uses the outgoing charge
$B_D(S)$ to make every proper local collision strictly integrable.  After
the two filter tails are represented by roots,
Lemma~\ref{lem:two-root-lattice-face} applies the same conservation law to
all macroscopic partition faces.  Thus absolute values do not erase the
essential information: signed cancellation is replaced by positive
conservation of homogeneity.  This is what removes any additional
anisotropic condition.

Proposition~\ref{prop:maxcontraction} and Lemma~\ref{lem:scaling} identify
as candidate zero-valuation branches the second-chaos sectors whose two
free labels belong to $M^2$, and they force the normalization
$N^{2h+\theta-1}$.  Proposition~\ref{prop:valuation} later proves that no
other branch has valuation zero.  Proposition~\ref{prop:active} identifies
the strong limit of the active sectors.  Its two deferred deterministic
inputs are in
Appendix~\ref{app:riemann}: Proposition~\ref{prop:displacedriemann} proves
weak convergence of the displaced Riemann sums, while
Proposition~\ref{prop:activescalarproducts} gives the exact norm and mixed
scalar-product identities.  Lemmas~\ref{lem:pairedrate} and
\ref{lem:toeplitz} turn those identities into the strong convergence proved
in Proposition~\ref{prop:active}.
For every other branch, Lemma~\ref{lem:exhaustive-closure} converts the
charge crossing the two lattice blocks into a strictly negative valuation;
Propositions~\ref{prop:residual} and \ref{prop:valuation} eliminate all
non-active second-chaos branches and all higher chaoses.  Finally,
Appendix~\ref{app:stochasticdetails} turns the resulting finite-chaos block
bounds into Malliavin--Sobolev and little-H\"older convergence; it adds no
new diagrammatic mechanism.

\paragraph{The zero-roughness boundary.}
The fixed-$h$ proof cannot be continued by substitution.  Proposition
\ref{prop:boundary-residue} computes the simple pole
\[
 \Lambda_d(\beta_0+h)\sim\frac{c_0}{h},
 \qquad
 A_{\boldsymbol\gamma,\beta_0+h}^2
 \sim\frac{h}{\mathcal B_0c_0},
\]
and the local logarithmic residue.  Proposition
\ref{prop:cutoff-covariance} computes the hard-cutoff covariance and the
common-residue dictionary.  These covariance calculations identify the only
possible limit, but they do not prove Gaussianity.  The missing input is
Appendix~\ref{app:boundary-diagrams}.  Lemma~\ref{lem:one-critical-root}
shows that every proper collision face has a uniform positive deficit,
whereas the complete collision pinned to one common endpoint has deficit
$|V|h$.  Hence a connected diagram can produce at most one factor $h^{-1}$,
or one cutoff logarithm.  After inserting the normalization above,
Lemma~\ref{lem:boundary-fourcopy} makes every proper contraction vanish.
The fixed-chaos criterion then gives the same Gaussian field for $h$-regularization
and cutoff regularization.  The final Brownian energy limit is proved in
the main text from the explicit one-dependent boundary sequence and the
martingale decomposition \eqref{eq:boundary-martingale-decomposition}.

\paragraph{The simultaneous transition.}
For $h=h_N\downarrow0$, the product formula gives the exact orthogonal
decomposition
\[
 \mathcal Q_N=\mathcal A_N+\mathcal B_N
\]
of \eqref{eq:active-bulk}, where $\mathcal A_N$ is the complete second
chaos and $\mathcal B_N$ contains the chaoses $4,\ldots,2q$.  The residue
expansion shows that the active sector has size
\[
 \lambda_N=h_NN^{1/2-\theta}.
\]
Appendix~\ref{app:root-regular} supplies the uniform engine shared by the
two sectors.  Within each connected component,
Lemma~\ref{lem:root-regular} separates the sole potentially critical
full-root sector, which carries the explicit $h^{-1}$ loss, from a regular
part that retains the full Riesz exponent of every edge across every nested
macroscopic cut.  Appendix~\ref{app:uniform-active} uses this decomposition
to prove the uniform Rosenblatt approximation of $\mathcal A_N$ in
Proposition~\ref{prop:uniform-active}.

Appendix~\ref{app:bulk} treats the complementary sector.  Proposition
\ref{prop:bulk-covariance} identifies the Brownian covariance
$3(s\wedge t)$.  Lemma~\ref{lem:four-block-grammar} organizes every proper
contraction into a four-block graph, and
Lemma~\ref{lem:dyadic-chain-face-reduction} reduces all multiscale
configurations to finitely many partition-face deficits.  The partition-face
bounds control all nonfragmented diagrams, the root-hypergraph count handles
the resonant components, and the completely fragmented packet is estimated
separately by the trace bound.  Together these arguments give
Proposition~\ref{prop:bulk-fourcopy}; the contraction Cauchy--Schwarz
estimate then also kills every active--bulk contraction.
Proposition~\ref{prop:bulk-clt-crossover} yields the Gaussian bulk and its
asymptotic independence from the Rosenblatt sector, while
Proposition~\ref{prop:crossover-blocks} supplies tightness in all three
regimes.

\section{Paired-filter integrability in the rough regime}
\label{sec:paired-filter}

The proof begins with the key estimate specific to the rough regime.  We
develop these analytic bounds first; Section~\ref{sec:chaos-decomposition}
then applies them to the exact Wiener-chaos expansion.  Let
\[
\rho_\beta(x)=|\varphi_\beta(x)|,
\qquad
q_\beta(z)=(\rho_\beta*\widetilde\rho_\beta)(z)
=\int_\R\rho_\beta(x)\rho_\beta(x-z)\dd x.
\]

\begin{lemma}[Absolute filter autocorrelation]\label{lem:autocorrelation}
For $-1<\beta<0$, $\rho_\beta\in L^1(\R)$.  Moreover, for $|z|\ge2$,
\begin{equation}
q_\beta(z)\le C_\beta |z|^{\beta-1},
\label{eq:qtail}
\end{equation}
whereas on compact sets
\begin{equation}
q_\beta(z)\le C_\beta
\left(1+\sum_{s\in\{-1,0,1\}}\Psi_\beta(|z-s|)\right),
\label{eq:qlocal}
\end{equation}
with
\[
\Psi_\beta(r)=
\begin{cases}
r^{2\beta+1},&-1<\beta<-1/2,\\
1+|\log r|,&\beta=-1/2,\\
1,&-1/2<\beta<0.
\end{cases}
\]
For every $s\in\{-1,0,1\}$ and all sufficiently small $|r|>0$, these
local estimates are two-sided:
\begin{equation}
 q_\beta(s+r)\asymp
 \begin{cases}
 |r|^{2\beta+1},&-1<\beta<-1/2,\\
 1+|\log|r||,&\beta=-1/2,\\
 1,&-1/2<\beta<0.
 \end{cases}
 \label{eq:qlocal-twosided}
\end{equation}
\end{lemma}

\begin{proof}
For $x\le0$, the fundamental theorem of calculus gives the exact
representation
\[
\rho_\beta(x)=(-x)^\beta-(1-x)^\beta
=-\beta\int_{-x}^{1-x}r^{\beta-1}\dd r.
\]
It follows that
\begin{equation}
\rho_\beta(x)\le C_\beta
\begin{cases}
|x|^\beta,&-1\le x<0,\\
|x|^{\beta-1},&x<-1,
\end{cases}
\label{eq:rhotail}
\end{equation}
while for $0<x<1$ one has $\rho_\beta(x)=(1-x)^\beta$, and for $x\ge1$
it vanishes.  Direct integration of these three bounds proves
$\rho_\beta\in L^1$.

We now quantify every possible local collision.  For $0<r<1/4$ put
\[
I_\beta(r)=\int_{|u|<2}|u|^\beta|u-r|^\beta\dd u.
\]
On $|u|\le2r$, the substitution $u=rv$ gives a contribution at most
$Cr^{2\beta+1}$.  On $2r<|u|<2$, one has $|u-r|\asymp|u|$, and hence
\begin{equation}
\int_{2r<|u|<2}|u|^{2\beta}\dd u\le C
\begin{cases}
1,&2\beta+1>0,\\
1+|\log r|,&2\beta+1=0,\\
r^{2\beta+1},&2\beta+1<0.
\end{cases}
\label{eq:endpointcollisionintegral}
\end{equation}
The singular endpoints of the two factors in
$\rho_\beta(x)\rho_\beta(x-z)$ are respectively $\{0,1\}$ and
$\{z,z+1\}$.  Their possible differences are $-1,0,1$.  A fixed partition
of unity around these four points reduces every colliding pair to
\eqref{eq:endpointcollisionintegral}; on the complement both factors are
bounded.  This proves \eqref{eq:qlocal}.

The same localization gives the matching lower bounds.  For each
$s\in\{-1,0,1\}$ select one colliding pair of one-sided endpoints.  After
translation, its contribution contains, up to fixed positive constants,
\[
 \int_0^c u^\beta(u+|r|)^\beta\dd u.
\]
For $\beta<-1/2$, restriction to $0<u<|r|$ followed by scaling gives
$c|r|^{2\beta+1}$; for $\beta=-1/2$, restriction to $|r|<u<c$ gives
$c|\log|r||$; and for $\beta>-1/2$, a fixed interval away from the
endpoint gives a positive constant.  This proves
\eqref{eq:qlocal-twosided}.

Finally take $z\ge4$ and split the convolution at $x=-z/2$.  If
$x\ge-z/2$, then $x-z\le-z/2$ and
$\rho_\beta(x-z)\le Cz^{\beta-1}$ by \eqref{eq:rhotail}.  If
$x<-z/2$, then $\rho_\beta(x)\le Cz^{\beta-1}$.  Therefore
\[
q_\beta(z)
\le Cz^{\beta-1}\left(
\int_\R\rho_\beta(x)\dd x
+\int_\R\rho_\beta(x-z)\dd x\right)
\le Cz^{\beta-1}.
\]
The convolution is even, and enlarging the constant covers $2\le|z|<4$.
This proves \eqref{eq:qtail}.
\end{proof}

\section{Collective filter closure in arbitrary order}
\label{sec:collective-closure}

The paired estimate of Section~\ref{sec:paired-filter} controls one local
Riesz convolution.  We now build the collective closure needed for the
arbitrary covariance diagrams generated by the exact decomposition in
Section~\ref{sec:chaos-decomposition}.

\begin{lemma}[Optimal rough convolution]\label{lem:roughconvolution}
Let $-1<\beta<0$ and
\begin{equation}
0<a<\min\{1,2(1+\beta)\}.
\label{eq:aconvolution}
\end{equation}
Then
\begin{equation}
\mathcal C_a(k):=\int_{\R^2}\rho_\beta(x)\rho_\beta(y)
|k+x-y|^{-a}\dd x\dd y
\le C_{a,\beta}(1+|k|)^{-a},
\qquad k\in\mathbb R.
\label{eq:roughconvolution}
\end{equation}
In addition, with $\rho_{\beta,R}=\rho_\beta\1_{(-\infty,-R)}$,
for every $R\ge2$,
\begin{equation}
\sup_{k\in\mathbb R}(1+|k|)^a
\int_{\R^2}\rho_{\beta,R}(x)\rho_\beta(y)
|k+x-y|^{-a}\dd x\dd y
\le C_{a,\beta}R^\beta.
\label{eq:uniformtail}
\end{equation}
\end{lemma}

\begin{proof}
By Fubini,
\[
\mathcal C_a(k)=\int_\R q_\beta(z)|k+z|^{-a}\dd z.
\]
For $|k|\le8$, Lemma~\ref{lem:autocorrelation} leaves only finitely many
local models.  Away from a collision, $|k+z|^{-a}$ is locally integrable
because $a<1$.  At a collision with $s\in\{-1,0,1\}$, put $r=z-s$.
The worst integrand is
$|r|^{2\beta+1-a}$ when $\beta<-1/2$, and this is integrable exactly when
$a<2(1+\beta)$.  The logarithmic and bounded cases are easier.  Thus
$\sup_{|k|\le8}\mathcal C_a(k)<\infty$.

Let $K=|k|>8$.  Split $\R$ into
\begin{align*}
A_0&=\{|z|\le K/2\},\\
A_1&=\{|k+z|\le K/2\},\\
A_2&=\R\setminus(A_0\cup A_1).
\end{align*}
On $A_0$, $|k+z|\ge K/2$, so the integral is at most
$CK^{-a}\|q_\beta\|_1$.  On $A_1$, $|z|\ge K/2$ and
\eqref{eq:qtail} gives
\begin{align*}
\int_{A_1}q_\beta(z)|k+z|^{-a}\dd z
&\le CK^{\beta-1}\int_{|w|\le K/2}|w|^{-a}\dd w\\
&\le CK^{\beta-a}\le CK^{-a}.
\end{align*}
On $A_2$, both $|z|$ and $|k+z|$ are at least $K/2$.  Scaling $z=Ku$
and using \eqref{eq:qtail} gives
\[
\int_{A_2}q_\beta(z)|k+z|^{-a}\dd z
\le CK^{\beta-a}
\int_{\{|u|,|u+\operatorname{sgn}k|\ge1/2\}}
|u|^{\beta-1}|u+\operatorname{sgn}k|^{-a}\dd u.
\]
The last integral is finite because its domain avoids both finite
singularities and its integrand is $O(|u|^{\beta-a-1})$ at infinity.
Since $\beta<0$, this is again at most $CK^{-a}$.  This proves
\eqref{eq:roughconvolution}.

For \eqref{eq:uniformtail}, first note
\[
\|\rho_{\beta,R}\|_1\le C R^\beta.
\]
Write $K=1+|k|$.  Suppose first that $K\ge R/2$.  In the region
$|k+x-y|\ge K/4$, the Riesz factor is $O(K^{-a})$ and integration of the
truncated filter gives $CR^\beta K^{-a}$.  In the resonant region put
$w=k+x-y$.  Split once more according as $|x|\ge K/4$ or
$|x|<K/4$.  If $K\le8$, then $R\le16$ and the already proved compact
bound absorbs this case, so assume $K>8$.  In the first part
$\rho_\beta(x)\le CK^{\beta-1}$, and
integration in the other filter variable gives
\[
CK^{\beta-1}\int_{|w|<K/4}|w|^{-a}\dd w
\le CK^{\beta-a}\le CR^\beta K^{-a},
\]
because $K\ge R/2$ and $\beta<0$.
In the second part the identity $y=k+x-w$ implies $|y|\ge K/3$, after
adjusting the harmless constants.  If $y>1$ the filter vanishes; otherwise
\eqref{eq:rhotail} gives $\rho_\beta(y)\le CK^{\beta-1}$.  Integrating the
truncated $x$-filter
gives the still smaller bound
\[
CR^\beta K^{\beta-1}\int_{|w|<K/4}|w|^{-a}\dd w
\le CR^\beta K^{\beta-a}\le CR^\beta K^{-a}.
\]

Now suppose $K<R/2$.  Split again according as
$|k+x-y|\ge|x|/4$ or not.  The first part is bounded by
\begin{align*}
C\int_{x<-R}\rho_\beta(x)|x|^{-a}\dd x
\int_\R\rho_\beta(y)\dd y
&\le CR^{\beta-a}\\
&\le CR^\beta K^{-a}.
\end{align*}
In the remaining part set $w=k+x-y$.  Since $|w|<|x|/4$ and
$|k|<R/2<|x|/2$, one has $|y|=|k+x-w|\asymp|x|$, so both filters are in
their tails.  Therefore
\begin{align*}
&\int_{x<-R}\int_{|w|<|x|/4}
\rho_\beta(x)\rho_\beta(k+x-w)|w|^{-a}\dd w\dd x\\
&\quad\le C\int_R^\infty r^{2\beta-2}
\left(\int_0^{r/4}w^{-a}\dd w\right)\dd r
\le CR^{2\beta-a}.
\end{align*}
Because $K<R/2$ and $\beta<0$,
$R^{2\beta-a}\le CR^{\beta-a}\le CR^\beta K^{-a}$.
The two cases prove \eqref{eq:uniformtail}.
\end{proof}

\begin{lemma}[Rough H\"older--Finner closure]\label{lem:roughfinner}
Let $\mathcal G=(V,E)$ be a finite multigraph.  A vertex $v$ carries the
measure $\rho_\beta(x_v)\dd x_v$, and an edge $e=\{v,w\}$ carries a
translation $k_e\in\R$ and a Riesz exponent $\lambda_e>0$.  Suppose that
there is a number
\begin{equation}
D<\min\{1,2(1+\beta)\}
\quad\text{such that}\quad
\sum_{e\ni v}\lambda_e\le D\quad(v\in V).
\label{eq:finnerdegree}
\end{equation}
Then
\begin{align}
&\int_{\R^V}\prod_{v\in V}\rho_\beta(x_v)
\prod_{e=\{v,w\}\in E}|k_e+x_v-x_w|^{-\lambda_e}
\prod_{v\in V}\dd x_v
\notag\\
&\hspace{35mm}\le
C_{\mathcal G,D}\prod_{e\in E}(1+|k_e|)^{-\lambda_e}.
\label{eq:roughfinner}
\end{align}
If one prescribed vertex $v_0$ is restricted to $x_{v_0}<-R$, the
right-hand side is multiplied by $CR^\beta$, uniformly for $R\ge2$.
\end{lemma}

\begin{proof}
Put $M_\beta=\int_\R\rho_\beta(x)\dd x$ and normalize
$\rho_\beta(x)\dd x$ to the probability measure
$\mu_\beta(\dd x)=M_\beta^{-1}\rho_\beta(x)\dd x$.  Set
$r_e=D/\lambda_e$.  Then
\begin{equation}
 \sum_{e\ni v}\frac1{r_e}
 =\frac1D\sum_{e\ni v}\lambda_e\le1
 \qquad(v\in V).
 \label{eq:finner-cover}
\end{equation}
For completeness, the version of H\"older--Finner used here
\cite{Finner1992} says that,
on a product probability space, edge functions $G_e$ satisfying
\eqref{eq:finner-cover} obey
\begin{equation}
 \int\prod_{e\in E}|G_e(x_v,x_w)|\prod_{v\in V}\mu_\beta(\dd x_v)
 \le
 \prod_{e\in E}
 \left(\int|G_e(x,y)|^{r_e}
 \mu_\beta(\dd x)\mu_\beta(\dd y)\right)^{1/r_e}.
 \label{eq:finner-explicit}
    \end{equation}
Indeed, add at each vertex a unary factor equal to $1$ with reciprocal
exponent $1-\sum_{e\ni v}r_e^{-1}$, and apply H\"older successively in the
coordinates $x_v$.  This proves \eqref{eq:finner-explicit} first for simple
functions; monotone convergence gives the nonnegative measurable case.

Apply \eqref{eq:finner-explicit} with
$G_e(x,y)=|k_e+x-y|^{-\lambda_e}$.  Restoring the harmless powers of
$M_\beta$, the left-hand side of \eqref{eq:roughfinner} is bounded by
\[
C\prod_{e\in E}
\left[
\int_{\R^2}\rho_\beta(x)\rho_\beta(y)
|k_e+x-y|^{-D}\dd x\dd y
\right]^{1/r_e}.
\]
Lemma~\ref{lem:roughconvolution} turns the factor associated with $e$ into
$(1+|k_e|)^{-D/r_e}=(1+|k_e|)^{-\lambda_e}$.

We now prove the marked assertion without using the same indicator twice.
Write
\[
 p_e=\frac1{r_e}=\frac{\lambda_e}{D},
 \qquad
 s_0=\sum_{e\ni v_0}p_e,
 \qquad
 p_0=1-s_0\ge0,
\]
and put $A_R=(-\infty,-R)$.  Since indicators are idempotent, one has the
exact fractional factorization
\begin{equation}
 \1_{A_R}(x_{v_0})
 =\1_{A_R}(x_{v_0})^{p_0}
  \prod_{e\ni v_0}\1_{A_R}(x_{v_0})^{p_e}.
 \label{eq:marked-indicator-factorization}
\end{equation}
Here and below the first factor is omitted when $p_0=0$; this avoids any
convention involving the value of $0^0$.
Represent each Riesz factor as
\[
 |k_e+x_v-x_w|^{-\lambda_e}
 =F_e(x_v,x_w)^{p_e},
 \qquad
 F_e(x_v,x_w)=|k_e+x_v-x_w|^{-D}.
\]
For $e\ni v_0$ replace $F_e$ by
$F_{e,R}=\1_{A_R}(x_{v_0})F_e$.  When $p_0>0$, introduce the unary
hyperedge $F_{0,R}=\1_{A_R}(x_{v_0})$ with weight $p_0$; when $p_0=0$,
omit it.  By
\eqref{eq:marked-indicator-factorization}, the product of these weighted
hyperfunctions is exactly the original marked integrand.  At $v_0$ their
total fractional weight is $p_0+\sum_{e\ni v_0}p_e=1$, and at every other
vertex it is at most one.  Applying \eqref{eq:finner-explicit}, including
the unary factor, gives
\begin{align*}
 I_R
 &\le C
 \left(\int_{A_R}\rho_\beta(x)\dd x\right)^{p_0}
 \prod_{e\ni v_0}
 \left[
  \int_{\R^2}\1_{A_R}(x)\rho_\beta(x)\rho_\beta(y)
  |k_e+x-y|^{-D}\dd x\dd y
 \right]^{p_e}\\
 &\qquad\times
 \prod_{e\not\ni v_0}
 \left[
  \int_{\R^2}\rho_\beta(x)\rho_\beta(y)
  |k_e+x-y|^{-D}\dd x\dd y
 \right]^{p_e}.
\end{align*}
If $v_0$ is written as the second endpoint, exchange $x,y$ and replace
$k_e$ by $-k_e$.  The elementary tail estimate
$\int_{A_R}\rho_\beta\le CR^\beta$, together with
\eqref{eq:uniformtail} and \eqref{eq:roughconvolution}, now yields
\[
 I_R\le
 C(R^\beta)^{p_0+\sum_{e\ni v_0}p_e}
 \prod_{e\in E}(1+|k_e|)^{-Dp_e}
 =CR^\beta\prod_{e\in E}(1+|k_e|)^{-\lambda_e}.
\]
This also covers an isolated marked vertex, for which $p_0=1$.
\end{proof}

\begin{lemma}[Two-root lattice face bound]
\label{lem:two-root-lattice-face}
Let $Q$ be a finite weighted multigraph on
$V\cup\{\partial_0,\partial_1\}$, with positive edge weights, and fix the
roots at $0$ and $k\in\mathbb Z$.  The remaining vertices range over
$\mathbb Z$.  Suppose that every $v\in V$ is joined to one of the roots by
an edge of weight strictly larger than $1$.  For a partition $\mathcal P$
separating the roots, let $w_Q(\mathcal P)$ be the total weight of its cut
edges, let $f(\mathcal P)$ be the number of rootless cells, and put
\[
 \Phi_Q(\mathcal P)=w_Q(\mathcal P)-f(\mathcal P).
\]
Consider
\begin{equation}
 \Sigma_Q(k)=
 \sum_{\mathbf n\in\mathbb Z^V}
 \prod_{e=\{v,w\}\in E(Q)}
 \langle n_v-n_w\rangle^{-\lambda_e},
 \qquad n_{\partial_0}=0,\quad n_{\partial_1}=k.
 \label{eq:two-root-lattice-sum}
\end{equation}
If $\Phi_Q(\mathcal P)>a$ for every root-separating partition, then
$\Sigma_Q(k)\le C_a\langle k\rangle^{-a}$.  The same conclusion holds if
$\Phi_Q(\mathcal P)\ge a$ for every such partition and equality occurs for
exactly one partition.
\end{lemma}

\begin{proof}
For bounded $k$, discard all non-root edge factors in
\eqref{eq:two-root-lattice-sum}.  The root-edge weights are larger than
$1$, so the resulting product of one-dimensional sums is finite.

Let now $K=|k|\ge2$ and $J=\lceil\log_2K\rceil$.  First restrict every
variable vertex to a finite interval.  For each dyadic level $r\ge0$, let
$\mathcal P_r$ be the partition into connected components of the graph on
$\{0,k\}\cup\{n_v:v\in V\}$ obtained by joining two points whose distance
is at most $2^r$.  The partitions are nested, and after resolving the
finitely many boundary equalities they give a disjoint family of sectors
with uniformly bounded multiplicity.  Below the first level at which the
two roots merge, every plateau is a root-separating partition
$\mathcal P$.

Fix one such plateau.  In each rootless cell choose a distinguished vertex
as its center and encode the remaining coordinates along a spanning forest.
The internal relative coordinates were already fixed at lower merger
levels, while the center has one translation parameter at the present
level.  Thus every rootless cell contributes exactly one scale dimension.
Every edge cut by $\mathcal P$ has length comparable to the plateau scale
and contributes its full weight.  A plateau of length $r$ consequently
contributes
\begin{equation}
 2^{-r\{w_Q(\mathcal P)-f(\mathcal P)\}}
 =2^{-r\Phi_Q(\mathcal P)}.
 \label{eq:two-root-plateau-factor}
\end{equation}
All comparison constants are uniform over the bounded offsets inside the
dyadic cells.  This is the weighted two-root version of the chain
calculation in \eqref{eq:chain-deficit-convex}.

There are finitely many partition chains.  For a fixed chain, factor
$2^{-aJ}$.  The remaining sum over its plateau lengths is geometric with
ratio $2^{-(\Phi_Q(\mathcal P)-a)}$ at every strict face.  If every face is
strict, all these sums are uniform.  If equality occurs only at
$\mathcal P_*$, monotonicity of the cluster tree allows that partition at
most once.  After the lengths of all strict plateaux have been chosen, its
length is fixed by the total depth $J+O(1)$; it therefore creates no
additional sum and no factor $J$.  Thus all sectors up
to the root-merger scale contribute $O(2^{-aJ})=O(K^{-a})$.

It remains to remove the finite-interval restriction.  Above the
root-merger scale, let $\mathcal Q$ be a partition with one cell containing
both roots and all other cells rootless.  If $\tau_v>1$ is the weight of
the prescribed root edge at $v$, its plateau reserve satisfies
\[
 w_Q(\mathcal Q)-f(\mathcal Q)
 \ge\sum_{C\ \mathrm{rootless}}
       \left(\sum_{v\in C}\tau_v-1\right)
 \ge(\min_v\tau_v-1)
       \sum_{C\ \mathrm{rootless}}|C|>0
\]
unless there is no outer cell.  Hence every outer plateau is summed by a geometric series, uniformly in
the truncation interval.  Letting that interval increase to $\mathbb Z$
proves the result.
\end{proof}

\begin{lemma}[Affine-form power counting]
\label{lem:affine-power-counting}
Let $Q\subset\R^n$ be a bounded cube, and let
$L_1,\ldots,L_m$ be nonconstant affine forms with linear parts
$\ell_1,\ldots,\ell_m$ and weights $\lambda_1,\ldots,\lambda_m>0$.  For
$\varnothing\ne J\subset\{1,\ldots,m\}$ put
\[
 r(J)=\operatorname{rank}\{\ell_j:j\in J\}.
\]
Assume that
\begin{equation}
 \sum_{j\in J}\lambda_j<r(J)
 \label{eq:affine-power-counting-condition}
\end{equation}
whenever the affine equations $L_j=0$, $j\in J$, have a common solution.
Then
\begin{equation}
 \int_Q\prod_{j=1}^m|L_j(x)|^{-\lambda_j}\dd x<\infty.
 \label{eq:affine-power-counting-integral}
\end{equation}
If the affine arrangements range over a finite family and
\eqref{eq:affine-power-counting-condition} holds for every member, the
integrals in \eqref{eq:affine-power-counting-integral} are uniformly
bounded over that family.
\end{lemma}

\begin{proof}
It is enough to work on the region where a prescribed subset of the forms
has modulus at most one; the other factors are bounded, and there are only
finitely many such subsets.  For the active forms introduce dyadic levels
$m_j\ge0$ by
\[
 2^{-m_j-1}<|L_j(x)|\le2^{-m_j},
\]
and put $G_k=\{j:m_j\ge k\}$.  If a subfamily $J$ has no common zero, then
compactness of $Q$ gives
\[
 \inf_{x\in Q}\max_{j\in J}|L_j(x)|>0.
\]
Consequently there is $k_0<\infty$ such that, on every nonempty dyadic cell,
$G_k$ has a common zero whenever $k\ge k_0$.

Order the forms by decreasing $m_j$ and choose greedily a basis of their
linear parts.  The matroid greedy identity gives
\[
 \sum_{j\in\mathcal B}m_j=\sum_{k\ge1}r(G_k).
\]
Using the independent affine forms as coordinates and completing them by
fixed coordinates on their common kernel shows that the volume of the
corresponding dyadic cell is at most
\[
 C_Q2^{-\sum_{k\ge1}r(G_k)}.
\]
The singular weights on that cell are at most
$2^{\sum_j\lambda_jm_j}
 =2^{\sum_{k\ge1}\sum_{j\in G_k}\lambda_j}$.  Its integral is therefore
bounded by
\begin{equation}
 C_Q2^{-\sum_{k\ge1}
 \left(r(G_k)-\sum_{j\in G_k}\lambda_j\right)}.
 \label{eq:affine-dyadic-defect}
\end{equation}
There are finitely many relevant subfamilies, so
\eqref{eq:affine-power-counting-condition} has a positive minimum margin
$\delta$ over those having a common zero.  If
$M=\max_jm_j$, then \eqref{eq:affine-dyadic-defect} is at most
$C2^{-\delta(M-k_0)_+}$.  The number of level vectors with maximum $M$ is
$O((M+1)^m)$, and summation over $M$ is finite.  This proves
\eqref{eq:affine-power-counting-integral}.  For a finite family of affine
arrangements, take the minimum of the finitely many positive margins and
the maximum of the finitely many coordinate constants.
\end{proof}

\begin{lemma}[Charged collective closure]
\label{lem:charged-collective-closure}
Put
\[
 \theta_j=-2\gamma_j-1,\qquad
 c_j=\frac{\theta_j}{2},\qquad
 d=\sum_{j=1}^q\theta_j,\qquad
 \tau=1-\beta>1.
\]
Let $D=(V,E)$ be a finite labelled closure diagram.  At every $v\in V$
there is exactly one port $(v,j)$ of charge $c_j$ for each
$j\in\{1,\ldots,q\}$; every port is paired exactly once with a port at a
different vertex, and an edge joining ports $a,b$ has weight
$\lambda_e=c_a+c_b$.  Fix a block map $b:V\to\{0,1\}$.  For $S\subset V$, let $w_D(S)$ be the
total weight of the edges internal to $S$, and let $B_D(S)$ be the total
charge of the ports based in $S$ whose partners lie in $S^c$.  Then
\begin{equation}
 w_D(S)=\frac d2|S|-B_D(S).
 \label{eq:fixed-h-charged-cut}
\end{equation}

Let $F\subset E$ be connected on the vertex set $S$, and put
\[
 R_D(F,S)=w_D(S)-\sum_{e\in F}\lambda_e\ge0.
\]
If no filter endpoint is selected, the affine deficit is
\begin{equation}
 \delta_{\rm free}(F)
 =|S|\left(1-\frac d2\right)-1+B_D(S)+R_D(F,S)>0.
 \label{eq:fixed-h-free-deficit}
\end{equation}
If $P\subset S$ is a nonempty family of $p=|P|$ compatible endpoint
marks, then
\begin{equation}
 \delta_{\rm pin}(F,P)
 =(|S|-p)\left(1-\frac d2\right)
   +ph+B_D(S)+R_D(F,S)>0.
 \label{eq:fixed-h-pinned-deficit}
\end{equation}
Consequently every compact free or endpoint-pinned collision is
absolutely integrable.

Set
\[
 p_D=\sum_{\substack{e=\{v,w\}\in E\\b(v)\ne b(w)}}\lambda_e
\]
and, for $p\ge0$, define
\begin{equation}
 \psi_p(k)=
 \begin{cases}
  1,&p=0,\\
  \langle k\rangle^{-p},&0<p\le1,\\[1mm]
  \langle k\rangle^{-1-\varepsilon_p},&p>1,
 \end{cases}
 \qquad
 \varepsilon_p=\frac12\bigl(\min\{p,\tau\}-1\bigr)>0
 \quad(p>1).
 \label{eq:collective-weight}
\end{equation}
For every fixed $k$, the map $p\mapsto\psi_p(k)$ is nonincreasing.
Then
\begin{align}
 \mathcal I_D(k)
 &:={}
 \int_{\R^V}\prod_{v\in V}\rho_\beta(x_v)
 \prod_{e=\{v,w\}\in E}
 |k(b(v)-b(w))+x_v-x_w|^{-\lambda_e}
 \dd\mathbf x
 \notag\\
 &\le C_D\psi_{p_D}(k),
 \qquad k\in\mathbb Z.
 \label{eq:charged-collective-decay}
\end{align}
Moreover, for every prescribed $v_0\in V$,
\begin{equation}
 \lim_{R\to\infty}\sup_{k\in\mathbb Z}
 \psi_{p_D}(k)^{-1}
 \int_{\R^V}\1_{\{x_{v_0}<-R\}}
 \prod_{v\in V}\rho_\beta(x_v)
 \prod_{e\in E}
 |k(b(v)-b(w))+x_v-x_w|^{-\lambda_e}
 \dd\mathbf x=0.
 \label{eq:charged-collective-marked-tail}
\end{equation}
\end{lemma}

\begin{proof}
Identity \eqref{eq:fixed-h-charged-cut} follows by counting the charges
based in $S$: an internal edge consumes both endpoint charges and the
unconsumed charge is $B_D(S)$.  For a connected selected edge family on
$S$, the incidence forms have rank $|S|-1$.  Adding at least one compatible
endpoint form raises the rank to $|S|$.  Thus the two deficits are
\[
 |S|-1-\sum_{e\in F}\lambda_e,
 \qquad
 |S|-\sum_{e\in F}\lambda_e-p(-\beta).
\]
Substitution of \eqref{eq:fixed-h-charged-cut} and
$-\beta=1-d/2-h$ gives
\eqref{eq:fixed-h-free-deficit} and
\eqref{eq:fixed-h-pinned-deficit}.  In the free case $|S|\ge2$, so the
first two terms are at least $1-d>0$.  In the pinned case $p\le|S|$ and
$h>0$.  Lemma~\ref{lem:affine-power-counting}, applied to the
connected components of every selected subfamily, proves compact local
integrability.

We turn to the collective tail.  The endpoint bounds
\eqref{eq:rhotail} give, with fixed cutoffs around $0$ and $1$,
\begin{equation}
 \rho_\beta(x)\le C\bigl[
 \chi_0(x)|x|^\beta+
 \chi_1(x)|1-x|^\beta+
 \langle x\rangle^{-\tau}\bigr].
 \label{eq:filter-root-tail-split}
\end{equation}
Put $y_v=b(v)k+x_v$ and use a bounded-overlap smooth partition subordinate
to the unit cells centered at $n_v\in\mathbb Z$.  Call an edge near when
$|n_v-n_w|\le6$.  Far edges factor out with the corresponding discrete
weight, while a selected tail term at $v$ factors out as
$C\langle n_v-b(v)k\rangle^{-\tau}$.  A selected endpoint term restricts
$n_v-b(v)k$ to one of finitely many offsets, on which that same discrete
weight is comparable to $1$.

The remaining cell integral contains only near Riesz forms and endpoint
forms.  For every subfamily having a common zero, decompose its near-edge
graph into connected components and apply
\eqref{eq:fixed-h-free-deficit} or
\eqref{eq:fixed-h-pinned-deficit}; a subfamily of incompatible endpoint
forms has no common singularity.  There are only finitely many near-edge
graphs, endpoint choices, and relative cell offsets, so
Lemma~\ref{lem:affine-power-counting} gives one uniform cell constant.  Summing these finite choices
gives
\begin{equation}
 \mathcal I_D(k)
 \le C_D\sum_{\mathbf n\in\mathbb Z^V}
 \prod_{v\in V}\langle n_v-b(v)k\rangle^{-\tau}
 \prod_{e=\{v,w\}\in E}
 \langle n_v-n_w\rangle^{-\lambda_e}.
 \label{eq:charged-cell-sum}
\end{equation}

Introduce roots $\partial_0,\partial_1$ at $0,k$ and attach every $v$ to
its native root $\partial_{b(v)}$ by an edge of weight $\tau$.  For a
partition $\mathcal P$ of
$V\cup\{\partial_0,\partial_1\}$ separating the roots, let
\[
 r_D(\mathcal P)
 =\#\{v:[v]_{\mathcal P}\ne[\partial_{b(v)}]_{\mathcal P}\},
 \qquad
 f(\mathcal P)
 =\#\{C\in\mathcal P:C\cap\{\partial_0,\partial_1\}=\varnothing\},
\]
and set
\begin{equation}
 \Phi_D(\mathcal P)
 =\sum_{\substack{e=\{v,w\}\in E\\
          [v]_{\mathcal P}\ne[w]_{\mathcal P}}}\lambda_e
   +\tau r_D(\mathcal P)-f(\mathcal P).
 \label{eq:charged-macro-face}
\end{equation}
For the augmented graph, \eqref{eq:charged-macro-face} is exactly the face
$\Phi_Q$ of Lemma~\ref{lem:two-root-lattice-face}: the crossed root edges
give $\tau r_D(\mathcal P)$ and every rootless cell contributes the
translation dimension subtracted in $f(\mathcal P)$.  It remains only to
bound these finitely many faces.

The native partition
\[
 \mathcal P_0=
 \bigl\{\{\partial_0\}\cup b^{-1}(0),
          \{\partial_1\}\cup b^{-1}(1)\bigr\}
\]
has $\Phi_D(\mathcal P_0)=p_D$.  Consider another partition.  If no native
cross-block edge becomes internal, its original-edge contribution is at
least $p_D$, while $\tau r_D-f\ge0$: a rootless cell containing $s$
vertices contributes $s\tau-1>0$, and a rooted cell contributes $\tau$
for each vertex separated from its native root.  Equality forces no
rootless cell and no vertex separated from its native root, hence the
native partition.  If cross-block edges of
total weight $L>0$ become internal, some cell is mixed.  A mixed rooted
cell costs at least $\tau$, whereas a mixed rootless cell costs at least
$2\tau-1>\tau$.  Hence
\begin{equation}
 \Phi_D(\mathcal P)\ge p_D-L+\tau
 \ge\min\{p_D,\tau\}.
 \label{eq:charged-face-lower-bound}
\end{equation}
Thus the native face is the unique minimizer when $p_D\le1<\tau$, and the
second conclusion of Lemma~\ref{lem:two-root-lattice-face} gives the exact
power $\langle k\rangle^{-p_D}$ without a logarithmic loss.  When
$p_D>1$, every face is strictly above $1$; since
$1+\varepsilon_p<\min\{p_D,\tau\}$, the first conclusion of that lemma
proves \eqref{eq:charged-collective-decay}.

For the marked assertion, only the tail term in
\eqref{eq:filter-root-tail-split} remains once $R$ is large.  Let $g_D>0$
be the minimum, over the finite family of relevant faces, of
$\Phi_D-p_D$ away from the native face when $p_D\le1$, and of
$\Phi_D-(1+\varepsilon_p)$ when $p_D>1$.  Choose
\[
 0<\delta<\min\{\tau-1,g_D\},
 \qquad \tau'=\tau-\delta.
\]
Lowering the root-edge weight at $v_0$ to $\tau'$ changes any face by at
most $\delta$.  The native face cuts no root edge and is unchanged, while
all required strict gaps and the outer-plateau hypothesis of
Lemma~\ref{lem:two-root-lattice-face} persist.
On $x_{v_0}<-R$,
\[
 \langle x_{v_0}\rangle^{-\tau}
 \le R^{-\delta}\langle x_{v_0}\rangle^{-\tau'}.
\]
Applying that lemma to the perturbed root weight and repeating the cell sum
gives the stronger bound $C_DR^{-\delta}\psi_{p_D}(k)$ and proves
\eqref{eq:charged-collective-marked-tail}.
\end{proof}

The paired estimate also has the quantitative form needed for strong kernel
rates.  We state it separately because this is where a one-variable argument
would be invalid in the rough regime.  Define
\begin{align}
\Gamma_*(k)={}&\int_{\R^4}
\varphi_\beta(a)\varphi_\beta(b)
\varphi_\beta(a')\varphi_\beta(b')
|a-b|^{-c}|a'-b'|^{-c}
\notag\\
&\hspace{22mm}\times
|k+a-a'|^{-\theta}|k+b-b'|^{-\theta}
\dd a\dd b\dd a'\dd b',
\label{eq:Gammastar}\\
J_*(z)={}&\int_{\R^2}\varphi_\beta(a)\varphi_\beta(b)
|a-b|^{-c}|z-a|^{-\theta}|z-b|^{-\theta}
\dd a\dd b.
\label{eq:Jstar}
\end{align}
Exchanging primed and unprimed variables gives
$\Gamma_*(-k)=\Gamma_*(k)$; this parity is used in the Toeplitz reductions
below.

\begin{lemma}[Marked tail for the anchored three-edge integral]
\label{lem:anchored-tail}
Let $-1<\beta<0$ and $0<c,\theta<1$ with $c+2\theta<2$.  Let
$\rho$ be nonnegative, supported in $(-\infty,1]$, integrable, and such
that $\rho(x)\le C|x|^{\beta-1}$ for $x\le-1$.  Then, for every $R\ge2$
and $|z|\ge2$,
\begin{align}
 &\int_{a<-R}\!\int_\R
 \rho(a)\rho(b)|a-b|^{-c}|z-a|^{-\theta}|z-b|^{-\theta}
 \dd b\dd a
 \notag\\
 &\hspace{35mm}\le
 C R^\beta(1+|z|)^{-2\theta}.
 \label{eq:anchored-tail}
\end{align}
\end{lemma}

\begin{proof}
Put $P_c(x)=\int_\R\rho(y)|x-y|^{-c}\dd y$.  If $x=-K\le-2$, split
the $y$-line into
\[
 \{|y|\le K/2\},\qquad
 \{|x-y|\le K/2\},\qquad
 \{|y|>K/2,\ |x-y|>K/2\}.
\]
On the first set the contribution is $O(K^{-c})$.  On the second,
$y\le-K/2$ and hence it is at most
\[
 CK^{\beta-1}\int_{|w|\le K/2}|w|^{-c}\dd w
 \le CK^{\beta-c}.
\]
On the third set, the substitution $y=Ku$ and the tail bound for $\rho$
give $CK^{\beta-c}$ times an integrable function of $u$; the domain stays
away from both finite singularities, and the power at infinity is
$\beta-c-1<-1$.  Since $\beta<0$,
\begin{equation}
 P_c(x)\le C|x|^{-c},\qquad x\le-2.
 \label{eq:Pc-tail}
\end{equation}
Consequently
\begin{equation}
 \int_{a<-R}\rho(a)P_c(a)\dd a
 \le C\int_R^\infty r^{\beta-c-1}\dd r
 \le CR^\beta.
 \label{eq:Pc-marked}
\end{equation}

If $z\ge2$, the support condition gives
$|z-a|\gtrsim1+z$ and $|z-b|\gtrsim1+z$ for $a,b\le1$.  Thus both anchored
factors are $O((1+z)^{-\theta})$, and
\eqref{eq:Pc-marked} proves \eqref{eq:anchored-tail}.

Let now $z=-K\le-2$.  Suppose first that $K<R/2$.  Partition the $b$-line
into
\[
 E_1=\{|b-a|<|a|/4\},\qquad
 E_2=\{|b-z|<K/2\},\qquad
 E_3=\R\setminus(E_1\cup E_2),
\]
assigning overlaps to the first set in this order.  On $E_1$, both filter
variables are in the same tail and both anchored distances are comparable
with $|a|$; hence the contribution is bounded by
\[
 C\int_R^\infty r^{2\beta-2-2\theta}
 \left(\int_0^{r/4}w^{-c}\dd w\right)\dd r
 \le CR^{2\beta-c-2\theta}.
\]
On $E_2\setminus E_1$, one has $|a-b|\asymp|a|$ and
$|z-a|\asymp|a|$, while
\[
 \int_{|b-z|<K/2}\rho(b)|z-b|^{-\theta}\dd b
 \le CK^{\beta-\theta}\le CK^{-\theta}.
\]
This gives $CR^{\beta-c-\theta}K^{-\theta}$.  On $E_3$, use
$|a-b|\gtrsim|a|$, $|z-b|\gtrsim K$, and $\rho\in L^1$ to obtain the
same bound.  After multiplication by $K^{2\theta}$, all three terms are
$O(R^\beta)$ because $K<R/2$, $\beta<0$, and $c>0$.

Assume finally that $K\ge R/2$, and put
\[
 A=\{|a-z|\le K/4\},\qquad B=\{|b-z|\le K/4\}.
\]
On $A^c\cap B^c$, both anchored factors are $O(K^{-\theta})$, so
\eqref{eq:Pc-marked} gives $CR^\beta K^{-2\theta}$.  On
$A\cap B^c$, use $|z-b|^{-\theta}\le CK^{-\theta}$,
\eqref{eq:Pc-tail}, and $\rho(a)\le CK^{\beta-1}$ to get
\[
 CK^{-\theta}K^{-c}K^{\beta-1}
 \int_{|u|\le K/4}|u|^{-\theta}\dd u
 \le CK^{\beta-c-2\theta}.
 \label{eq:one-resonant-a}
\]
For the nonsymmetric sector $A^c\cap B$, note first that, uniformly for
$|b-z|\le K/4$,
\begin{equation}
 \int_{a<-R}\rho(a)|a-b|^{-c}\dd a
 \le CR^\beta K^{-c}.
 \label{eq:truncated-Pc}
\end{equation}
Indeed, on $|a-b|\le K/2$ use
$\rho(a)\le CK^{\beta-1}$ and integrate $|a-b|^{-c}$; on the complement
use $|a-b|^{-c}\le CK^{-c}$ and
$\int_{a<-R}\rho(a)\dd a\le CR^\beta$.  Since
$|z-a|^{-\theta}\le CK^{-\theta}$ on $A^c$, equations
\eqref{eq:truncated-Pc} and
$\int_B\rho(b)|z-b|^{-\theta}\dd b\le CK^{\beta-\theta}$ give
\[
 CR^\beta K^{\beta-c-2\theta}
 \le CR^\beta K^{-2\theta}.
\]
Finally, on $A\cap B$, set $a=z+Ku$, $b=z+Kv$.  The contribution is at
most
\[
 CK^{2\beta-c-2\theta}
 \int_{|u|,|v|\le1/4}
 |u|^{-\theta}|v|^{-\theta}|u-v|^{-c}\dd u\dd v.
\]
The last integral is finite: pair collisions require $c<1$, the two pinned
collisions require $\theta<1$, and the full collision requires
$c+2\theta<2$.  Because $K\ge R/2$ and $\beta<0$, both
$K^{\beta-c}$ and $K^{2\beta-c}$ are $O(R^\beta)$.  Combining the four
sectors proves \eqref{eq:anchored-tail}.
\end{proof}

\begin{lemma}[Quantitative paired correlations]\label{lem:pairedrate}
Under the assumptions of Theorem~\ref{thm:main}, with
$\delta_\beta=-\beta/(1-\beta)$, one has, for $|k|\ge1$,
\begin{equation}
\left|\Gamma_*(k)-\Lambda_c(\beta)^2|k|^{-2\theta}\right|
\le C|k|^{-2\theta-\delta_\beta}.
\label{eq:Gammarate}
\end{equation}
The function $J_*$ is locally integrable and, for $|z|\ge2$,
\begin{equation}
\left|J_*(z)-\Lambda_c(\beta)|z|^{-2\theta}\right|
\le C|z|^{-2\theta-\delta_\beta}.
\label{eq:Jrate}
\end{equation}
\end{lemma}

\begin{proof}
Truncate every filter variable to $[-R,1]$.  In \eqref{eq:Gammastar}, each
filter vertex has incident degree $c+\theta=d$; the same is true of the two
filter vertices in \eqref{eq:Jstar}.  Since
\[
d<\min\{1,2(1+\beta)\},
\]
Lemma~\ref{lem:roughfinner}, with $D=d$, applies at every saturated vertex
of the rectangle \eqref{eq:Gammastar}.  Its marked-tail form gives
\begin{align}
\sup_{k\in\mathbb R}(1+|k|)^{2\theta}
|\Gamma_*(k)-\Gamma_{*,R}(k)|&\le CR^\beta,
\label{eq:Gammatail}\\
\sup_{|z|\ge2}(1+|z|)^{2\theta}
|J_*(z)-J_{*,R}(z)|&\le CR^\beta.
\label{eq:Jtail}
\end{align}
For \eqref{eq:Jtail}, the two long edges end at an anchored point rather
than another filter vertex.  Since $c=d-\theta$, one has
$c+2\theta=d+\theta<2$.  Lemma~\ref{lem:anchored-tail}, applied with
$\rho=\rho_\beta$, bounds the term in which $a<-R$ by
$CR^\beta(1+|z|)^{-2\theta}$; exchanging $a$ and $b$ gives the other
term.  The union bound for the two marked tails proves \eqref{eq:Jtail}.
Notice that this argument never invokes the generally false one-filter
rough estimate
$\int\rho_\beta(x)|z-x|^{-d}\dd x\lesssim(1+|z|)^{-d}$.

On the truncated domain, if $|k|,|z|\ge4(R+1)$, the mean value theorem gives
a relative error $O(R/|k|)$ and $O(R/|z|)$ when the long edges are replaced
by $|k|^{-\theta}$ and $|z|^{-\theta}$.  The local masses converge to
$\Lambda_c(\beta)$ with error $O(R^\beta)$, again by the marked-tail
estimate.  Therefore
\[
\begin{split}
|k|^{2\theta}
\left|\Gamma_*(k)-\Lambda_c(\beta)^2|k|^{-2\theta}\right|
&\le C\left(R^\beta+\frac R{|k|}\right),\\
|z|^{2\theta}
\left|J_*(z)-\Lambda_c(\beta)|z|^{-2\theta}\right|
&\le C\left(R^\beta+\frac R{|z|}\right).
\end{split}
\]
For this choice, $R/|k|=|k|^{\beta/(1-\beta)}\to0$; hence the condition
$|k|\ge4(R+1)$ holds for all sufficiently large $|k|$.  Choosing
$R=|k|^{1/(1-\beta)}$ and $R=|z|^{1/(1-\beta)}$ proves
\eqref{eq:Gammarate} and \eqref{eq:Jrate} in that range.  The remaining
compact ranges are absorbed by enlarging the constant.  For the $J_*$
term, this uses the following direct uniform fact: on every compact subset
of $\{|z|\ge2\}$, the same four-sector partition as in
Lemma~\ref{lem:anchored-tail}, now without the marked restriction, gives
$\sup_z|J_*(z)|<\infty$.  Indeed, near $a=z$ or $b=z$ the filter is a
uniformly bounded smooth tail and the only collision conditions are
$c,\theta<1$ and $c+2\theta<2$; away from those two neighborhoods the
anchored factors are uniformly bounded and Lemma~\ref{lem:roughconvolution}
applies.  To verify local
integrability on the whole line explicitly, fix $T>0$ and
integrate the absolute value of \eqref{eq:Jstar} over $|z|\le T$.  A
partition around the filter endpoints leaves the following possible pinned
collision surpluses:
\[
\begin{array}{c|c}
\text{colliding variables}&
\text{volume plus filter powers minus edge powers}\\
\hline
\{a,b\}&2+2\beta-c=\theta+2h\\
\{a,z\}\text{ or }\{b,z\}&2+\beta-\theta\\
\{a,b,z\}&3+2\beta-c-2\theta=1+2h-\theta.
\end{array}
\]
They are positive because $h>0$, $\beta>-1$, and $0<\theta<1/2$.
Away from a filter endpoint, pair collisions require only $c<1$ and
$\theta<1$, and the free triple collision has surplus
$2-c-2\theta>0$.  If exactly one filter variable leaves a fixed compact
set, \eqref{eq:rhotail} gives the integrable radial power
$r^{\beta-1-c-\theta}$.  If both escape together, retaining the common
translation coordinate gives
$r^{2\beta-2-c-2\theta}r$, again integrable.  This proves
$\int_{-T}^T|J_*(z)|\dd z<\infty$.
\end{proof}

\section{Chaos decomposition and the leading contraction}
\label{sec:chaos-decomposition}

At unit scale let
\begin{equation}
f_i=L_{i+1}^{(\beta)}-L_i^{(\beta)},
\qquad Y_i=I_q(f_i)=X_{i+1}-X_i,
\label{eq:stationarykernel}
\end{equation}
and let $U_N:L^2(\R)\to L^2(\R)$ be the unitary dilation
\[
(U_Ng)(x)=N^{1/2}g(Nx).
\]
For the mesh-$N^{-1}$ increment kernel write
\[
f_{i,N}=L_{(i+1)/N}^{(\beta)}-L_{i/N}^{(\beta)}.
\]

\begin{lemma}[Exact stationary reduction]\label{lem:stationaryreduction}
For every $i,N$,
\begin{equation}
f_{i,N}=N^{-h}U_N^{\otimes q}f_i.
\label{eq:unitaryscaling}
\end{equation}
If
\begin{equation}
H_{i,m}=f_i\widetilde\otimes_{q-m}f_i,
\qquad
a_{q,m}=(q-m)!\binom q{q-m}^{\!2},
\label{eq:Him}
\end{equation}
then the normalized variation has the exact decomposition
\begin{equation}
Z_N^{(h)}(t)=
\sum_{m=1}^q a_{q,m}
I_{2m}\!\left(
N^{\theta-1}U_N^{\otimes2m}
\sum_{i< Nt}H_{i,m}
\right).
\label{eq:exactchaosdecomposition}
\end{equation}
Thus, after normalization, the explicit power of $N$ is independent of the
visible exponent $h$; all dependence on $h$ remains inside the stationary
filter $\varphi_\beta$.
\end{lemma}

\begin{proof}
In the definition of $f_{i,N}$ set $u=v/N$.  Homogeneity of
$g_{\boldsymbol\gamma}$ gives
\[
f_{i,N}(\mathbf x)
=N^{-\beta-\alpha-1}f_i(N\mathbf x)
=N^{-h}U_N^{\otimes q}f_i(\mathbf x),
\]
because $h=\alpha+q/2+1+\beta$.  Unitary dilations commute with contractions,
so
\[
f_{i,N}\widetilde\otimes_{q-m}f_{i,N}
=N^{-2h}U_N^{\otimes2m}H_{i,m}.
\]
The product formula below, summed over $i$, then gives
$V_N^{(h)}$ with the prefactor $N^{-2h}$.  Multiplication by
$N^{2h+\theta-1}$ proves \eqref{eq:exactchaosdecomposition}.
\end{proof}

The product formula gives
\begin{equation}
I_q(f_{i,N})^2-q!\|f_{i,N}\|_2^2
=\sum_{r=0}^{q-1}r!\binom qr^2
I_{2q-2r}\big(f_{i,N}\widetilde\otimes_r f_{i,N}\big).
\label{eq:productformula}
\end{equation}
Hence $V_N^{(h)}$ is a finite orthogonal sum of even chaoses.  Its only
second-chaos component corresponds to $r=q-1$.

We next enumerate the monomial branches, which makes the word ``all'' in
the residual-chaos argument literal.  For a set $A\subset\{1,\ldots,q\}$
write $\gamma_A=\sum_{a\in A}\gamma_a$.

\begin{proposition}[Finite grammar of contraction branches]
\label{prop:branchgrammar}
For $1\le m\le q$, the kernel $H_{i,m}$ is a finite linear combination of
symmetrizations of
\begin{align}
T_i^{A,B,\varepsilon}(\mathbf x,\mathbf y)
={}&C_{A,B,\varepsilon}
\int_{\R^2}\varphi_\beta(u-i)\varphi_\beta(v-i)
|u-v|^{-c_{A,B}}
\1_{\{\varepsilon(u-v)>0\}}
\notag\\
&\quad\times
\prod_{a\in A}(u-x_a)_+^{\gamma_a}
\prod_{b\in B}(v-y_b)_+^{\gamma_b}\dd u\dd v,
\label{eq:monomialbranch}
\end{align}
where $|A|=|B|=m$, $\varepsilon\in\{-1,1\}$, and
\begin{equation}
c_{A,B}=-2\alpha+\gamma_A+\gamma_B-q+m\ge0.
\label{eq:cAB}
\end{equation}
The equality $c_{A,B}=0$ occurs only when $m=q$.  Every coefficient
$C_{A,B,\varepsilon}$ is a finite sum of products of beta functions.
\end{proposition}

\begin{proof}
Fix two permutations $\sigma,\tau\in\mathfrak S_q$ and a bijection $\pi$
between the $q-m$ contracted coordinates in the two copies.  The labels not
met by $\pi$ form the sets $A$ and $B$.  For a contracted coordinate with
labels $a$ and $b$, \eqref{eq:betaidentity} reads
\begin{align*}
&\int_\R(u-x)_+^{\gamma_a}(v-x)_+^{\gamma_b}\dd x\\
&\quad=\1_{\{u>v\}}B(\gamma_b+1,-\gamma_a-\gamma_b-1)
(u-v)^{\gamma_a+\gamma_b+1}\\
&\qquad\quad+\1_{\{v>u\}}B(\gamma_a+1,-\gamma_a-\gamma_b-1)
(v-u)^{\gamma_a+\gamma_b+1}.
\end{align*}
Thus, when $m<q$, each choice of $\sigma,\tau,\pi$ and of the orientation
of $u-v$ produces exactly one term of the form
\eqref{eq:monomialbranch}.  When $m=q$, there is no contracted coordinate;
we obtain the same two-term grammar by inserting the almost-everywhere
identity
\(
1=\1_{\{u>v\}}+\1_{\{v>u\}}.
\)
The sum of
the $q-m$ powers of $|u-v|$ is
\[
2\alpha-\gamma_A-\gamma_B+q-m=-c_{A,B}.
\]
Every elementary contracted exponent $\gamma_a+\gamma_b+1$ is strictly
negative, so $c_{A,B}>0$ when $q-m\ge1$; when $m=q$, one has
$A=B=\{1,\ldots,q\}$ and the displayed formula gives $c_{A,B}=0$.
The coefficient of the term is the product of the $q-m$ displayed beta
factors, multiplied by $(q!)^{-2}$ and by the finite multiplicity with which
the same triple $(A,B,\varepsilon)$ occurs.  Summing over the
$(q!)^2$ permutation pairs, the finitely many contraction bijections, and
the two orientations proves both the formula and its exhaustivity.
\end{proof}

\begin{lemma}[Truncation justification]\label{lem:diagramtruncation}
The product formula, the monomial expansion, all Fubini rearrangements, and
all covariance closures used below are limits of identities for bounded
compactly supported kernels.  The domination is uniform in the truncation
parameters.
\end{lemma}

\begin{proof}
For $R>2$ and $0<\varepsilon<1$, restrict every temporal variable to
$[-R,R]$ and replace
\[
(z)_+^{\gamma_j}
\quad\hbox{by}\quad
k_{j,R,\varepsilon}(z)
=(z\vee\varepsilon)^{\gamma_j}\1_{\{0<z<R\}}.
\]
The resulting kernels are bounded and compactly supported.  All products
are therefore absolutely integrable, and the product formula, monomial
expansion, and covariance closures follow from ordinary Fubini and finite
sums.

Let $f_i^{R,\varepsilon}$ denote the truncated increment kernel.  Pointwise,
$k_{j,R,\varepsilon}(z)\to(z)_+^{\gamma_j}$ as first
$\varepsilon\downarrow0$ and then $R\uparrow\infty$, and
\[
0\le k_{j,R,\varepsilon}(z)\le(z)_+^{\gamma_j}.
\]
For nonnegative factors $0\le a_j\le b_j$, the telescoping identity gives
\[
0\le\prod_{j=1}^qb_j-\prod_{j=1}^qa_j
\le\sum_{\ell=1}^q(b_\ell-a_\ell)
\prod_{j\ne\ell}b_j.
\]
Apply this inequality to every monomial before expanding
$\|f_i^{R,\varepsilon}-f_i\|_2^2$, and dominate the temporal filter by
$\rho_\beta$.  Every spatial variable is then bounded by the corresponding
untruncated oriented beta integral \eqref{eq:betaidentity}.  The resulting
temporal majorants are finite sums of two-filter graphs of total Riesz
exponent $d$.  Lemma~\ref{lem:roughconvolution}, with $a=d$, makes them
integrable, uniformly in $(R,\varepsilon)$.  Pointwise convergence and
dominated convergence yield
\begin{equation}
\lim_{R\to\infty}\lim_{\varepsilon\downarrow0}
\|f_i^{R,\varepsilon}-f_i\|_{L^2(\R^q)}=0.
\label{eq:truncatedkernelL2}
\end{equation}
Contractions are continuous since
\[
\|f\otimes_r f-g\otimes_r g\|_2
\le(\|f\|_2+\|g\|_2)\|f-g\|_2.
\]
Hence the product formula and every contracted kernel pass to the limit in
$L^2$.  For a covariance closure, mark every factor changed by either
$R$ or $\varepsilon$.
Lemma~\ref{lem:charged-collective-closure} gives its local integrable
majorant, and \eqref{eq:charged-collective-marked-tail} makes every marked
$R$-tail tend to zero uniformly in the lattice displacement.  Thus every
closed diagram converges absolutely.  Since
Proposition~\ref{prop:branchgrammar}
contains only finitely many branches, limits and branch sums may be
interchanged term by term.  This proves all four assertions of the lemma.
\end{proof}

For the maximal contraction $m=1$, the free labels are $r,s$.  Put
\begin{equation}
c_{r,s}=-2\alpha+\gamma_r+\gamma_s-q+1.
\label{eq:crs}
\end{equation}
For each bijection
$\pi:\{1,\ldots,q\}\setminus\{r\}\to
\{1,\ldots,q\}\setminus\{s\}$ define
\begin{align}
D_{r,s}^{+}
&=\frac1{q^2(q-1)!}\sum_\pi
\prod_{j\ne r}B(\gamma_j+1,-\gamma_j-\gamma_{\pi(j)}-1),
\label{eq:Drsplus}\\
D_{r,s}^{-}
&=\frac1{q^2(q-1)!}\sum_\pi
\prod_{j\ne r}B(\gamma_{\pi(j)}+1,
-\gamma_j-\gamma_{\pi(j)}-1).
\label{eq:Drsminus}
\end{align}

\begin{proposition}[Exact maximal contraction]\label{prop:maxcontraction}
With
\[
h_{i,N}(u)=\left(\frac{i+1}{N}-u\right)_+^\beta
-\left(\frac iN-u\right)_+^\beta,
\]
one has
\begin{align}
&(f_{i,N}\otimes_{q-1}f_{i,N})(x,y)
\notag\\
={}&A_{\boldsymbol\gamma,\beta}^{2}
\sum_{r,s=1}^q\int_{\R^2}h_{i,N}(u)h_{i,N}(v)
(u-x)_+^{\gamma_r}(v-y)_+^{\gamma_s}|u-v|^{-c_{r,s}}
\notag\\
&\hspace{24mm}\times
\big[D_{r,s}^{+}\1_{\{u<v\}}
     +D_{r,s}^{-}\1_{\{v<u\}}\big]\dd u\dd v.
\label{eq:exactmaxcontraction}
\end{align}
If $r,s\in M$, then
\begin{equation}
D_{r,s}^{+}=D_{r,s}^{-}
=\frac{\mathcal B_*}{q^2(q-1)!}.
\label{eq:activeD}
\end{equation}
Thus the sum of all active sectors is unoriented and has coefficient
\begin{equation}
D_M=\frac{\mu^2\mathcal B_*}{q^2(q-1)!}.
\label{eq:DM}
\end{equation}
\end{proposition}

\begin{proof}
Fix one permutation in each copy of the symmetrized kernel.  If the free
label in the first copy is $r$ and that in the second is $s$, the remaining
$q-1$ labels are matched by a bijection
\[
\pi:\{1,\ldots,q\}\setminus\{r\}
\longrightarrow\{1,\ldots,q\}\setminus\{s\}.
\]
For $u<v$, applying \eqref{eq:betaidentity} to the variable carrying label
$j$ in the first copy and label $\pi(j)$ in the second gives
\[
B(\gamma_j+1,-\gamma_j-\gamma_{\pi(j)}-1)
|u-v|^{\gamma_j+\gamma_{\pi(j)}+1}.
\]
Multiplying over $j\ne r$ yields the coefficient in
\eqref{eq:Drsplus}; the power of $|u-v|$ is
\begin{align*}
\sum_{j\ne r}(\gamma_j+\gamma_{\pi(j)}+1)
&=2\alpha-\gamma_r-\gamma_s+q-1\\
&=-c_{r,s}.
\end{align*}
For $v<u$, the beta identity exchanges the first beta argument and gives
the product in \eqref{eq:Drsminus}.  For fixed $(r,s,\pi)$ there are
exactly $(q-1)!$ pairs of permutations: the contracted positions in the
first copy may be ordered arbitrarily, and their order in the second copy
is then fixed by $\pi$.  Division by the two symmetrization factors gives
$(q-1)!/(q!)^2=1/[q^2(q-1)!]$.  This proves
\eqref{eq:exactmaxcontraction} including its coefficient.

If $r,s\in M$, deleting row $r$ and column $s$ from $W$ leaves the same two
label multisets, up to permutations.  The sum over $\pi$ is therefore the
permanent $\mathcal B_*$.  Since the permanent is invariant under row and
column permutations, both orientations equal
$\mathcal B_*/[q^2(q-1)!]$, proving \eqref{eq:activeD}.  Finally there are
exactly $\mu^2$ pairs in $M^2$, which gives \eqref{eq:DM}.
\end{proof}

\begin{lemma}[Scale of one active branch]\label{lem:scaling}
Fix a monomial branch of the contraction $f_{i,N}\otimes_{q-1}f_{i,N}$
whose two free labels belong to $M$.  After the change of variables
$u=i/N+x/N$ in the two filter variables, the branch has deterministic scale
\begin{equation}
N^{-2\beta-2+c}=N^{-2h-\theta}.
\label{eq:branchscale}
\end{equation}
Consequently multiplication of the sum over $i$ by
$N^{2h+\theta-1}$ leaves a Riemann sum with prefactor $1/N$.
\end{lemma}

\begin{proof}
The $q-1$ spatial contractions are evaluated with
\begin{equation}
\begin{aligned}
\int_\R(u-x)_+^{\gamma_j}(v-x)_+^{\gamma_k}\dd x
={}&\1_{\{u<v\}}
B(\gamma_j+1,-\gamma_j-\gamma_k-1)
(v-u)^{\gamma_j+\gamma_k+1}\\
&+\1_{\{u>v\}}
B(\gamma_k+1,-\gamma_j-\gamma_k-1)
(u-v)^{\gamma_j+\gamma_k+1},
\end{aligned}
\label{eq:betaidentity}
\end{equation}
with equality away from the null diagonal $u=v$.  Homogeneity of the
contracted factors
contributes $N^c$, the two filter variables contribute $N^{-2\beta-2}$,
and
\[
-2\beta-2+c=-2(h-H_0)-2+2(1-H_0)-\theta
=-2h-\theta.
\]
Summation produces $N$ terms, so the normalization in
\eqref{eq:mainlimit} is forced.
\end{proof}

We use the following elementary two-point Toeplitz estimate.

\begin{lemma}[Discrete and mixed Toeplitz estimates]\label{lem:toeplitz}
Let $0<a<1$ and $\delta>0$, and set
\begin{equation}
\mathfrak r_N(a,\delta)=
\begin{cases}
N^{-\delta},&\delta<1-a,\\
N^{-(1-a)}\log N,&\delta=1-a,\\
N^{-(1-a)},&\delta>1-a.
\end{cases}
\label{eq:rN}
\end{equation}
If an even sequence $G$ satisfies
\[
|G(k)-\Lambda|k|^{-a}|\le C|k|^{-a-\delta},\qquad |k|\ge1,
\]
and $G(0)$ is finite, then, uniformly for $t$ in compact subsets of
$\R_+$,
\begin{align}
&\left|N^{a-2}\sum_{i,j< Nt}G(j-i)
-\Lambda\int_0^t\!\int_0^t|s-r|^{-a}\dd s\dd r\right|
\le C\mathfrak r_N(a,\delta).
\label{eq:discretetoeplitz}
\end{align}
If $J$ is locally integrable, $\int_{-2}^2|J(z)|\dd z<\infty$, and
\[
|J(z)-\Lambda|z|^{-a}|\le C|z|^{-a-\delta},
\qquad |z|\ge2,
\]
then
\begin{align}
&\left|N^{a-2}\sum_{i< Nt}\int_0^{Nt}J(u-i)\dd u
-\Lambda\int_0^t\!\int_0^t|s-r|^{-a}\dd s\dd r\right|
\le C\mathfrak r_N(a,\delta).
\label{eq:mixedtoeplitz}
\end{align}
\end{lemma}

\begin{proof}
Fix $0\le t\le T$, put $n=\lfloor Nt\rfloor$ and $t_N=n/N$.  Write
$G(k)=\Lambda |k|^{-a}+R(k)$ for $k\ne0$.  Evenness gives
\begin{align}
N^{a-2}\sum_{i,j<n}G(j-i)
={}&N^{a-2}nG(0)
 +\frac{2\Lambda}{N}\sum_{k=1}^{n-1}
 (t_N-k/N)(k/N)^{-a}\notag\\
&+2N^{a-2}\sum_{k=1}^{n-1}(n-k)R(k).
\label{eq:toeplitzsplit}
\end{align}
The diagonal term is $O_T(N^{a-1})$.  The function
$x\mapsto(t_N-x)x^{-a}$ is nonnegative and decreasing on $(0,t_N)$.
Consequently the difference between its right Riemann sum and its integral
is bounded by the missing first cell plus the first rectangle:
\begin{align}
&\left|\frac1N\sum_{k=1}^{n-1}(t_N-k/N)(k/N)^{-a}
-\int_0^{t_N}(t_N-x)x^{-a}\dd x\right|\notag\\
&\hspace{20mm}\le
\int_0^{1/N}t_Nx^{-a}\dd x
+\frac1N t_N(1/N)^{-a}
\le C_TN^{-(1-a)}.
\label{eq:singularriemann}
\end{align}
Since
\[
2\int_0^t(t-x)x^{-a}\dd x
=\int_0^t\!\int_0^t|s-r|^{-a}\dd s\dd r
=\frac{2t^{2-a}}{(1-a)(2-a)},
\]
and the last function has bounded derivative on $[0,T]$, replacing $t_N$
by $t$ costs $O_T(N^{-1})$.

For the remainder, $|R(k)|\le Ck^{-a-\delta}$ and $n-k\le NT$, whence
\begin{equation}
N^{a-2}\sum_{k=1}^{n-1}(n-k)|R(k)|
\le C_TN^{a-1}\sum_{k=1}^{\lfloor NT\rfloor}k^{-a-\delta}.
\label{eq:toeplitzremainder}
\end{equation}
The sum is bounded respectively by
$CN^{1-a-\delta}$, $C\log N$, or $C$ according as
$a+\delta<1$, $a+\delta=1$, or $a+\delta>1$.  Combining
\eqref{eq:toeplitzsplit}--\eqref{eq:toeplitzremainder} proves
\eqref{eq:discretetoeplitz} with exactly the three cases in
\eqref{eq:rN}.

For the mixed estimate set $J(z)=\Lambda|z|^{-a}+R(z)$ on $|z|>2$.
The contribution of $|u-i|\le2$ to both $J$ and $|u-i|^{-a}$ is at most
\begin{equation}
N^{a-2}\sum_{i<n}
\int_{|u-i|\le2}\big(|J(u-i)|+\Lambda|u-i|^{-a}\big)\dd u
\le C_TN^{-(1-a)}.
\label{eq:mixedstrip}
\end{equation}
On the complement, Tonelli's theorem and the same three elementary
integrals give
\begin{align}
&N^{a-2}\sum_{i<n}\int_0^{Nt}|R(u-i)|\1_{\{|u-i|>2\}}\dd u\notag\\
&\quad\le C_TN^{a-1}
\begin{cases}
N^{1-a-\delta},&a+\delta<1,\\
\log N,&a+\delta=1,\\
1,&a+\delta>1.
\end{cases}
\label{eq:mixedremainder}
\end{align}
It remains to compare the main power with its double integral.  Scaling
$u=Nx$ gives
\[
N^{a-2}\sum_{i<n}\int_0^{Nt}|u-i|^{-a}\dd u
=\frac1N\sum_{i<n}\int_0^t|x-i/N|^{-a}\dd x.
\]
For $r_0=|y-y'|$, split the $x$-integral into
$|x-y|\le2r_0$ and its complement.  The first part is at most
$C_ar_0^{1-a}$; on the second, the mean value theorem gives
\[
r_0\int_{2r_0}^{2T}r^{-a-1}\dd r\le C_{a,T}r_0^{1-a}.
\]
Therefore
\begin{equation}
\int_0^t\big||x-y|^{-a}-|x-y'|^{-a}\big|\dd x
\le C_{a,T}|y-y'|^{1-a}.
\label{eq:riesztranslationmodulus}
\end{equation}
Averaging \eqref{eq:riesztranslationmodulus} over
$y\in[i/N,(i+1)/N]$, then summing $i$, bounds the Riemann error by
$C_TN^{-(1-a)}$; the last partial cell has the same bound.  Together with
\eqref{eq:mixedstrip}--\eqref{eq:mixedremainder}, this proves
\eqref{eq:mixedtoeplitz}.
\end{proof}

Define the normalized active deterministic kernel
\begin{align}
F_{N,t}(x,y)
={}&\frac{D_M}{N}\sum_{i< Nt}
\int_{\R^2}\varphi_\beta(a)\varphi_\beta(b)|a-b|^{-c}
\notag\\
&\quad\times
\left(\frac{i+a}{N}-x\right)_+^{\gamma_*}
\left(\frac{i+b}{N}-y\right)_+^{\gamma_*}\dd a\dd b,
\label{eq:FNactive}
\end{align}
and set
\begin{equation}
F_t=D_M\Lambda_c(\beta)K_t^{(\gamma_*)}.
\label{eq:Ftactive}
\end{equation}
Lemma~\ref{lem:scaling} and Proposition~\ref{prop:maxcontraction} show that
the active part of $Z_N^{(h)}$ is exactly
\begin{equation}
Z_{N,*}^{(h)}(t)=q^2(q-1)!A_{\boldsymbol\gamma,\beta}^2 I_2(F_{N,t}).
\label{eq:randomactive}
\end{equation}

\begin{proposition}[Strong active-kernel convergence]\label{prop:active}
For every $T>0$,
\begin{equation}
\sup_{0\le t\le T}\|F_{N,t}-F_t\|_{L^2(\R^2)}
\le C_T\ell_{\mathrm{act},N}N^{-\kappa_{\mathrm{act}}},
\label{eq:activekernelrate}
\end{equation}
where
\[
\kappa_{\mathrm{act}}
=\frac12\min\{\delta_\beta,1-2\theta\},
\qquad \delta_\beta=\frac{-\beta}{1-\beta},
\]
and $\ell_{\mathrm{act},N}=(\log N)^{1/2}$ if
$\delta_\beta=1-2\theta$, and $1$ otherwise.  Consequently
\begin{equation}
Z_{N,*}^{(h)}(t)\longrightarrow
C_{\boldsymbol\gamma,h}I_2(K_t^{(\gamma_*)})
\quad\text{in every }\mathbb D^{k,p},
\label{eq:activeconvergence}
\end{equation}
locally uniformly in $t$.
\end{proposition}

\begin{proof}
We first identify the weak limit.  For $f\in C_c^\infty(\R)$ put
\[
\mathcal A_f(s)=\int_\R(s-x)_+^{\gamma_*}f(x)\dd x.
\]
Testing \eqref{eq:FNactive} against $f\otimes g$ gives
\begin{align*}
\langle F_{N,t},f\otimes g\rangle
={}&D_M\int_{\R^2}\varphi_\beta(a)\varphi_\beta(b)|a-b|^{-c}\\
&\quad\times\frac1N\sum_{i<Nt}
\mathcal A_f\!\left(\frac{i+a}{N}\right)
\mathcal A_g\!\left(\frac{i+b}{N}\right)\dd a\dd b.
\end{align*}
Let $\nu=\gamma_*+1\in(0,1/2)$.  Proposition~\ref{prop:displacedriemann},
proved in Appendix~\ref{app:riemann}, gives, for every $R\ge2$, the
quantitative bound
\begin{align}
&\sup_{0\le t\le T}\left|
\langle F_{N,t},f\otimes g\rangle
-D_M\Lambda_c(\beta)\int_0^t
\mathcal A_f(s)\mathcal A_g(s)\dd s\right|\notag\\
&\hspace{30mm}\le C_{R,T,f,g}N^{-\nu}+C_{T,f,g}R^\beta.
\label{eq:weaktestbound}
\end{align}
First let $N\to\infty$ and then $R\to\infty$.  Tensor products of smooth
compactly supported functions are dense in $L^2(\R^2)$.  In addition,
\eqref{eq:FNnorm}, Lemma~\ref{lem:pairedrate}, and
Lemma~\ref{lem:toeplitz} give
$\sup_N\sup_{t\le T}\|F_{N,t}\|_2<\infty$.  Hence convergence on this
dense class extends to every $L^2$ test function and yields
\begin{equation}
F_{N,t}\rightharpoonup F_t\quad\text{weakly in }L^2(\R^2).
\label{eq:activeweak}
\end{equation}

Let
\[
b_{\gamma_*}=B(\gamma_*+1,-2\gamma_*-1).
\]
Applying \eqref{eq:betaidentity} to the two free variables in the squared
norm yields the exact identity (the complete Fubini calculation is recorded
in Proposition~\ref{prop:activescalarproducts} of
Appendix~\ref{app:riemann})
\begin{equation}
\|F_{N,t}\|_2^2
=D_M^2b_{\gamma_*}^2N^{2\theta-2}
\sum_{i,j< Nt}\Gamma_*(j-i).
\label{eq:FNnorm}
\end{equation}
Lemma~\ref{lem:pairedrate} and
\eqref{eq:discretetoeplitz}, with $a=2\theta$, show that this converges to
\[
D_M^2b_{\gamma_*}^2\Lambda_c(\beta)^2
\int_0^t\!\int_0^t|s-r|^{-2\theta}\dd s\dd r
=\|F_t\|_2^2,
\]
and that the difference of the squared norms is bounded by
$C_T\mathfrak r_N(2\theta,\delta_\beta)$.

For the quantitative scalar product, first take $t_N=\lfloor Nt\rfloor/N$.
Two further applications of \eqref{eq:betaidentity} give
\begin{equation}
\langle F_{N,t_N},F_{t_N}\rangle
=D_M^2\Lambda_c(\beta)b_{\gamma_*}^2N^{2\theta-2}
\sum_{i< Nt}\int_0^{\lfloor Nt\rfloor}J_*(u-i)\dd u.
\label{eq:FNcross}
\end{equation}
Equations \eqref{eq:Jrate} and \eqref{eq:mixedtoeplitz} bound its difference
from $\|F_{t_N}\|_2^2$ by the same
$C_T\mathfrak r_N(2\theta,\delta_\beta)$.  Therefore
\[
\|F_{N,t_N}-F_{t_N}\|_2^2
\le C_T\mathfrak r_N(2\theta,\delta_\beta).
\]
The remaining endpoint satisfies
$\|F_t-F_{t_N}\|_2\le C|t-t_N|^{1-\theta}$ and is of smaller order.
Taking square roots proves \eqref{eq:activekernelrate}.

It remains to compute the process coefficient without an unknown
normalization.  Proposition~\ref{prop:processconstruction} gives
$A_{\boldsymbol\gamma,\beta}^2
=1/[\mathcal B_0\Lambda_d(\beta)]$.
Using \eqref{eq:DM} in \eqref{eq:randomactive}, the limiting coefficient is
\[
q^2(q-1)!A_{\boldsymbol\gamma,\beta}^2D_M\Lambda_c(\beta)
=\frac{\mu^2\mathcal B_*}{\mathcal B_0}
\frac{\Lambda_c(\beta)}{\Lambda_d(\beta)}
=C_{\boldsymbol\gamma,h}.
\]
Finally, on the second chaos, $L^2$ kernel convergence is equivalent to
convergence in every fixed $\mathbb D^{k,p}$ by hypercontractivity and the
explicit Malliavin derivative formula.  This proves
\eqref{eq:activeconvergence}.
\end{proof}

\section{Elimination of the residual chaoses}
\label{sec:residual-chaoses}

We record the diagram estimate in a form that will also yield rates.  Let
$T_{N,r}$ be the projection of $Z_N^{(h)}$ on chaos $2q-2r$.  For
$1\le m\le q$ put
\begin{equation}
p_m=2\sum_{j=1}^m\theta_j,
\qquad \theta_j=-2\gamma_j-1,
\label{eq:pm}
\end{equation}
and write
\[
S_n(p)=\sum_{0\le i,j<n}(1+|i-j|)^{-p}.
\]
In terms of \eqref{eq:collective-weight}, also put
\begin{equation}
 \mathscr S_n(p)=\sum_{0\le i,j<n}\psi_p(j-i).
 \label{eq:collective-toeplitz-sum}
\end{equation}

\begin{lemma}[Exhaustive closure of two monomial branches]
\label{lem:exhaustive-closure}
Fix $1\le m\le q$.  Let $B_i$ and $\widetilde B_j$ be two, possibly
different, symmetrized monomial summands in the expansion of $H_{i,m}$ and
$H_{j,m}$ from Proposition~\ref{prop:branchgrammar}.  Denote their four
free-label sets by $A,B,C,D$, each of cardinality $m$, and put
\[
 \Theta_A=\sum_{a\in A}\theta_a,
 \qquad \theta_a=-2\gamma_a-1,
\]
with analogous notation for $B,C,D$.  Then
$\langle B_i,\widetilde B_j\rangle$ is a finite linear combination of
closed four-vertex graph integrals.  Every such closure has exactly $2m$
edges joining the lattice block $i$ to the lattice block $j$.  If a cross
edge pairs labels $a$ and $c$, its Riesz exponent is
\begin{equation}
 \lambda_{a,c}=-\gamma_a-\gamma_c-1
 =\frac{\theta_a+\theta_c}{2}>0.
 \label{eq:cross-edge-weight}
\end{equation}
Consequently the total macroscopic weight of every closure is exactly
\begin{equation}
 p(A,B;C,D)
 =\frac{\Theta_A+\Theta_B+\Theta_C+\Theta_D}{2}
 \ge 2\sum_{a=1}^m\theta_a=p_m.
 \label{eq:closure-weight}
\end{equation}
There is a constant, independent of $i,j$, such that
\begin{equation}
 |\langle B_i,\widetilde B_j\rangle|
 \le C_{B,\widetilde B}\,
 \psi_{p(A,B;C,D)}(j-i).
 \label{eq:cross-branch-decay}
\end{equation}
The number of closures and the sum of their constants are bounded by a
finite constant depending only on $q$, the exponent vector, and the filter.
\end{lemma}

\begin{proof}
Before symmetrization, a branch in $H_{i,m}$ has two groups of $m$ free
spatial variables, carrying the labels in $A$ and $B$ and attached to its
two temporal filter vertices.  A branch in $H_{j,m}$ has the analogous
groups $C,D$.  Expand both symmetrizations.  A scalar product in
$L^2(\R^{2m})$ pairs every free coordinate of the first branch with exactly
one free coordinate of the second; after fixing the first ordering, such a
pairing is indexed by a permutation in $\mathfrak S_{2m}$.  Thus there are
exactly $2m$ cross-block spatial integrations in each summand and only
finitely many summands.  This argument also covers cross terms between
different label sets and different orientations.

More formally, if $F_i$ and $G_j$ denote the corresponding unsymmetrized
branches, self-adjointness of orthogonal symmetrization gives
\begin{equation}
 \langle\Sym F_i,\Sym G_j\rangle
 =\frac1{(2m)!}\sum_{\pi\in\mathfrak S_{2m}}
   \langle F_i,G_j\circ\pi\rangle.
 \label{eq:relative-pairings}
\end{equation}
For fixed $\pi$, the two half-space terms in each of the $2m$ beta
identities are indexed by an orientation vector
$\omega\in\{-,+\}^{2m}$.  Hence one branch pair produces at most
$(2m)!2^{2m}$ closed oriented graphs before equal terms are collected.

Suppose one of these integrations pairs a factor
$(u-x)_+^{\gamma_a}$ from the block $i$ with a factor
$(v-x)_+^{\gamma_c}$ from the block $j$.  Formula
\eqref{eq:betaidentity} replaces it, on each of its two orientation
half-spaces, by a finite beta coefficient times
\[
 |u-v|^{\gamma_a+\gamma_c+1}
 =|u-v|^{-\lambda_{a,c}},
\]
where \eqref{eq:cross-edge-weight} follows from the definition of
$\theta_a$.  The temporal variables have the form $i+x_v$ in the first
branch and $j+x_w$ in the second.  Hence every such edge is
\[
 |(j-i)+x_w-x_v|^{-\lambda_{a,c}}
\]
and carries the macroscopic translation $j-i$.  Since every labelled slot
in the disjoint union $A\sqcup B\sqcup C\sqcup D$ occurs at one endpoint of
exactly one cross edge,
summing \eqref{eq:cross-edge-weight} over the $2m$ edges gives the equality
in \eqref{eq:closure-weight}.  The sequence
$\theta_1\le\cdots\le\theta_q$ is increasing, and each of the four sets has
cardinality $m$; therefore each $\Theta$ is at least
$\sum_{a=1}^m\theta_a$, proving the inequality in
\eqref{eq:closure-weight}.

The two internal contracted packets of each branch supply only local edges
inside the block $i$ or $j$.  Split each aggregate factor
$|u-v|^{-c_{A,B}}$ into its $q-m$ labelled port edges.  Their total weight
is
\[
 \sum_{\mathrm{contracted}\ (a,b)}
 \frac{\theta_a+\theta_b}{2}
 =d-\frac{\Theta_A+\Theta_B}{2}=c_{A,B}.
\]
When $m=q$ this sum is empty and the factor is $1$.  The norm closure pairs
every remaining port exactly once, with weights
\eqref{eq:cross-edge-weight}.  Hence each of the four filter vertices
carries once and only once the complete charge list
$c_1,\ldots,c_q$, and every connected component meets both lattice
blocks.  Lemma~\ref{lem:charged-collective-closure} applies with
$k=j-i$ and $p_D=p(A,B;C,D)$, proving
\eqref{eq:cross-branch-decay}.

Finally, Proposition~\ref{prop:branchgrammar} has finitely many label sets,
orientations, contraction bijections, and beta coefficients, while
\eqref{eq:relative-pairings} and its orientation expansion give at most
$(2m)!2^{2m}$ closed graphs for each branch pair.  The charged-closure
constants do not depend on the compact or small-scale truncation parameters.
Summing the corresponding finite constants proves the last assertion,
uniformly throughout the truncation procedure.
\end{proof}

\begin{proposition}[Residual diagrams]\label{prop:residual}
Under the assumptions of Theorem~\ref{thm:main}:
\begin{enumerate}
\item every non-active branch of $T_{N,q-1}$ converges to zero in $L^2$;
\item for each $0\le r\le q-2$, $T_{N,r}\to0$ in $L^2$;
\item the conclusions remain valid in every fixed $\mathbb D^{k,p}$ and
uniformly on compact time intervals.
\end{enumerate}
\end{proposition}

\begin{proof}
For a second-chaos branch with free labels $(a,b)$, the two macroscopic edges
have total weight
\begin{equation}
\eta_{a,b}=\theta_a+\theta_b.
\label{eq:etaab}
\end{equation}
Expand the finitely many monomial summands and orientations in this label
sector.  Wiener isometry gives a finite sum over $i,j$ of their pairwise
scalar products.  Lemma~\ref{lem:exhaustive-closure}, with $m=1$ and free
sets $\{a\},\{b\},\{a\},\{b\}$, shows that every closure has total weight
$\eta_{a,b}$.  Hence
\begin{equation}
\|T_{N,q-1}^{a,b}(t)\|_2^2
\le C N^{2\theta-2}\mathscr S_{\lfloor Nt\rfloor}(\eta_{a,b}).
\label{eq:secondbranchbound}
\end{equation}
If $(a,b)\notin M^2$, then $\eta_{a,b}>2\theta$.  Since
\begin{equation}
S_n(p)\le\mathscr S_n(p)\le C_p
\begin{cases}
n^{2-p},&0<p<1,\\
n\log(1+n),&p=1,\\
n,&p>1,
\end{cases}
\label{eq:toeplitzthree}
\end{equation}
the right-hand side of \eqref{eq:secondbranchbound} tends to zero in all
three cases.  If both labels are maximal, the branch is one of those already
included in Proposition~\ref{prop:active}.

For $r\le q-2$, put $m=q-r\ge2$.  The projection $T_{N,r}$ is a finite
sum of the branches in Proposition~\ref{prop:branchgrammar}.  Expand its
squared norm, including all cross terms between distinct label sets,
orientations, and symmetrizations.  Wiener isometry and the unitarity of
$U_N$ give the prefactor
$(2m)!N^{2\theta-2}$ and a double sum over the lattice indices.
Lemma~\ref{lem:exhaustive-closure} indexes every resulting closure and
shows that its macroscopic weight is at least $p_m$.  Since the number and
the constants of the closures have a finite sum, one obtains
\begin{equation}
\|T_{N,r}(t)\|_2^2
\le C N^{2\theta-2}
\mathscr S_{\lfloor Nt\rfloor}(p_m).
\label{eq:toeplitzresidual}
\end{equation}
If $p_m<1$, this is
$O(N^{2\theta-p_m})=O(N^{-2\sum_{j=2}^m\theta_j})$; if $p_m=1$, it is
$O(N^{2\theta-1}\log N)$; and if $p_m>1$, it is
$O(N^{2\theta-1})$.  Every expression tends to zero because $m\ge2$ and
$\theta<1/2$.  The bound is stable when one filter is truncated, and its
marked remainder tends uniformly to zero by
\eqref{eq:charged-collective-marked-tail}; letting the truncation radius
tend to infinity completes the proof.  Orthogonality of the chaoses gives the full
$L^2$ statement.  Hypercontractivity and the explicit action of Malliavin
derivatives on a fixed chaos give the last assertion.
\end{proof}

\begin{proposition}[Branch valuation formula]\label{prop:valuation}
Let $\mathfrak b$ be any closed monomial branch of chaos $2m$, and let
$p(\mathfrak b)$ be the exact sum of the exponents of its $2m$
macroscopic edges.  After the normalization in
\eqref{eq:exactchaosdecomposition}, the squared $L^2$ norm has polynomial
degree at most
\begin{equation}
\Delta(\mathfrak b)=2\theta-2+\omega(p(\mathfrak b)),
\qquad
\omega(p)=\max\{2-p,1\},
\label{eq:branchvaluation}
\end{equation}
with one logarithm at $p=1$.  Moreover:
\begin{enumerate}
\item if $m=1$ and the free labels belong to $M^2$, then
$p=2\theta$ and $\Delta=0$;
\item if $m=1$ and the branch is not active, then $p>2\theta$ and
$\Delta<0$;
\item if $m\ge2$ and $p<1$, then
$\Delta\le-2\sum_{j=2}^m\theta_j<0$;
\item if $m\ge2$ and $p\ge1$, then
the upper degree bound is $\Delta=2\theta-1<0$.
\end{enumerate}
Thus the active sectors listed in Proposition~\ref{prop:maxcontraction} are
the only branches of valuation zero.
\end{proposition}

\begin{proof}
Lemma~\ref{lem:charged-collective-closure} contracts every local collision
and controls the collective lattice tail.  Thus a covariance branch is
bounded by
\[
C N^{2\theta-2}\mathscr S_{\lfloor Nt\rfloor}(p(\mathfrak b)).
\]
Equation \eqref{eq:toeplitzthree} gives \eqref{eq:branchvaluation}.  The
four cases follow from $p=\theta_a+\theta_b$ in the second chaos and
$p\ge p_m=2\sum_{j=1}^m\theta_j$ in chaos $2m\ge4$.
\end{proof}

\begin{proof}[Proof of Theorem~\ref{thm:main}]
Fix $t_1,\ldots,t_\ell$.  For each $j$, Proposition~\ref{prop:active}
gives $L^2$ convergence of the active sector to $Z(t_j)$, while
\eqref{eq:secondbranchbound} and \eqref{eq:toeplitzresidual}, with
$n=\max_j\lfloor Nt_j\rfloor$, make the finite residual sum tend to zero in
$L^2$.  Hence
\[
\|Z_N^{(h)}(t_j)-Z(t_j)\|_2\longrightarrow0
\qquad(1\le j\le\ell).
\]
Therefore
\[
\left\|
\bigl(Z_N^{(h)}(t_j)-Z(t_j)\bigr)_{j=1}^\ell
\right\|_{L^2(\Omega;\R^\ell)}^2
=\sum_{j=1}^\ell\|Z_N^{(h)}(t_j)-Z(t_j)\|_2^2
\longrightarrow0,
\]
which is \eqref{eq:mainlimit}.  The one-time $\mathbb D^{k,p}$ conclusion
is already contained in Propositions~\ref{prop:active} and
\ref{prop:residual}.
\end{proof}

\section{Quantitative and functional convergence}
\label{sec:functional-convergence}

Throughout this section we assume the hypotheses of
Theorem~\ref{thm:main}.

We make the exponents explicit.  Write
\[
\theta_j=-2\gamma_j-1,
\qquad \theta=\theta_1.
\]
If $\mu<q$, put $g_*=\gamma_*-\gamma_{\mu+1}$.  Define
\begin{align}
\delta_\beta&=\frac{-\beta}{1-\beta}
=\frac{H_0-h}{1+H_0-h},\label{eq:deltabeta}\\
\kappa_{\mathrm{act}}
&=\frac12\min\{\delta_\beta,1-2\theta\},\label{eq:kactive}\\
\kappa_2&=
\begin{cases}
\displaystyle\min\left\{g_*,\frac{1-2\theta}{2}\right\},&\mu<q,\\[2mm]
+\infty,&\mu=q,
\end{cases}\label{eq:ksecond}\\
\kappa_{\ge4}&=\min\left\{\theta_2,
\frac{1-2\theta}{2}\right\},\label{eq:khigher}\\
\kappa&=\min\{\kappa_{\mathrm{act}},\kappa_2,
\kappa_{\ge4}\}.\label{eq:kglobal}
\end{align}
When $q=2$ the symbol $\theta_2$ has its literal meaning; for a completely
isotropic vector it equals $\theta$.  For $N\ge2$, set
$\ell_N=(\log N)^{1/2}$ if at least one of the following equalities holds:
\begin{equation}
\delta_\beta=1-2\theta,
\qquad
\mu<q\ \text{ and }\ \theta+\theta_{\mu+1}=1,
\qquad
2(\theta+\theta_2)=1.
\label{eq:logfactor}
\end{equation}
Set $\ell_N=1$ otherwise.

\begin{proof}[Proof of the quantitative assertion in
Theorem~\ref{thm:main}]
Lemma~\ref{lem:pairedrate} gives an error
$O(|k|^{-2\theta-\delta_\beta})$ for the squared active norm and the
corresponding mixed error for its scalar product with the limiting kernel.
Indeed the truncation error is $O(R^\beta)$, the compact Taylor error is
$O(R/|k|)$, and the choice $R=|k|^{1/(1-\beta)}$ gives
$R^\beta=|k|^{-\delta_\beta}$.  The discrete and mixed Toeplitz sums with
main exponent $2\theta$ cap the squared-norm rate at $1-2\theta$.  Taking a
square root gives \eqref{eq:kactive}; equality yields the indicated
square-root logarithm.

Suppose first that $\mu<q$.  For a non-active second-chaos branch $(a,b)$,
if
$\eta_{a,b}=\theta_a+\theta_b<1$, equation
\eqref{eq:secondbranchbound} gives the $L^2$ exponent
\[
\frac{\eta_{a,b}}2-\theta.
\]
The smallest value is attained when one label is maximal and the other is
$\mu+1$; it equals
$(\theta_{\mu+1}-\theta)/2
=\gamma_*-\gamma_{\mu+1}=g_*$.  If $\eta_{a,b}\ge1$, the Toeplitz sum gives
the capped exponent $(1-2\theta)/2$.  This proves \eqref{eq:ksecond} when
$\mu<q$.  If $\mu=q$, every second-chaos sector is active, so there is no
second-chaos remainder and the correct convention is $\kappa_2=+\infty$.

For chaos $2m\ge4$, if $p_m<1$, equation
\eqref{eq:toeplitzresidual} gives the norm exponent
$\sum_{j=2}^m\theta_j$, whose minimum is $\theta_2$ at $m=2$.
If $p_m\ge1$, the exponent is capped at $(1-2\theta)/2$.  This proves
\eqref{eq:khigher}.
Orthogonality gives \eqref{eq:kglobal}; hypercontractivity gives
\eqref{eq:malliavinrate}.
Since the variables are already coupled on the same isonormal Gaussian
space, the same estimate also implies convergence in every fixed
Wasserstein distance.
\end{proof}

\begin{lemma}[Block inequality]\label{lem:block}
For every $T>0$, $k\ge0$, and $p>1$, there is a constant such that, for
$0\le j<j+L\le NT$,
\begin{equation}
\left\|Z_N^{(h)}\!\left(\frac{j+L}{N}\right)
-Z_N^{(h)}\!\left(\frac jN\right)\right\|_{\mathbb D^{k,p}}
\le C\left(\frac LN\right)^{1-\theta}.
\label{eq:blockbound}
\end{equation}
Consequently the polygonal interpolation satisfies, for all
$s,t\in[0,T]$,
\begin{equation}
\|\overline Z_N^{(h)}(t)-\overline Z_N^{(h)}(s)\|_{\mathbb D^{k,p}}
\le C|t-s|^{1-\theta}.
\label{eq:incrementbound}
\end{equation}
\end{lemma}

\begin{proof}
For a random variable in the fixed finite sum
$\bigoplus_{r=0}^{2q}\mathcal C_r$, hypercontractivity and the formula for
Malliavin derivatives imply
\begin{equation}
\|F\|_{\mathbb D^{k,p}}\le C_{q,k,p}\|F\|_2.
\label{eq:finitechaosnorm}
\end{equation}
For the chaos-$2m$ component of a block, the Wiener isometry,
\eqref{eq:exactchaosdecomposition}, and the general graph closure gives
\[
\|\text{block}_{2m}\|_2^2
\le C_mN^{2\theta-2}\mathscr S_L(p_m)
\le C_m\left(\frac LN\right)^{2-2\theta}.
\]
For the second inequality, use $p_m\ge2m\theta\ge2\theta$ when
$p_m<1$, use $L\log(1+L)\le C_\theta L^{2-2\theta}$ when $p_m=1$, and
use $L\le L^{2-2\theta}$ when $p_m>1$.  Summing the finitely many chaoses, taking square
roots, and applying \eqref{eq:finitechaosnorm} proves
\eqref{eq:blockbound}.

If $s,t$ lie in the same cell, polygonal interpolation gives
\[
\|\overline Z_N^{(h)}(t)-\overline Z_N^{(h)}(s)\|_{\mathbb D^{k,p}}
\le CN|t-s|N^{-(1-\theta)}\le C|t-s|^{1-\theta}.
\]
For different cells, split the increment into the two partial endpoint
cells and the complete block between them.  This proves
\eqref{eq:incrementbound}.
\end{proof}

\begin{proof}[Proof of the functional assertion in
Theorem~\ref{thm:main}]
Fix $\eta<\eta'<1-\theta$ and choose a moment order
$R>\max\{p,2\}$ so large that
$\eta'<1-\theta-1/R$.  Lemma~\ref{lem:block} and the Hilbert-valued
Kolmogorov calculation in Proposition~\ref{prop:holderupgrade} of
Appendix~\ref{app:stochasticdetails} give
\begin{equation}
\sup_N\E\|\overline Z_N^{(h)}\|_{C^{\eta'}([0,T])}^R<\infty.
\label{eq:uniformholdermoment}
\end{equation}
The same proposition applies to each $D^j\overline Z_N^{(h)}$, valued in
$L^2(\R^j)$, because Proposition~\ref{prop:finitechaosmalliavin} in the
same appendix converts the block $L^2$ estimate into the corresponding
derivative estimate.  The
Rosenblatt limit satisfies the identical bounds by applying Wiener isometry
to $K_t^{(\gamma_*)}-K_s^{(\gamma_*)}$.

Propositions~\ref{prop:active} and \ref{prop:residual} give convergence in
$L^p(\Omega;L^2(\R^j))$ at every grid-independent time, for
$0\le j\le k$.  If $t\in[m/N,(m+1)/N]$, Lemma~\ref{lem:block} with
$L=1$ gives
\[
\|\overline Z_N^{(h)}(t)-Z_N^{(h)}(t)\|_{\mathbb D^{k,p}}
\le C N^{-(1-\theta)},
\]
so the same fixed-time convergence holds for the interpolation.
Proposition~\ref{prop:holderupgrade} combines this
pointwise convergence on a countable dense set with
\eqref{eq:uniformholdermoment}; its interpolation inequality then yields
\[
\E\|D^j(\overline Z_N^{(h)}-Z)\|_{C^\eta([0,T];L^2(\R^j))}^p\longrightarrow0.
\]
The processes on both sides are H\"older of order $\eta'>\eta$ and hence
belong to the little space $c^\eta$.  Its norm is the restriction of the
$C^\eta$ norm, so the last display is precisely
\eqref{eq:functional}.  Notice that this componentwise statement does not
invoke, or identify with, a Banach-valued Malliavin space.
\end{proof}

\section{The universal Gaussian boundary at zero roughness}
\label{sec:zero-boundary}

The value $h=0$ is not obtained by substituting $h=0$ in
Theorem~\ref{thm:main}.  It is the sharp singular boundary at which the bare
multiple-integral kernel ceases to have finite energy.  We now identify the
renormalized boundary selected by two different regularizations; throughout
this section only \eqref{eq:admissible} is assumed.

\subsection{Logarithmic failure of the bare kernel}

Put
\begin{equation}
 \beta_0=-H_0=\frac d2-1,
 \qquad
 c_0=2B\left(\frac d2,1-d\right).
 \label{eq:boundary-constants}
\end{equation}
For $h>0$ we write explicitly
\begin{equation}
 X_t^{(h)}:=I_q\!\left(L_t^{(\beta_0+h)}\right),
 \qquad
 A_h:=A_{\boldsymbol\gamma,\beta_0+h}.
 \label{eq:analytic-family}
\end{equation}
At the boundary, where the normalizing constant is not defined, the bare
kernel means
\begin{equation}
 \widetilde L_{t,0}(\mathbf x)
 :=\int_\R\big[(t-u)_+^{\beta_0}-(-u)_+^{\beta_0}\big]
 g_{\boldsymbol\gamma}(u\one-\mathbf x)\dd u.
 \label{eq:bare-boundary-kernel}
\end{equation}
The exact energy formula \eqref{eq:Lambdadexplicit} gives the following
residue.

\begin{proposition}[Sharp endpoint residue]
\label{prop:boundary-residue}
As $h\downarrow0$,
\begin{equation}
 h\Lambda_d(\beta_0+h)\longrightarrow c_0,
 \qquad
 A_{\boldsymbol\gamma,\beta_0+h}^2
 \sim\frac{h}{\mathcal B_0c_0}.
 \label{eq:boundary-residue}
\end{equation}
Moreover, for every fixed $\eta>0$,
\begin{equation}
 \int_\varepsilon^\eta\!\int_\varepsilon^\eta
 x^{\beta_0}y^{\beta_0}|x-y|^{-d}\dd x\dd y
 =c_0\log(\varepsilon^{-1})+O_\eta(1).
 \label{eq:boundary-log}
\end{equation}
Consequently $\widetilde L_{t,0}\notin L^2(\R^q)$ for every $t>0$.
The condition $h>0$ is therefore necessary and sufficient for the
unregularized point evaluation used in \eqref{eq:process}; at $h=0$ the
divergence is logarithmic.
\end{proposition}

\begin{proof}
In \eqref{eq:Lambdadexplicit}, use
$\Gamma(1+2h)\to1$, $h/\sin(\pi h)\to1/\pi$, and
$\Gamma(\beta_0+h+1)\to\Gamma(d/2)$.  The reflection formula changes the
resulting constant into $2B(d/2,1-d)=c_0$.  This proves
\eqref{eq:boundary-residue}.

For \eqref{eq:boundary-log}, integrate over $\varepsilon<y<x<\eta$, put
$y=tx$, and use $2\beta_0-d+1=-1$.  The principal term is
\[
 2B(\beta_0+1,1-d)\int_\varepsilon^\eta\frac{\dd x}{x}.
\]
Replacing the lower limit $\varepsilon/x$ in the beta integral by zero
costs $O(1)$, because $\beta_0>-1$.  This proves
\eqref{eq:boundary-log}.  Each of the two common filter endpoints gives
this divergence, so the unregularized energy is infinite.  More generally,
at $\beta=\beta_0+h$ the radial endpoint integral is comparable to
\[
 \int_\varepsilon^1 r^{2h-1}\dd r.
\]
It is finite exactly for $h>0$, logarithmic for $h=0$, and for $h<0$
diverges like $(-2h)^{-1}\varepsilon^{2h}$.  This proves the asserted
necessity and sufficiency.
\end{proof}

\subsection{Two regularizations, one Gaussian selection}

On an auxiliary product space, let $(\zeta_a)_{a\ge0}$ be an independent
family of standard Gaussian variables, and define the cylindrical
Gaussian boundary field
\begin{equation}
 G_0=0,
 \qquad
 G_t=\frac{\zeta_t-\zeta_0}{\sqrt2},\quad t>0.
 \label{eq:boundary-field}
\end{equation}
Thus $\E[G_t^2]=1$ and $\E[G_tG_s]=1/2$ for distinct positive times.

\begin{proof}[Proof of the $h$-regularization assertions in
Theorem~\ref{thm:zero-main}]
By \eqref{eq:processcovariance}, the covariance converges to that of
\eqref{eq:boundary-field}.  After discarding the coefficient at time zero,
the limiting variance of a linear combination is
\[
 \frac12\sum_{t>0}a_t^2+\frac12\left(\sum_{t>0}a_t\right)^2,
\]
so it vanishes only when all remaining coefficients vanish.  Lemma~\ref{lem:boundary-fourcopy}, proved in
Appendix~\ref{app:boundary-diagrams}, shows that every proper contraction
of every variance-normalized finite linear combination is $O(\sqrt h)$ in
tensor norm.  The fixed-chaos Malliavin--Stein inequality and
Cram\'er--Wold therefore prove the first convergence in
\eqref{eq:zero-main-fdd}.  For the
marginal rate, apply the inequality first to the exactly normalized
variable $t^{-h}X_t^{(h)}$.  Since $|t^h-1|=O_t(h)$, the Wasserstein
triangle inequality then gives the first estimate in
\eqref{eq:zero-main-rates}.

For distinct $s,t>0$, $G_t-G_s\sim N(0,1)$ independently of $|t-s|$, so
$G$ has no stochastically continuous modification.  Finally, suppose that
a marginal converged in probability on the original isonormal space.
Fixed-chaos hypercontractivity makes its squared family uniformly
integrable, so the convergence would hold in $L^2$ and the limit would lie
in the closed $q$th chaos.  But a nonzero $q$th-chaos variable with
$q\ge2$ cannot have a Gaussian law: its contraction of order $q-1$ is
nonzero and the product formula makes its fourth cumulant strictly
positive.  This excludes strong convergence.
\end{proof}

There is an equivalent hard-cutoff selection of the same boundary.  For
$0<\varepsilon<e^{-1}$ define
\begin{align}
 \varphi_{\varepsilon,t}(u)
 &=(t-u)^{\beta_0}\1_{\{u\le t-\varepsilon\}}
   -(-u)^{\beta_0}\1_{\{u\le-\varepsilon\}},
 \label{eq:cutoff-filter}\\
 K_{t,\varepsilon}(\mathbf x)
 &=\int_\R\varphi_{\varepsilon,t}(u)
   g_{\boldsymbol\gamma}(u\one-\mathbf x)\dd u,
 \label{eq:cutoff-kernel}\\
 X_{\varepsilon,t}
 &=\frac{I_q(K_{t,\varepsilon})}
 {\sqrt{2\mathcal B_0c_0\log(1/\varepsilon)}}.
 \label{eq:cutoff-process}
\end{align}

\begin{proposition}[Hard-cutoff covariance and common residue]
\label{prop:cutoff-covariance}
For every fixed $t>0$,
\begin{equation}
 \Var I_q(K_{t,\varepsilon})
 =2\mathcal B_0c_0\log(1/\varepsilon)+O_t(1),
 \label{eq:cutoff-variance}
\end{equation}
whereas two distinct positive times share only one endpoint and hence have
leading covariance
$\mathcal B_0c_0\log(1/\varepsilon)$.  The $h$-regularization and cutoff parameters
are related, endpoint by endpoint, by
\begin{equation}
 \frac1{2h}\quad\longleftrightarrow\quad\log(1/\varepsilon).
 \label{eq:boundary-dictionary}
\end{equation}
\end{proposition}

\begin{proof}[Cutoff covariance and selection]
For $t>0$ put
\[
 E_t=\{0,t\},\qquad \sigma_t(0)=-1,\quad \sigma_t(t)=1,
 \qquad L_\varepsilon=\log(1/\varepsilon).
\]
Localize the double time integral defining the covariance in disjoint
fixed neighborhoods of the endpoints in $E_s\cup E_t$.  At a common
endpoint, translation and \eqref{eq:boundary-log} give its sign product
times $c_0L_\varepsilon$.  Distinct endpoint pairs are uniformly separated
and contribute $O_{s,t}(1)$.  On the complement, including the common
left tail, the difference of the two monomials gains one power and is
integrable.  Hence
\begin{equation}
 q!\langle K_{s,\varepsilon},K_{t,\varepsilon}\rangle
 =\mathcal B_0c_0L_\varepsilon
 \sum_{a\in E_s\cap E_t}\sigma_s(a)\sigma_t(a)+O_{s,t}(1).
 \label{eq:cutoff-oriented-covariance}
\end{equation}
This proves \eqref{eq:cutoff-variance} and the stated cross-covariances.

The cutoff part of Lemma~\ref{lem:boundary-fourcopy} bounds every proper
contraction after the displayed approximate normalization by
$O(L_\varepsilon^{-1/2})$.  For the marginal rate, first divide by the
exact standard deviation $\sigma_{\varepsilon,t}$.  The
Malliavin--Stein bound applies to that unit-variance variable, while
\eqref{eq:cutoff-oriented-covariance} gives
\[
 \frac{\sigma_{\varepsilon,t}}
 {\sqrt{2\mathcal B_0c_0L_\varepsilon}}
 =1+O_t(L_\varepsilon^{-1}).
\]
The Wasserstein triangle inequality proves the second estimate in
\eqref{eq:zero-main-rates}; Cram\'er--Wold, with the same
nondegenerate-variance observation as in the $h$-regularization proof, proves the
second convergence in \eqref{eq:zero-main-fdd}.
\end{proof}

The constants in this section are for fixed $q$, fixed
$\boldsymbol\gamma$, and fixed observation times.  Neither boundary
selection is uniform over a growing time grid; the quadratic-variation
result below deliberately takes the boundary limit with $N$ fixed before
sending $N$ to infinity.

\subsection{Brownian dynamics of the boundary energy}

The cylindrical field $G$ has no pathwise quadratic variation.  We instead
study the well-defined boundary-first selection.  Define
\begin{equation}
 Q_{N,h}(t)=\frac1{\sqrt N}\sum_{i<\lfloor Nt\rfloor}
 \left[\bigl(N^h(X_{(i+1)/N}^{(h)}-X_{i/N}^{(h)})\bigr)^2-1\right],
 \label{eq:boundary-qv}
\end{equation}
and define the cutoff statistic by
\begin{equation}
 Q_{N,\varepsilon}(t)=\frac1{\sqrt N}
 \sum_{i<\lfloor Nt\rfloor}\left[
 (X_{\varepsilon,(i+1)/N}-X_{\varepsilon,i/N})^2
 -\E(X_{\varepsilon,(i+1)/N}-X_{\varepsilon,i/N})^2
 \right].
 \label{eq:boundary-cutoff-qv}
\end{equation}
\begin{proof}[Proof of the boundary-energy assertion in
Theorem~\ref{thm:zero-main}]
For fixed $N$, \eqref{eq:zero-main-fdd} gives joint convergence on the
finite observation grid.  The polynomial map from grid values to partial sums of centered
squared increments is continuous; in the cutoff case
\eqref{eq:cutoff-oriented-covariance} also makes the deterministic
centering converge.  Consequently both inner limits converge in
$D([0,T])$ to the partial-sum process associated with
\[
 D_i=\frac{\zeta_{i+1}-\zeta_i}{\sqrt2}.
\]
It is Gaussian and one-dependent, with covariance $-1/2$ at adjacent
indices.  Therefore, for $Y_i=D_i^2-1$,
\[
 \Var(Y_i)=2,
 \qquad
 \Cov(Y_i,Y_{i+1})=\frac12,
\]
and the long-run variance is $3$.  The functional central limit theorem
for a one-dependent sequence gives $N^{-1/2}\sum_{i<N\cdot}Y_i
\Rightarrow\sqrt3B$.  We verify this directly.  An exact expansion gives
\begin{equation}
 \sum_{i=0}^{n-1}Y_i
 =\sum_{i=1}^{n-1}\xi_i+R_n,
 \qquad
 \xi_i=(\zeta_i^2-1)-\zeta_{i-1}\zeta_i,
 \label{eq:boundary-martingale-decomposition}
\end{equation}
where
\[
 R_n=\frac12(\zeta_0^2-1)+\frac12(\zeta_n^2-1)
 -\zeta_{n-1}\zeta_n.
\]
For $\mathcal F_i=\sigma(\zeta_0,\ldots,\zeta_i)$,
\[
 \E[\xi_i\mid\mathcal F_{i-1}]=0,
 \qquad
 \E[\xi_i^2\mid\mathcal F_{i-1}]=2+\zeta_{i-1}^2.
\]
The predictable quadratic variation divided by $N$ converges uniformly on
compact time intervals to $3t$, and the conditional Lindeberg condition
follows from $\sup_i\E\xi_i^4<\infty$.  The martingale functional central
limit theorem yields the required $\sqrt3B$ limit.  Moreover
$\sup_n\E|R_n|^4<\infty$, so a union bound gives
$N^{-1/2}\max_{n\le NT}|R_n|\to0$ in probability.  The continuous mapping
theorem at fixed $N$, followed by this functional limit, proves the
iterated assertion.
\end{proof}

\section{The simultaneous Gaussian--Rosenblatt transition}
\label{sec:crossover}

We now let the roughness and the mesh vary simultaneously.  The complete
transition is stated first in the isotropic model
\begin{equation}
 \gamma_1=\cdots=\gamma_q=\gamma,
 \qquad
 \theta=-2\gamma-1\in(0,1/q),
 \qquad
 d=q\theta.
 \label{eq:crossover-isotropic}
\end{equation}
This is the maximal range in which the second-chaos projection has no
anisotropic residual sector.

The proof is divided according to the orthogonal decomposition into an
active second-chaos sector and a higher-chaos bulk.  Appendix
\ref{app:root-regular} establishes the uniform diagram estimate shared by
both sectors; Appendix~\ref{app:uniform-active} identifies the active
Rosenblatt contribution; and Appendix~\ref{app:bulk} proves the Gaussian
bulk limit, its asymptotic independence from the active sector, and the
block bounds required for tightness.

For $h_N\downarrow0$, put
\begin{align}
 Y_{i,N}
 &=N^{h_N}\bigl(X_{(i+1)/N}^{(h_N)}-X_{i/N}^{(h_N)}\bigr),
 \label{eq:triangular-increments}\\
 \mathcal Q_N(t)
 &=\frac1{\sqrt N}\sum_{i<\lfloor Nt\rfloor}(Y_{i,N}^2-1)
 =N^{2h_N-1/2}V_N^{(h_N)}(t),
 \label{eq:central-qv}\\
 \lambda_N&=h_NN^{1/2-\theta}.
\end{align}
Let
\begin{equation}
 \overline C_{q,\theta}
 =\frac q{B((1-\theta)/2,\theta)}
 \frac{\Lambda_{(q-1)\theta}(q\theta/2-1)}
 {2B(q\theta/2,1-q\theta)}.
 \label{eq:crossover-constant}
\end{equation}
Proposition~\ref{prop:boundary-residue} gives
\begin{equation}
 C_{\boldsymbol\gamma,h}
 =h\overline C_{q,\theta}+O(h^2).
 \label{eq:active-linear-collapse}
\end{equation}

\begin{proof}[Proof of Theorem~\ref{thm:crossover}]
The product formula gives the exact orthogonal decomposition
\begin{equation}
 \mathcal Q_N=\mathcal A_N+\mathcal B_N,
 \label{eq:active-bulk}
\end{equation}
where $\mathcal A_N$ is the complete second-chaos contraction and
$\mathcal B_N$ is the sum of the chaoses $4,6,\ldots,2q$.
Proposition~\ref{prop:uniform-active} gives, locally uniformly in time,
\begin{equation}
 \left\|\mathcal A_N-\lambda_N\overline C_{q,\theta}
 I_2(K^{(\gamma)})\right\|_2
 \le C_T\bigl(\sqrt{h_N}+\lambda_Nh_N
                    +\lambda_NN^{-\kappa_0}\bigr)
 \label{eq:active-triangular-bound}
\end{equation}
for some $\kappa_0>0$.

Propositions~\ref{prop:bulk-covariance} and
\ref{prop:bulk-fourcopy} show that every finite linear combination of
$\mathcal B_N$ has variance converging to the Brownian covariance $3$ and
all its proper self-contractions tend to zero.  Lemma
\ref{lem:contraction-CS-crossover} then makes every active--bulk contraction
tend to zero as well.  Proposition~\ref{prop:bulk-clt-crossover}, by the
multivariate Malliavin--Stein criterion, yields
\begin{equation}
 (\mathcal A_N,\mathcal B_N)
 \Longrightarrow
 \bigl(\lambda\overline C_{q,\theta}I_2(K^{(\gamma)}),\sqrt3B\bigr)
 \label{eq:joint-crossover}
\end{equation}
with independent coordinates whenever $\lambda_N\to\lambda<\infty$.
The block estimates in Proposition~\ref{prop:crossover-blocks} give joint
tightness and prove \eqref{eq:crossover-brownian}--
\eqref{eq:crossover-mixed}.

If $\lambda_N\to\infty$, \eqref{eq:active-triangular-bound} divided by
$\lambda_N$ gives the Rosenblatt finite-dimensional distributions, and
the active part of Proposition~\ref{prop:crossover-blocks} supplies
tightness.  The relative bulk block bound gives
$\lambda_N^{-1}\mathcal B_N\to0$.  This proves
\eqref{eq:crossover-rosenblatt}.
\end{proof}

The mixed regime differs from the classical $H=3/4$ threshold for the
quadratic variation of fractional Brownian motion.  There the critical
logarithmic normalization remains Gaussian; here a collapsing second-chaos
sector and a Gaussian bulk survive on the same scale, producing an
independent Brownian--Rosenblatt sum.

The phase parameter \eqref{eq:phase-parameter}, rather than
$h_N\log N$, is forced by the orthogonal second-chaos projection.  In
particular, if $\lambda_N\to\infty$, the central family
$\mathcal Q_N(t)$ is not tight at any fixed $t>0$; finite-chaos
hypercontractivity and Paley--Zygmund upgrade the divergent $L^2$ norm of
its active projection to non-tightness.

\section{Sharp normalization and paired-filter thresholds}
\label{sec:sharpness}

Several parts of the theory are optimal in a literal sense.

\begin{proposition}[Unique polynomial normalization]
\label{prop:normalization-sharp}
Fix the parameters of Theorem~\ref{thm:main} and $t>0$.  Put
$r_*=2h+\theta-1$.  Then, for every $r\in\R$,
\begin{equation}
 \begin{aligned}
 r<r_*&\Longrightarrow N^rV_N^{(h)}(t)\to0\quad\text{in }L^2,\\
 r=r_*&\Longrightarrow N^rV_N^{(h)}(t)\to Z(t)\ne0,\\
 r>r_*&\Longrightarrow \|N^rV_N^{(h)}(t)\|_2\to\infty.
 \end{aligned}
 \label{eq:normalization-trichotomy}
\end{equation}
Thus $2h+\theta-1$ is the unique polynomial normalization exponent with a
nondegenerate limit; the scale itself is unique up to a fixed nonzero
multiplicative constant.
\end{proposition}

\begin{proof}
This is the identity
$N^rV_N=N^{r-r_*}(N^{r_*}V_N)$ together with the nonzero $L^2$ limit in
Theorem~\ref{thm:main}.
\end{proof}

\begin{proposition}[Sharp paired-filter and path thresholds]
\label{prop:thresholds-sharp}
The paired convolution threshold
\begin{equation}
 a<\min\{1,2(1+\beta)\}
 \label{eq:filter-threshold-sharp}
\end{equation}
is necessary and sufficient.  Equality at either active boundary gives at
least a logarithmic divergence; at the corner $\beta=-1/2$, $a=1$, the
truncated divergence is of order $(\log(1/\varepsilon))^2$.  Moreover the path-space conclusions of
Theorem~\ref{thm:main} are sharp: the limit does not belong to
$c^{1-\theta}$.
\end{proposition}

\begin{proof}
The obstruction at $a=1$ is the one-dimensional Riesz singularity.  If
$\beta<-1/2$, \eqref{eq:qlocal-twosided} shows that the local integrand is
$|r|^{2\beta+1-a}$ and is integrable exactly under the second inequality
in \eqref{eq:filter-threshold-sharp}.  The logarithmic cases follow at
equality.  At $\beta=-1/2$, $a=1$, the integrand is comparable to
$|r|^{-1}\log(1/|r|)$, which gives the stated squared logarithm.

Let $H_R=1-\theta$.  Self-similarity gives
$t^{-H_R}Z(t)\overset{\rm law}=Z(1)$ with a nondegenerate right-hand side.
Membership in the little space $c^{H_R}$ would force the left-hand side to
converge to zero almost surely as $t\downarrow0$, a contradiction.
\end{proof}

The charged deficits
\eqref{eq:fixed-h-free-deficit}--\eqref{eq:fixed-h-pinned-deficit}
show that no additional anisotropic inequality is hidden in the closure:
every proper local face is strictly integrable for $h>0$.  When $h=0$, a
fully endpoint-pinned connected component can have zero deficit, and the
resulting logarithm is exactly the divergence identified in
Proposition~\ref{prop:boundary-residue}.  Thus the range $h>0$ in
Theorem~\ref{thm:main} is optimal for the unrenormalized process.

\section{Scope and conclusion}

The fixed-parameter theorem is anisotropic.  The simultaneous theorem is
currently isotropic, because isotropy removes every residual second-chaos
sector before the four-copy Malliavin analysis.  Extending the mixed
transition to arbitrary anisotropy requires a uniform treatment of those
additional second-chaos sectors; it is not a change of normalization.

The endpoint $h=0$ remains outside the bare process, but
Section~\ref{sec:zero-boundary} constructs and identifies its two
renormalized boundary selections.  Their common limit is cylindrical and
has no stochastically continuous, continuous, or c\`adl\`ag realization.
Theorem~\ref{thm:crossover} concerns the
different, simultaneous order of limits and quantifies the resulting
noncommutation.

In summary, the paper proves a full noncentral limit theorem for every
admissible anisotropy in the symmetrized monomial class and every fixed
$0<h\le1/2$, a universal
renormalized boundary field, and the complete isotropic
Gaussian--Rosenblatt phase diagram.  At fixed positive $h$, the
normalization sees the observed roughness while the self-similarity of the
limit is set by the hidden kernel.  At zero roughness, the path topology
disappears but the renormalized quadratic energy retains the Brownian
dynamics $\sqrt3B$.  Finally, the linear collapse
\[
 C_{\boldsymbol\gamma,h}
 =h\overline C_{q,\theta}+O(h^2)
\]
forces the crossover parameter $h_NN^{1/2-\theta}$ and joins the fixed,
boundary-first, and simultaneous limits into one phase diagram.

\clearpage
\appendix

\section*{Guide to the appendices}
This guide complements the conceptual strategy in
Section~\ref{sec:proof-strategy}.  The appendices contain the deferred parts
of the three main proofs; their logical outputs are as follows.
\begin{itemize}
 \item Appendix~\ref{app:riemann} proves weak convergence and supplies the
 exact norm and mixed scalar-product identities for the active kernel.
 Combined with Lemmas~\ref{lem:pairedrate} and \ref{lem:toeplitz}, these give
 Proposition~\ref{prop:active} and hence the leading term of
 Theorem~\ref{thm:main}.
 \item Appendix~\ref{app:stochasticdetails} converts the fixed-time
 finite-chaos estimates into Malliavin--Sobolev and little-H\"older
 convergence.  It closes the two strengthened convergence assertions of
 Theorem~\ref{thm:main}.
 \item Appendix~\ref{app:boundary-diagrams} isolates the unique logarithmic
 boundary root and shows that every proper contraction vanishes after
 normalization.  This is the Gaussianity and rate input for
 Theorem~\ref{thm:zero-main}.
 \item Appendices~\ref{app:root-regular}--\ref{app:bulk} are the three stages
 of Theorem~\ref{thm:crossover}: uniform endpoint-root control, convergence
 of the active second chaos, then the Gaussian bulk, asymptotic independence,
 and tightness.
\end{itemize}
\section{Positive-roughness active sector: displaced Riemann sums and scalar products}
\label{app:riemann}

This appendix closes the deterministic part of
Proposition~\ref{prop:active}.  Proposition~\ref{prop:displacedriemann}
proves the weak convergence of the active kernel while making explicit the
order of the limits in the mesh $N$ and in the tail radius.
Proposition~\ref{prop:activescalarproducts} then supplies the exact norm and
mixed scalar-product identities.  Combined with the Toeplitz estimates in
the main text, these two outputs yield strong $L^2$ convergence of the
active kernel, and therefore the Rosenblatt sector in
Theorem~\ref{thm:main}.

\begin{lemma}[Fractional test transform]
\label{lem:testtransform}
Let $-1<\gamma<0$, $\nu=\gamma+1$, and
\[
\mathcal A_f(s)=\int_\R(s-x)_+^\gamma f(x)\dd x,
\qquad f\in C_c^1(\R).
\]
Then
\begin{equation}
\|\mathcal A_f\|_\infty\le C_{\gamma,f},
\qquad
|\mathcal A_f(s)-\mathcal A_f(t)|
\le C_{\gamma,f}|s-t|^\nu.
\label{eq:testtransformholder}
\end{equation}
\end{lemma}

\begin{proof}
The boundedness on a fixed neighbourhood of $\operatorname{supp}f$ follows
from $\int_0^Lr^\gamma\dd r=L^{\gamma+1}/(\gamma+1)$.  To the right of
that neighbourhood, $(s-x)^\gamma\le C(1+s)^\gamma$ on the support of
$f$, while to the left the transform vanishes.  This proves the first
bound.

Let $r=|s-t|$ and assume $t=s+r$.  Extending the $x$-integral only
increases its absolute value, so
\begin{align*}
|\mathcal A_f(t)-\mathcal A_f(s)|
&\le\|f\|_\infty
\int_\R|(y+r)_+^\gamma-y_+^\gamma|\dd y\\
&=\|f\|_\infty r^{\gamma+1}
\int_\R|(u+1)_+^\gamma-u_+^\gamma|\dd u.
\end{align*}
The last integral is finite: the two endpoint singularities are integrable
because $\gamma>-1$, and at infinity the integrand is
$O(u^{\gamma-1})$.  This proves \eqref{eq:testtransformholder}.
\end{proof}

\begin{proposition}[Displaced and truncated Riemann sum]
\label{prop:displacedriemann}
Let $\gamma=\gamma_*$, $\nu=\gamma+1$, and let
$f,g\in C_c^1(\R)$.  For $R\ge2$ and $0\le t\le T$ put
\begin{align*}
\mathcal R_{N,R}(t)
={}&\int_{[-R,1]^2}\varphi_\beta(a)\varphi_\beta(b)|a-b|^{-c}\\
&\quad\times\frac1N\sum_{i< Nt}
\mathcal A_f\!\left(\frac{i+a}{N}\right)
\mathcal A_g\!\left(\frac{i+b}{N}\right)\dd a\dd b.
\end{align*}
Then
\begin{align}
&\sup_{0\le t\le T}\left|
\mathcal R_{N,R}(t)
-\Lambda_{c,R}(\beta)\int_0^t
\mathcal A_f(s)\mathcal A_g(s)\dd s\right|
\le C_{R,T,f,g}N^{-\nu},
\label{eq:compactdisplacedriemann}\\
&\sup_{N\ge1}\sup_{0\le t\le T}
\left|\mathcal R_N(t)-\mathcal R_{N,R}(t)\right|
\le C_{T,f,g}R^\beta,
\label{eq:displacedtail}\\
&|\Lambda_c(\beta)-\Lambda_{c,R}(\beta)|\le CR^\beta,
\label{eq:energymassTail}
\end{align}
where $\Lambda_{c,R}$ is the energy restricted to $[-R,1]^2$ and
$\mathcal R_N$ denotes the same integral over $\R^2$.  Consequently
\begin{equation}
\lim_{R\to\infty}\limsup_{N\to\infty}
\sup_{t\le T}\left|\mathcal R_N(t)
-\Lambda_c(\beta)\int_0^t
\mathcal A_f(s)\mathcal A_g(s)\dd s\right|=0.
\label{eq:orderedriemannlimits}
\end{equation}
\end{proposition}

\begin{proof}
Set $H(s)=\mathcal A_f(s)\mathcal A_g(s)$.  Lemma
\ref{lem:testtransform} gives $\|H\|_\infty\le C_{f,g}$ and
$|H(s)-H(t)|\le C_{f,g}|s-t|^\nu$.  For $a,b\in[-R,1]$, insert and subtract
$H(i/N)$ in the summand.  The product inequality
$|xy-x'y'|\le|x-x'||y|+|y-y'||x'|$ yields
\begin{align}
&\left|\mathcal A_f\!\left(\frac{i+a}{N}\right)
\mathcal A_g\!\left(\frac{i+b}{N}\right)-H(i/N)\right|
\le C_{f,g}\left(\frac{R+1}{N}\right)^\nu.
\label{eq:shifterror}
\end{align}
Also, with $n=\lfloor Nt\rfloor$,
\begin{align}
\left|\frac1N\sum_{i=0}^{n-1}H(i/N)-\int_0^tH(s)\dd s\right|
&\le\sum_{i=0}^{n-1}\int_{i/N}^{(i+1)/N}
|H(i/N)-H(s)|\dd s+\frac{\|H\|_\infty}{N}\notag\\
&\le C_{T,f,g}(N^{-\nu}+N^{-1}).
\label{eq:holderriemannerror}
\end{align}
Since $\nu<1$, the last display is $O(N^{-\nu})$.  Multiply
\eqref{eq:shifterror}--\eqref{eq:holderriemannerror} by the absolute
energy density and integrate over $[-R,1]^2$.  Its mass is finite by
Lemma~\ref{lem:roughconvolution} at $k=0$, proving
\eqref{eq:compactdisplacedriemann}.

Because $\varphi_\beta$ is supported on $(-\infty,1]$, the complement of
$[-R,1]^2$ is contained in $\{a<-R\}\cup\{b<-R\}$.  The summand is
bounded by $C_{T,f,g}$, uniformly in $N,i,a,b$.  Equation
\eqref{eq:uniformtail}, at $k=0$ and with exponent $c$, therefore gives
\begin{align*}
&\int_{\{a<-R\}\cup\{b<-R\}}
\rho_\beta(a)\rho_\beta(b)|a-b|^{-c}\dd a\dd b\\
&\hspace{30mm}\le 2C R^\beta.
\end{align*}
This proves \eqref{eq:displacedtail}.  The same estimate with the summand
replaced by $1$ proves \eqref{eq:energymassTail}.  Combining the three
bounds, first sending $N$ to infinity for fixed $R$ and then sending
$R$ to infinity, proves \eqref{eq:orderedriemannlimits}.
\end{proof}

\begin{proposition}[Exact active scalar products]
\label{prop:activescalarproducts}
For $t_N=\lfloor Nt\rfloor/N$, the identities
\eqref{eq:FNnorm} and \eqref{eq:FNcross} hold.  Moreover
\begin{equation}
\|F_t-F_s\|_2^2
=D_M^2\Lambda_c(\beta)^2b_{\gamma_*}^2
\int_s^t\!\int_s^t|u-v|^{-2\theta}\dd u\dd v
\le C_T|t-s|^{2-2\theta}.
\label{eq:limitkernelincrementexact}
\end{equation}
\end{proposition}

\begin{proof}
Write $s_{i,a}=(i+a)/N$.  After expanding the two copies of
$F_{N,t}$, Tonelli's theorem is justified by the absolute graph bounds.
Twice using
\[
\int_\R(s-x)_+^{\gamma_*}(r-x)_+^{\gamma_*}\dd x
=b_{\gamma_*}|s-r|^{-\theta}
\]
gives
\begin{align*}
\|F_{N,t}\|_2^2
={}&\frac{D_M^2b_{\gamma_*}^2}{N^2}
\sum_{i,j<Nt}\int_{\R^4}
\prod_{z\in\{a,b,a',b'\}}\varphi_\beta(z)\\
&\quad\times |a-b|^{-c}|a'-b'|^{-c}
|s_{i,a}-s_{j,a'}|^{-\theta}
|s_{i,b}-s_{j,b'}|^{-\theta}\dd a\dd b\dd a'\dd b'\\
={}&D_M^2b_{\gamma_*}^2N^{2\theta-2}
\sum_{i,j<Nt}\Gamma_*(j-i),
\end{align*}
which is \eqref{eq:FNnorm}.

Next $F_{t_N}=D_M\Lambda_c(\beta)K_{t_N}^{(\gamma_*)}$.
Taking its scalar product with $F_{N,t_N}$ and integrating first in the two
spatial variables gives
\begin{align*}
\langle F_{N,t_N},F_{t_N}\rangle
={}&\frac{D_M^2\Lambda_c(\beta)b_{\gamma_*}^2}{N}
\sum_{i<Nt}\int_0^{t_N}\int_{\R^2}
\varphi_\beta(a)\varphi_\beta(b)|a-b|^{-c}\\
&\quad\times|s_{i,a}-s|^{-\theta}|s_{i,b}-s|^{-\theta}
\dd a\dd b\dd s.
\end{align*}
Set $u=Ns$.  The two Riesz factors contribute $N^{2\theta}$ and
$\dd s=N^{-1}\dd u$, proving \eqref{eq:FNcross}.

Finally, $F_t-F_s=D_M\Lambda_c(\beta)(K_t-K_s)$, and
$K_t-K_s$ is the defining temporal integral restricted to $[s,t]$.
Applying the beta identity to its squared norm proves the equality in
\eqref{eq:limitkernelincrementexact}.  The double integral equals
$2|t-s|^{2-2\theta}/[(1-2\theta)(2-2\theta)]$, proving the bound.
\end{proof}

\section{Positive-roughness stochastic upgrades: finite chaos and little H\"older regularity}
\label{app:stochasticdetails}

The active and residual kernel estimates prove the fixed-time $L^2$
convergence in Theorem~\ref{thm:main}.  This appendix supplies exactly the
two functional-analytic upgrades used afterwards.  Proposition
\ref{prop:finitechaosmalliavin} transfers $L^2$ bounds inside a fixed finite
sum of chaoses to every stated Malliavin--Sobolev norm, and Proposition
\ref{prop:holderupgrade} combines those bounds with the block estimate to
give convergence in the little H\"older topology.  No additional diagram
estimate is required here.

\begin{proposition}[Finite-chaos Malliavin norms]
\label{prop:finitechaosmalliavin}
Let $Q\in\mathbb N$, $k\in\mathbb N$, and $p>1$.  There is a constant
$C_{Q,k,p}$ such that every $F\in\bigoplus_{r=0}^Q\mathcal C_r$ satisfies
\begin{equation}
\|F\|_{\mathbb D^{k,p}}\le C_{Q,k,p}\|F\|_2.
\label{eq:finitechaosmalliavinfull}
\end{equation}
If $F_n\to F$ in $L^2$ and all variables belong to the same finite sum,
then $F_n\to F$ in every $\mathbb D^{k,p}$.
\end{proposition}

\begin{proof}
For $F=\sum_{r=0}^QI_r(f_r)$, combine orthogonality, the derivative
identity
$\E\|D^jI_r(f_r)\|_2^2=r!\E|I_r(f_r)|^2/(r-j)!$, and Hilbert-valued
hypercontractivity, then sum over the finite set $r\le Q$, $j\le k$.
This proves \eqref{eq:finitechaosmalliavinfull} for $p\ge2$; monotonicity
of probability norms gives $1<p<2$.  Apply the estimate to $F_n-F$.
\end{proof}

\begin{proposition}[Quantified H\"older upgrade]
\label{prop:holderupgrade}
Let $B$ be a separable Hilbert space and let $X_N:[0,T]\to B$ be continuous
random functions.  Suppose that for every $r\ge2$ there is $C_r$ such that
\begin{equation}
\sup_N\E\|X_N(0)\|_B^r\le C_r,
\qquad
\sup_N\E\|X_N(t)-X_N(s)\|_B^r
\le C_r|t-s|^{r\chi},
\qquad s,t\in[0,T],
\label{eq:holderupgradeassumption}
\end{equation}
for some $0<\chi\le1$.  Then, for every
$\eta'<\chi-1/r$,
\begin{equation}
\sup_N\E\|X_N\|_{C^{\eta'}([0,T];B)}^r<\infty.
\label{eq:holderupgrademoment}
\end{equation}
Assume in addition that $X_N(t)\to X(t)$ in $L^p(\Omega;B)$ on a dense
subset containing $0$, that $X$ obeys the same increment bound and
$\E\|X(0)\|_B^r<\infty$, and choose
$0<\eta<\eta'<\chi-1/r$ and $r>p$.  Then
\begin{equation}
X_N\longrightarrow X
\quad\hbox{in }L^p(\Omega;C^\eta([0,T];B)).
\label{eq:holderupgradeconvergence}
\end{equation}
\end{proposition}

\begin{proof}
After rescaling take $T=1$.  For the dyadic increments
$\Delta_{m,j}X_N$, \eqref{eq:holderupgradeassumption} gives
\[
 \E\left[2^{m\eta'}\max_j\|\Delta_{m,j}X_N\|_B\right]^r
 \le C_r2^{-m[r(\chi-\eta')-1]}.
\]
Minkowski's inequality and dyadic chaining prove
\eqref{eq:holderupgrademoment}.  For $Y_N=X_N-X$, convergence on each
fixed dyadic grid and the uniform $C^{\eta'}$ moment first give
$\|Y_N\|_\infty\to0$ in $L^p$.  The deterministic interpolation bound
$\|y\|_{C^\eta}\le
C\|y\|_\infty^{1-\eta/\eta'}\|y\|_{C^{\eta'}}^{\eta/\eta'}$ and
H\"older's inequality then prove \eqref{eq:holderupgradeconvergence}.
\end{proof}

\section{Gaussian boundary selection: charged diagrams and the unique critical root}
\label{app:boundary-diagrams}

This appendix provides the Gaussianity mechanism missing from the covariance
calculation in Section~\ref{sec:zero-boundary}.  Lemma
\ref{lem:one-critical-root} proves that a connected boundary diagram has
only one potentially critical face, producing at most the pole $h^{-1}$ or
the cutoff logarithm.  Lemma~\ref{lem:boundary-fourcopy}, combined with the
residue in Proposition~\ref{prop:boundary-residue}, then shows that every
proper contraction of a normalized finite linear combination vanishes.
These are precisely the two inputs used for both regularizations in the
proof of Theorem~\ref{thm:zero-main}, including its Wasserstein rates.

Use the port charges
$c_j=\theta_j/2$ of Lemma~\ref{lem:charged-collective-closure}; a spatial
contraction of ports $j,k$ produces the time edge of weight $c_j+c_k$.

Let $D$ be a finite closed diagram on a vertex set $V$: at every vertex
there are exactly the $q$ ports labelled $1,\ldots,q$, and every port is
paired exactly once with a port at a different vertex.  For
$S\subset V$, let $w_D(S)$ be the sum of the weights of the edges internal
to $S$, and let $B_D(S)$ be the total charge of the ports in $S$ paired to
ports in $S^c$.

Identity \eqref{eq:fixed-h-charged-cut} applies without using a block map.
If $D$ is connected and $S$ is nonempty and proper, then $B_D(S)>0$ by
connectivity and positivity of all port charges.

For $\beta_h=d/2-1+h$, put
\[
 \rho_{h,t}(u)=\left|(t-u)_+^{\beta_h}-(-u)_+^{\beta_h}\right|.
\]
For a fixed connected diagram $D$, define
\begin{equation}
 J_{D,h}(\mathbf t)=
 \int_{\R^V}\prod_{v\in V}\rho_{h,t_v}(u_v)
 \prod_{e=\{v,w\}}|u_v-u_w|^{-\lambda_e}\dd\mathbf u.
 \label{eq:boundary-diagram-integral}
\end{equation}
The cutoff version replaces $\rho_{h,t}$ by
$|\varphi_{\varepsilon,t}|$ from \eqref{eq:cutoff-filter}.

\begin{lemma}[One critical root]
\label{lem:one-critical-root}
Fix a connected $D$, a finite set $T$ of positive observation times, and
$0<h_0<\min\{1/2,1-d/2\}$.  Then
\begin{equation}
 J_{D,h}(\mathbf t)\le\frac{C_{D,T,h_0}}{h},
 \qquad 0<h\le h_0,
 \label{eq:boundary-one-pole}
\end{equation}
and
\begin{equation}
 J_{D,\varepsilon}^{\rm cut}(\mathbf t)
 \le C_{D,T,h_0}\bigl(1+\log(1/\varepsilon)\bigr).
 \label{eq:boundary-one-log}
\end{equation}
The constants are uniform over the chosen finite set of times.
\end{lemma}

\begin{proof}
Put $n=|V|$ and
\[
 \kappa_h=1-\frac d2-h,
 \qquad
 \kappa_{\partial}=1-\frac d2.
\]
We first record the precise dyadic estimate that will also be used at
infinity.  Fix choices $a_v\in\{0,t_v\}$ and consider the multiset
$\mathcal L_{\mathbf a}$ consisting of the edge forms
$L_e(\mathbf u)=u_v-u_w$, of weight $\lambda_e$, and the $n$ marked
forms $M_v(\mathbf u)=u_v-a_v$, of weight $\kappa_h$.  For a subfamily
$\mathcal G\subset\mathcal L_{\mathbf a}$, let $r(\mathcal G)$ be the
rank of its normals and put
\[
 \delta_h(\mathcal G)
 =r(\mathcal G)-\sum_{L\in\mathcal G}\operatorname{wt}_h(L).
\]
Decompose the graph of the selected edge forms into its nonempty connected
vertex sets $S$, adding a marked isolated vertex as a one-point component,
and let $p$ be the number of selected marks based in $S$.  If $p=0$, then
$|S|\ge2$, the rank on this component is $|S|-1$, and
\begin{equation}
 \delta_h(\mathcal G;S)
 =|S|\left(1-\frac d2\right)-1+B_D(S)+R_D(\mathcal G,S)
 \ge1-d>0,
 \label{eq:boundary-defect-unmarked}
\end{equation}
where
$R_D(\mathcal G,S)=w_D(S)-w_{\mathcal G}(S)\ge0$.  If $p\ge1$, one mark
together with the incidence vectors of a spanning tree has rank $|S|$,
and hence
\begin{equation}
 \delta_h(\mathcal G;S)
 =(|S|-p)\left(1-\frac d2\right)+ph
   +B_D(S)+R_D(\mathcal G,S).
 \label{eq:boundary-defect-marked}
\end{equation}
The ranks and defects add over the components.  It follows from
\eqref{eq:boundary-defect-unmarked}--\eqref{eq:boundary-defect-marked},
connectivity of $D$, and positivity of all edge weights that there exists
$\delta_*>0$, depending only on $D$ and $h_0$, such that
\begin{equation}
 \delta_h(\mathcal G)\ge\delta_*
 \quad\text{for every }\varnothing\ne\mathcal G
       \subsetneq\mathcal L_{\mathbf a},
 \qquad
 \delta_h(\mathcal L_{\mathbf a})=nh,
 \label{eq:boundary-proper-gap}
\end{equation}
uniformly for $0<h\le h_0$ and for all endpoint choices.  Indeed, if
$p<|S|$, the first term in \eqref{eq:boundary-defect-marked} gives a fixed
gap; if $p=|S|$ and $S\subsetneq V$, then $B_D(S)>0$; and if $p=|S|$,
$S=V$, but an edge is missing, then $R_D(\mathcal G,V)>0$.  Thus equality
can degenerate only for the full edge family with all $n$ marks, for which
the defect is exactly $nh$.

We now make the dyadic summation explicit.  On any cube of fixed side
length, decompose the region $|L|\le1$ into
$2^{-m_L-1}<|L|\le2^{-m_L}$, $m_L\ge0$; the parts $|L|>1$ are harmless.
For $k\ge1$ set
\[
 \mathcal G_k(\mathbf m)=\{L:m_L\ge k\}.
\]
Order the forms by decreasing $m_L$ and choose greedily a basis of their
normals.  The matroid greedy identity gives
\[
 \sum_{L\in\mathcal B}m_L
 =\sum_{k\ge1}r(\mathcal G_k(\mathbf m)).
\]
Using the independent forms as coordinates and completing them by fixed
linear coordinates on their common kernel shows, uniformly in all affine
translations, that the corresponding cell has volume at most
\[
 C2^{-\sum_{k\ge1}r(\mathcal G_k(\mathbf m))}.
\]
After inserting the singular weights, its contribution is therefore at
most
\begin{equation}
 C2^{-\sum_{k\ge1}\delta_h(\mathcal G_k(\mathbf m))}.
 \label{eq:boundary-cell-defect}
\end{equation}
Let $\ell=\min_{L\in\mathcal L_{\mathbf a}}m_L$.  For
$1\le k\le\ell$ the family $\mathcal G_k$ is the full family and
contributes $nh$ per level.  After level $\ell$, every nonempty
$\mathcal G_k$ is proper and has defect at least $\delta_*$.  Since the
decreasing sequence $(\mathcal G_k)$ has only finitely many distinct
plateaux, summing first the lengths of its proper plateaux and then $\ell$
gives
\begin{equation}
 \sum_{\mathbf m}
 2^{-\sum_k\delta_h(\mathcal G_k(\mathbf m))}
 \le C_D\sum_{\ell\ge0}2^{-nh\ell}
 \le\frac{C_D}{h}.
 \label{eq:boundary-one-series}
\end{equation}
This is the only geometric series with a vanishing defect; all proper
plateau sums are bounded uniformly.  For the hard cutoff, every marked
term is supported where $|M_v|\ge\varepsilon$.  Thus the common depth
$\ell$ is at most $C+\log_2(1/\varepsilon)$, whereas all proper plateau
sums are unchanged.  At $h=0$, \eqref{eq:boundary-cell-defect}
consequently gives
\begin{equation}
 C_D\bigl(1+\log(1/\varepsilon)\bigr).
 \label{eq:boundary-cutoff-series}
\end{equation}

We apply these estimates first on a fixed compact interval containing all
endpoints.  Since $\beta_h=-\kappa_h$,
\begin{equation}
 \rho_{h,t}(u)
 \le |t-u|^{-\kappa_h}\1_{\{u<t\}}
     +|u|^{-\kappa_h}\1_{\{u<0\}},
 \label{eq:boundary-endpoint-expansion}
\end{equation}
and
\begin{equation}
 |\varphi_{\varepsilon,t}(u)|
 \le |t-u|^{-\kappa_{\partial}}\1_{\{u\le t-\varepsilon\}}
     +|u|^{-\kappa_{\partial}}\1_{\{u\le-\varepsilon\}}.
 \label{eq:boundary-cutoff-expansion}
\end{equation}
Expanding the finite products in
\eqref{eq:boundary-endpoint-expansion} proves the $h$-regularized compact-region
bound after the half-line indicators are dropped.  In the cutoff case, we
replace each half-line indicator by the weaker condition
$\1_{\{|M_v|\ge\varepsilon\}}$ and do not discard this cutoff.  Then
\eqref{eq:boundary-one-series} and \eqref{eq:boundary-cutoff-series},
respectively, prove the required compact-region bounds.

It remains to show that the complement is uniformly bounded; here one
must retain the cancellation in each filter.  Put
$\beta_1=\beta_0+h_0<0$.  Since $T$ is finite, there is $R_T>1$ such
that, for every $t\in T$, $0<h\le h_0$, $0<\varepsilon<e^{-1}$, and
$u\le-R_T$, the mean-value theorem gives
\begin{equation}
 \rho_{h,t}(u)+|\varphi_{\varepsilon,t}(u)|
 \le C_T(1+|u|)^{\beta_1-1}
 =C_T(1+|u|)^{-a},
 \qquad
 a=2-\frac d2-h_0>1.
 \label{eq:boundary-tail-decay}
\end{equation}
Notice that \eqref{eq:boundary-tail-decay} is a tail estimate; it is not
asserted near the cutoff endpoints.

Choose a smooth partition $\chi_0+\chi_\infty=1$ on the support of the
filters, with $\chi_0$ compactly supported and
$\operatorname{supp}\chi_\infty\subset(-\infty,-R_T]$, and expand the
product over the nonempty set $A\subset V$ of vertices carrying
$\chi_\infty$.  Partition each real line into a bounded-overlap smooth
partition subordinate to intervals of uniformly bounded diameter, indexed
by $k_v\in\mathbb Z$.  For $v\in A$,
\eqref{eq:boundary-tail-decay} supplies on that cell the factor
$C(1+|k_v|)^{-a}$.  The indices belonging to $v\notin A$ range over a
fixed finite set.  Declare an edge near when $|k_v-k_w|\le6$.  Every far
edge is bounded by a constant on the cell and may be discarded.  For the
near edges, expand only the compact filters as in
\eqref{eq:boundary-endpoint-expansion} or
\eqref{eq:boundary-cutoff-expansion}.  The resulting affine family is a
proper subfamily of $\mathcal L_{\mathbf a}$, because every vertex of the
nonempty set $A$ lacks a marked form.  Hence
\eqref{eq:boundary-proper-gap}--\eqref{eq:boundary-cell-defect}, now with
no critical full plateau, bound its integral over the unit cell by a
constant independent of $h$, $\varepsilon$, and the indices.  Consequently
the whole cell is bounded by
\[
 C\prod_{v\in A}(1+|k_v|)^{-a}.
\]
Since $a>1$, summing independently over all tail indices is finite.  This
argument allows arbitrary numbers of separated escaping clusters and
therefore covers all compact--tail and tail--tail configurations.  The
finitely many nonempty choices of $A$ contribute $O(1)$ uniformly.
Combining this with \eqref{eq:boundary-one-series}--
\eqref{eq:boundary-cutoff-series} proves
\eqref{eq:boundary-one-pole} and \eqref{eq:boundary-one-log}.
\end{proof}

\begin{lemma}[Four-copy contraction at the boundary]
\label{lem:boundary-fourcopy}
Let $T$ be finite, $a_t\in\R$, and
\[
 f_h=\sum_{t\in T}a_tL_t^{(\beta_0+h)}.
\]
For every $1\le r\le q-1$,
\begin{equation}
 \|f_h\otimes_rf_h\|^2=O(h).
 \label{eq:boundary-fourcopy-analytic}
\end{equation}
If
\[
 k_\varepsilon=
 \frac{\sum_{t\in T}a_tK_{t,\varepsilon}}
 {\sqrt{2\mathcal B_0c_0\log(1/\varepsilon)}},
\]
then
\begin{equation}
 \|k_\varepsilon\otimes_rk_\varepsilon\|^2
 =O\bigl((\log(1/\varepsilon))^{-1}\bigr).
 \label{eq:boundary-fourcopy-cutoff}
\end{equation}
\end{lemma}

\begin{proof}
Expand the squared contraction norm into four kernel copies and then into
the finitely many exponent labels.  Spatial integration produces a closed
charged diagram on four temporal vertices.  It is connected: the $r$
contracted ports join the first two copies in each half, and the remaining
$q-r$ ports join the two halves when the norm is closed.  Lemma
\ref{lem:one-critical-root} bounds every unnormalized term by $C/h$ for the
$h$-regularization and by $C\log(1/\varepsilon)$ for the cutoff family.
Proposition~\ref{prop:boundary-residue} gives
$A_{\boldsymbol\gamma,\beta_0+h}^4=O(h^2)$, which proves
\eqref{eq:boundary-fourcopy-analytic}.  The inverse fourth power of the
displayed cutoff denominator is
$O((\log(1/\varepsilon))^{-2})$, which proves
\eqref{eq:boundary-fourcopy-cutoff}.
\end{proof}

\section{Crossover engine: endpoint roots and nested cuts}
\label{app:root-regular}

This is the common diagrammatic input for the simultaneous limit in
Theorem~\ref{thm:crossover}.  Appendix~\ref{app:boundary-diagrams} controls
fixed observation times, but it is not uniform over the translated endpoint
filters and growing time grids of the triangular array.  The lemma below
separates the unique resonant root, with its explicit $h^{-1}$ loss, from a
regular part that retains the full decay across every nested macroscopic
cut.  Appendix~\ref{app:uniform-active} uses the former to control the
active sector, while Appendix~\ref{app:bulk} uses both parts for covariance,
four-copy, and block estimates.

We now assume isotropy \eqref{eq:crossover-isotropic}.  Put
\[
 \beta_h=\frac d2-1+h,
 \qquad A_h=A_{\boldsymbol\gamma,\beta_h},
 \qquad
 \varphi_h(x)=(1-x)_+^{\beta_h}-(-x)_+^{\beta_h},
 \qquad
 \varpi_h=|\varphi_h|.
\]
The following lemma supplies the uniformity missing from a fixed-$h$
diagram estimate.

\begin{lemma}[Endpoint--tail decomposition and nested cuts]
\label{lem:root-regular}
Let $C=(V,E)$ be a connected loopless $q$-regular multigraph, let
$b:V\to\{1,\ldots,s\}$ attach its primitive vertices to macroscopic
blocks, and put
\begin{align}
 J_{C,h}(\mathbf i)
 =\int_{\R^V}&\prod_{v\in V}\varpi_h(x_v)
 \prod_{e=\{v,w\}\in E}
 |i_{b(v)}-i_{b(w)}+x_v-x_w|^{-\theta}\dd\mathbf x.
 \label{eq:translated-component}
\end{align}
Let $\mathfrak R_C(\mathbf i)$ be the event that there are
$z\in\mathbb Z$ and $\varepsilon_v\in\{0,1\}$ such that
\begin{equation}
 i_{b(v)}+\varepsilon_v=z
 \qquad(v\in V).
 \label{eq:root-sector-exact}
\end{equation}
Then
\begin{equation}
 J_{C,h}(\mathbf i)
 \le\frac{C_C}{h}\1_{\mathfrak R_C(\mathbf i)}
      +C_CW_C(\mathbf i),
 \label{eq:root-regular-raw}
\end{equation}
where $W_C$ is a finite sum of shifted graph weights
\begin{equation}
 W_C(\mathbf i)
 =\sum_\alpha\prod_{e=\{v,w\}\in E}
 \left\langle i_{b(v)}-i_{b(w)}+\delta_{\alpha,e}
 \right\rangle^{-\theta},
 \qquad |\delta_{\alpha,e}|\le2.
 \label{eq:regular-soft-weight}
\end{equation}
In particular, since $A_h\asymp\sqrt h$,
\begin{equation}
 A_h^{|V|}J_{C,h}(\mathbf i)
 \le C_C\left[
 h^{|V|/2-1}\1_{\mathfrak R_C(\mathbf i)}
 +h^{|V|/2}W_C(\mathbf i)\right].
 \label{eq:root-regular-normalized}
\end{equation}

The weight in \eqref{eq:regular-soft-weight} has simultaneous nested-cut
decay.  More precisely, in a nested clustering sector of the macroblocks,
assign an edge $e$ to the first scale $R_{\ell(e)}$ at which its endpoint
blocks separate.  Then
\begin{equation}
 W_C(\mathbf i)\le C_C\prod_{e\in E}R_{\ell(e)}^{-\theta}.
 \label{eq:nested-cut-weight}
\end{equation}
Thus every partition $\mathcal P$ carries the exact conormal weight
$\theta e_C(\mathcal P)$, where $e_C(\mathcal P)$ is the number of edges
crossing its cells.
\end{lemma}

\begin{proof}
Choose $0<h_0<\min\{1/2,1-d/2\}$, put
$\beta_1=d/2-1+h_0<0$, and set
$\tau=1-\beta_1>1$.  With one-sided cutoffs supported in intervals of
radius smaller than $1/4$ around $0$ and $1$, one has uniformly for
$0<h\le h_0$
\begin{equation}
 \varpi_h(x)\le C\left[
 \chi_0(x)|x|^{\beta_h}
 +\chi_1(x)|1-x|^{\beta_h}
 +\omega(x)\right],
 \qquad
 \omega(x)\le C\langle x\rangle^{-\tau}.
 \label{eq:endpoint-tail-splitting}
\end{equation}
Expand the product over $V$.

We first exclude hidden affine subcollisions.  Let $F$ be any selected
subfamily of edge forms on a vertex set $A$, and let $M\subset A$ be the
vertices at which an endpoint mark is selected.  On a connected component
$S$ of $(A,F)$ without a mark and containing at least one selected edge,
one has $|S|\ge2$, the affine rank is $|S|-1$, and
\begin{equation}
 \delta_0(F|_S,\varnothing)
 =|S|-1-\theta|F|_S
 \ge |S|\left(1-\frac d2\right)-1>0.
 \label{eq:root-unmarked-subdefect}
\end{equation}
Unmarked isolated vertices carry no selected form and are omitted.  On a
component carrying $m\ge1$ marks, its rank is $|S|$ and the exact defect
at $h\ge0$ is
\begin{equation}
 \delta_h(F|_S,M\cap S)
 =\frac\theta2\bigl(q|S|-2|F|_S\bigr)
 +(|S|-m)\left(1-\frac d2\right)+mh.
 \label{eq:root-marked-subdefect}
\end{equation}
All three terms are nonnegative.  At $h=0$, equality requires every vertex
of $S$ to be marked and every one of its $q|S|$ ports to be paired by $F$
inside $S$.  Since $C$ is connected, this can happen only for $S=V$ and
$F=E$.  In particular, the case $S=V$, $m=|V|$ with at least one omitted
edge has fixed reserve $\theta(|E|-|F|)\ge\theta$.  Thus the full root is
the only affine family whose defect can vanish at $h=0$.

Consider first a term in which every vertex chooses an endpoint
$\varepsilon_v\in\{0,1\}$.  If a set $S$ of variables approaches a common
assigned center, its full affine defect is
\begin{align}
 |S|(1+\beta_h)-\theta|E(S)|
 &=|S|h+\frac\theta2|\partial_CS|.
 \label{eq:isotropic-affine-defect}
\end{align}
Here we used $q\theta=d$ and
$q|S|=2|E(S)|+|\partial_CS|$.  Every proper cluster has reserve at least
$\theta/2$.  The whole connected component has defect $|V|h$, and this
face exists only if all integer centers $i_{b(v)}+\varepsilon_v$ coincide.
The dyadic argument in the proof of Lemma~\ref{lem:one-critical-root}
therefore gives precisely the first term of \eqref{eq:root-regular-raw}.
If the centers do not all coincide, every equal-center cluster is proper;
its local integral is uniform and every edge between distinct clusters
factors with its full distance power.  The radius-$1/4$ supports ensure
that distinct integer centers remain separated by at least $1/2$.

If at least one vertex chooses $\omega$, the corresponding endpoint mark
is absent and the full affine defect has a fixed positive reserve.  It
remains to check that integrating a tail does not spend the decay of more
than one macroscopic cut.  The required star estimate is the following.  If
$a_j>0$, $A=\sum_ja_j<1$, and $\omega(x)\le C\langle x\rangle^{-\tau}$
with $\tau>1$, then H\"older's inequality gives
\begin{align}
 \int_\R\omega(x)\prod_j|x-z_j|^{-a_j}\dd x
 &\le\prod_j\left(
   \int_\R\omega(x)|x-z_j|^{-A}\dd x
   \right)^{a_j/A}\notag\\
 &\le C\prod_j\langle z_j\rangle^{-a_j}.
 \label{eq:tail-star}
\end{align}
The last convolution bound follows by splitting into neighbourhoods of
$0$, of $z_j$, and their complement; it uses only $A<1$ and $\tau>1$.
At every primitive vertex the sum of the incident exponents is
$q\theta=d<1$, so \eqref{eq:tail-star} applies.  The same proof, with the
tail restricted to $x<-R$, gives the marked estimate
\begin{equation}
 \int_{x<-R}\omega(x)\prod_j|x-z_j|^{-a_j}\dd x
 \le CR^{\beta_1}\prod_j\langle z_j\rangle^{-a_j},
 \qquad R\ge2,
 \label{eq:tail-star-marked}
\end{equation}
where $\beta_1=1-\tau<0$.

We make the elimination invariant explicit.  Order the tail vertices as
$v_1,\ldots,v_m$.  After eliminating $v_1,\ldots,v_j$, associate with every
original edge $e=\{v,w\}$ exactly one factor of exponent $\theta$:
\begin{equation}
 \mathfrak f_e^{(j)}=
 \begin{cases}
 |\Delta_e+x_v-x_w|^{-\theta},
   &v,w\notin\{v_1,\ldots,v_j\},\\
 \langle\Delta_e+\sigma_e x_w\rangle^{-\theta},
   &v\in\{v_1,\ldots,v_j\},\ w\notin\{v_1,\ldots,v_j\},\\
 \langle\Delta_e\rangle^{-\theta},
   &v,w\in\{v_1,\ldots,v_j\},
 \end{cases}
 \label{eq:edge-elimination-invariant}
\end{equation}
where $\Delta_e$ is the corresponding macroscopic displacement plus a
bounded endpoint shift, and $\sigma_e\in\{-1,1\}$ depends only on the
orientation of $e$.  The inductive invariant is that the product of the
factors \eqref{eq:edge-elimination-invariant} contains every original edge
once and only once.

At the next tail vertex, dominate every incident bracket by the
corresponding Riesz factor and apply \eqref{eq:tail-star}.  The sum of the
incident exponents is $q\theta=d<1$, so the integration returns one
one-endpoint bracket for each incident edge and leaves every nonincident
factor unchanged.  This proves the invariant by induction: no edge exponent
is lost, split, or used twice.  After all tail vertices have been removed,
the endpoint supports turn the remaining factors into the shifted factors
in \eqref{eq:regular-soft-weight}; the affine-defect argument then
integrates the remaining proper endpoint clusters uniformly.

A nested sector is described by a rooted cluster tree on the macroblocks.
Set $R_{\{a\}}=1$ at every leaf and assign dyadic radii $R_A\ge4$ to the
internal clusters, increasing towards the root, such that
$\langle i_a-i_b\rangle\asymp R_{A(a,b)}$, where $A(a,b)$ is the least
common ancestor of $a,b$.  Bounded shifts then give
\[
 \langle i_a-i_b+\delta\rangle^{-\theta}
 \le CR_{A(a,b)}^{-\theta}.
\]
Multiplying this inequality over the original edges assigns each edge
once, at its least common ancestor, and proves
\eqref{eq:nested-cut-weight}.  This completes the proof.
\end{proof}

\section{Crossover I: the uniform active kernel}
\label{app:uniform-active}

This appendix proves the active half of the decomposition
\eqref{eq:active-bulk}.  Proposition~\ref{prop:uniform-active} turns the
endpoint-root estimate of Appendix~\ref{app:root-regular} into the uniform
approximation \eqref{eq:active-triangular-bound}.  Together with the
boundary residue, it identifies the Rosenblatt contribution and forces the
phase parameter $h_NN^{1/2-\theta}$.  The fixed-$h$ estimate in
Proposition~\ref{prop:active} is insufficient here because its constants
need not remain bounded as $h\downarrow0$.

Retain isotropy and put
\begin{equation}
 b_\theta=B\left(\frac{1-\theta}{2},\theta\right),
 \qquad
 D_\theta=b_\theta^{q-1},
 \qquad
 \mathcal B_0=q!b_\theta^q.
 \label{eq:isotropic-beta-constants}
\end{equation}
Define
\begin{align}
 F_{N,t}^{(h)}(x,y)
 ={}&\frac{D_\theta}{N}\sum_{i<\lfloor Nt\rfloor}
 \int_{\R^2}\varphi_h(a)\varphi_h(b)|a-b|^{-(q-1)\theta}
 \notag\\
 &\quad\times
 \left(\frac{i+a}{N}-x\right)_+^\gamma
 \left(\frac{i+b}{N}-y\right)_+^\gamma\dd a\dd b,
 \label{eq:triangular-active-kernel}\\
 F_t^{(h)}
 ={}&D_\theta\Lambda_{(q-1)\theta}(\beta_h)
 K_t^{(\gamma)}.
 \label{eq:continuous-active-kernel}
\end{align}
The product formula gives the exact identity
\begin{equation}
 \mathcal A_N^{(h)}(t)
 =q^2(q-1)!A_h^2N^{1/2-\theta}I_2(F_{N,t}^{(h)}).
 \label{eq:active-exact-triangular}
\end{equation}

\begin{proposition}[Uniform triangular active kernel]
\label{prop:uniform-active}
There are $\kappa_0>0$, $h_0>0$, and $C_T<\infty$ such that, for
$0<h\le h_0$,
\begin{equation}
 \sup_{t\le T}\left\|
 \mathcal A_N^{(h)}(t)
 -hN^{1/2-\theta}\overline C_{q,\theta}
 I_2(K_t^{(\gamma)})\right\|_2
 \le C_T\left(
 \sqrt h+h^2N^{1/2-\theta}
 +hN^{1/2-\theta}N^{-\kappa_0}
 \right).
 \label{eq:uniform-active-result}
\end{equation}
One may take every
\begin{equation}
 0<\kappa_0<\frac12\min\left\{
 1-2\theta,
 \frac{1-d/2}{2-d/2}
 \right\}.
 \label{eq:uniform-active-kappa}
\end{equation}
\end{proposition}

\begin{proof}
Fix $\kappa_0$ in \eqref{eq:uniform-active-kappa} and put
\[
 \delta_{\rm act}=\frac{1-d/2}{2-d/2}.
\]
Choose $h_0>0$ sufficiently small that
$\beta_1=d/2-1+h_0<0$ and
\[
 \delta_1:=-\frac{\beta_1}{1-\beta_1}>2\kappa_0.
\]
This is possible because $\delta_1\to\delta_{\rm act}$ as
$h_0\downarrow0$.

The coefficient expansion follows directly from
Proposition~\ref{prop:boundary-residue}:
\begin{equation}
 \Lambda_d(\beta_h)=\frac{c_0}{h}+O(1),
 \qquad
 \Lambda_{(q-1)\theta}(\beta_h)
 =\Lambda_{(q-1)\theta}(\beta_0)+O(h),
 \label{eq:active-energy-expansions}
\end{equation}
and hence
\begin{equation}
 \frac q{b_\theta}
 \frac{\Lambda_{(q-1)\theta}(\beta_h)}
      {\Lambda_d(\beta_h)}
 =h\overline C_{q,\theta}+O(h^2).
 \label{eq:active-coefficient-expansion}
\end{equation}

It remains to make the deterministic Riemann approximation uniform.  After
contracting the two free variables in the squared norm of
$F_{N,t}^{(h)}$, the stationary correlation is
\begin{align}
 \Gamma_h(k)=\int_{\R^4}&
 \prod_{z\in\{a,b,a',b'\}}\varphi_h(z)
 |a-b|^{-(q-1)\theta}|a'-b'|^{-(q-1)\theta}
 \notag\\
 &\times|k+a-a'|^{-\theta}|k+b-b'|^{-\theta}
 \dd a\dd b\dd a'\dd b'.
 \label{eq:active-four-filter-correlation}
\end{align}
At $k\in\{-1,0,1\}$, all four filter endpoints can coincide.  The full
radial surplus is
\[
 4+4\beta_h-2(q-1)\theta-2\theta=4h.
\]
Lemma~\ref{lem:root-regular} therefore gives
\begin{equation}
 |\Gamma_h(k)|\le C/h,
 \qquad k\in\{-1,0,1\}.
 \label{eq:active-resonant-bound}
\end{equation}
This loss is real: on a cone at one common endpoint, all factors have fixed
sign and the radial integral is
$\int_0^\eta r^{4h-1}\dd r\asymp h^{-1}$.

For $|k|\ge2$, truncate every filter tail at
$R=|k|^{1/(1-\beta_1)}$.  Estimate \eqref{eq:tail-star-marked} makes the
discarded part $O(R^{\beta_1})$, while first-order expansion of the two
long edges on the retained part costs $O(R/|k|)$.  The two errors balance,
with
\[
 \delta_1=-\frac{\beta_1}{1-\beta_1}>0.
\]
The nested-cut part of Lemma~\ref{lem:root-regular} then gives, uniformly
in $h$,
\begin{equation}
 \left|\Gamma_h(k)
 -\Lambda_{(q-1)\theta}(\beta_h)^2|k|^{-2\theta}\right|
 \le C|k|^{-2\theta-\delta_1}.
 \label{eq:active-regular-tail}
\end{equation}
Define the mixed discrete--continuous correlation entering
$\langle F_{N,t}^{(h)},F_t^{(h)}\rangle$ by
\begin{equation}
 J_h(z)=\int_{\R^2}\varphi_h(a)\varphi_h(b)
 |a-b|^{-(q-1)\theta}|z+a|^{-\theta}|z+b|^{-\theta}\dd a\dd b.
 \label{eq:active-mixed-correlation}
\end{equation}
Its only full local collision has defect $1-\theta+2h>0$ in the variables
$(a,b,z)$; all proper faces have the fixed reserves from
Lemma~\ref{lem:root-regular}.  Consequently
\begin{equation}
 \sup_{0<h\le h_0}\int_{-2}^2|J_h(z)|\dd z<\infty,
 \qquad
 \left|J_h(z)-\Lambda_{(q-1)\theta}(\beta_h)|z|^{-2\theta}\right|
 \le C|z|^{-2\theta-\delta_1}
 \label{eq:active-mixed-uniform}
\end{equation}
for $|z|\ge2$.  Thus the mixed chart has no $h^{-1}$ pole.

Put
\[
 a_*=2\theta,
 \qquad
 L_h=\Lambda_{(q-1)\theta}(\beta_h),
 \qquad
 c_\theta=D_\theta^2b_\theta^2,
\]
and, for $n=\lfloor Nt\rfloor$, put $t_N=n/N$ and
\[
 I_{a_*}(s)=\int_0^s\!\int_0^s|u-v|^{-a_*}\dd u\dd v.
\]
Since $a_*=2\theta<1$, two spatial beta contractions give the exact
identities
\begin{align}
 \|F_{N,t}^{(h)}\|_2^2
 &=c_\theta N^{a_*-2}
   \sum_{i,j<n}\Gamma_h(i-j),
 \label{eq:uniform-FN-norm}\\
 \langle F_{N,t}^{(h)},F_{t_N}^{(h)}\rangle
 &=c_\theta L_hN^{a_*-2}
   \sum_{i<n}\int_0^nJ_h(i-u)\dd u,
 \label{eq:uniform-FN-mixed}\\
 \|F_{t_N}^{(h)}\|_2^2
 &=c_\theta L_h^2I_{a_*}(t_N).
 \label{eq:uniform-F-limit-norm}
\end{align}
Exchanging $(a,b)$ with $(a',b')$ gives
$\Gamma_h(-k)=\Gamma_h(k)$.  Also set $J_h^\vee(z)=J_h(-z)$.  Then
$J_h(i-u)=J_h^\vee(u-i)$, and the bounds in
\eqref{eq:active-mixed-uniform} hold for $J_h^\vee$ with the same
constants.

To isolate the three resonant lags, define the even sequence
\[
 \widetilde\Gamma_h(k)=
 \begin{cases}
  \Gamma_h(k),&|k|\ge2,\\
  L_h^2,&|k|=1,\\
  0,&k=0.
 \end{cases}
\]
The quantities $L_h$ are uniformly bounded for $0<h\le h_0$.
Equation~\eqref{eq:active-regular-tail} and the choice at $|k|=1$ show
that Lemma~\ref{lem:toeplitz}, with $a=a_*$, $\delta=\delta_1$, and
$\Lambda=L_h^2$, applies to $\widetilde\Gamma_h$ with a constant
independent of $h$.  On the other hand,
\begin{align}
 &N^{a_*-2}
 \sum_{i,j<n}|\Gamma_h(i-j)-\widetilde\Gamma_h(i-j)|\notag\\
 &\qquad\le C_TN^{a_*-2}\frac Nh
 =C_Th^{-1}N^{-(1-a_*)}
 \label{eq:uniform-resonant-remainder}
\end{align}
by \eqref{eq:active-resonant-bound}.  Hence
\begin{align}
 \left|
 N^{a_*-2}\sum_{i,j<n}\Gamma_h(i-j)
 -L_h^2I_{a_*}(t_N)
 \right|
 \le C_T\left\{
 h^{-1}N^{-(1-a_*)}+\mathfrak r_N(a_*,\delta_1)
 \right\}.
 \label{eq:uniform-discrete-comparison}
\end{align}

For the mixed term, apply the mixed part of Lemma~\ref{lem:toeplitz} to
$J_h^\vee$ at time $t_N$.  The local estimate and the two-sided tail
estimate in \eqref{eq:active-mixed-uniform} are exactly its hypotheses.
Therefore
\begin{align}
 \left|
 N^{a_*-2}\sum_{i<n}\int_0^nJ_h(i-u)\dd u
 -L_hI_{a_*}(t_N)
 \right|
 \le C_T\mathfrak r_N(a_*,\delta_1).
 \label{eq:uniform-mixed-comparison}
\end{align}
Here the constants are again uniform in $h$.

The choice of $\kappa_0$ and $h_0$ gives the two strict inequalities
\[
 2\kappa_0<1-a_*=1-2\theta,
 \qquad
 2\kappa_0<\delta_1.
\]
Consequently, including the logarithmic threshold
$\delta_1=1-a_*$ in \eqref{eq:rN},
\[
 \mathfrak r_N(a_*,\delta_1)
 \le C_{\kappa_0}N^{-2\kappa_0}.
\]
Combining \eqref{eq:uniform-FN-norm}--
\eqref{eq:uniform-F-limit-norm} with
\eqref{eq:uniform-discrete-comparison}--
\eqref{eq:uniform-mixed-comparison} yields
\[
 \|F_{N,t}^{(h)}-F_{t_N}^{(h)}\|_2^2
 \le C_T\left(
 h^{-1}N^{-(1-2\theta)}+N^{-2\kappa_0}
 \right).
\]
Finally, $F_{N,t}^{(h)}=F_{N,t_N}^{(h)}$, while the beta identity and
stationary increments give
\[
 \|F_t^{(h)}-F_{t_N}^{(h)}\|_2^2
 \le C_T|t-t_N|^{2-2\theta}
 \le C_TN^{-(2-2\theta)}.
\]
Since $2-2\theta>1-2\theta>2\kappa_0$, we obtain, uniformly for $t\le T$,
\begin{equation}
 \|F_{N,t}^{(h)}-F_t^{(h)}\|_2^2
 \le C_T\left(
 h^{-1}N^{-(1-2\theta)}+N^{-2\kappa_0}
 \right),
 \label{eq:active-deterministic-error}
\end{equation}
as claimed.

Take square roots in \eqref{eq:active-deterministic-error}, multiply by
the coefficient $A_h^2N^{1/2-\theta}=O(hN^{1/2-\theta})$ in
\eqref{eq:active-exact-triangular}, and use
\eqref{eq:active-coefficient-expansion}.  The resonant term becomes
$O(\sqrt h)$, the regular term becomes
$O(hN^{1/2-\theta}N^{-\kappa_0})$, and the coefficient error is
$O(h^2N^{1/2-\theta})$.  This proves the proposition.
\end{proof}

\section{Crossover II: the Gaussian bulk, independence, and tightness}
\label{app:bulk}

This appendix proves the complementary half of
Theorem~\ref{thm:crossover} after the active second chaos has been removed.
Proposition~\ref{prop:bulk-covariance} identifies the limiting covariance
$3(s\wedge t)$.  Lemma~\ref{lem:four-block-grammar} and the partition-face
bound lead to Proposition~\ref{prop:bulk-fourcopy}, which kills every
proper contraction; Proposition~\ref{prop:bulk-clt-crossover} then gives
the Gaussian bulk and its asymptotic independence from the active sector.
Finally, Proposition~\ref{prop:crossover-blocks} supplies the tightness
bounds in all three regimes.  These outputs are invoked, in this order, in
the proof of Theorem~\ref{thm:crossover}.

Let $f_i^{(h)}$ be the normalized unit-increment kernel, so that
$Y_i^{(h)}=I_q(f_i^{(h)})$ and $q!\|f_i^{(h)}\|^2=1$.  Put
\begin{equation}
 H_{i,r}^{(h)}
 =f_i^{(h)}\widetilde\otimes_{q-r}f_i^{(h)},
 \qquad
 a_{q,r}=(q-r)!\binom qr^2.
 \label{eq:bulk-product-kernels}
\end{equation}
Then
\begin{equation}
 (Y_i^{(h)})^2-1
 =\sum_{r=1}^qa_{q,r}I_{2r}(H_{i,r}^{(h)}).
 \label{eq:bulk-product-formula}
\end{equation}
In the isotropic model the $r=1$ term is exactly the local active sector.
Let $\xi_{i,h}^{\rm act}$ denote it and let
\begin{equation}
 \xi_{i,h}^{\rm bulk}
 =(Y_i^{(h)})^2-1-\xi_{i,h}^{\rm act}.
 \label{eq:local-bulk}
\end{equation}

The increment correlation is
\begin{equation}
 \rho_h(k)=\frac12\left(
 |k+1|^{2h}+|k-1|^{2h}-2|k|^{2h}\right).
 \label{eq:triangular-increment-correlation}
\end{equation}
For $0<h<1/2$, concavity and telescoping give
\begin{equation}
 \|\rho_h\|_{\ell^1(\mathbb Z)}=2,
 \qquad
 \|\rho_h-\rho_0\|_{\ell^1(\mathbb Z)}
 =2(2^{2h}-1)=O(h),
 \label{eq:rho-l1}
\end{equation}
where $\rho_0(0)=1$, $\rho_0(\pm1)=-1/2$, and $\rho_0(k)=0$ otherwise.

\begin{proposition}[Exact two-copy bulk envelope]
\label{prop:bulk-covariance}
For $0<h\le h_0$ and $k\in\mathbb Z$,
\begin{align}
 |\Cov(\xi_{0,h}^{\rm act},\xi_{k,h}^{\rm act})|
 &\le Ch\1_{\{|k|\le1\}}
 +Ch^2\langle k\rangle^{-2\theta},
 \label{eq:active-covariance-envelope}\\
 \left|\Cov(\xi_{0,h}^{\rm bulk},\xi_{k,h}^{\rm bulk})
 -2\rho_h(k)^2\right|
 &\le Ch\1_{\{|k|\le1\}}
 +Ch^2\sum_{r=2}^q\langle k\rangle^{-2r\theta}.
 \label{eq:bulk-covariance-envelope}
\end{align}
Consequently, for $s,t\ge0$,
\begin{equation}
 \frac1N\sum_{i<Ns}\sum_{j<Nt}
 \Cov(\xi_{i,h_N}^{\rm bulk},\xi_{j,h_N}^{\rm bulk})
 \longrightarrow3(s\wedge t)
 \label{eq:bulk-brownian-covariance}
\end{equation}
whenever $h_N\downarrow0$ and
$h_NN^{1/2-\theta}=O(1)$.
\end{proposition}

\begin{proof}
Expand the centered fourth moment of
$(Y_0^{(h)},Y_k^{(h)})$ into primitive pairing diagrams.  The two complete
macro-pairing classes, after summing their internal label pairings, give
exactly $2\rho_h(k)^2$.  A component contained in a single square factors through
the product of the variances and is removed by centering.  Every remaining
diagram is a connected $q$-regular graph on the four primitive copies.

The cut between the two time blocks contains an even number $2r$ of edges.
For $r=1$ the diagram is precisely the covariance of the active
second-chaos sector: each square has $q-1$ internal edges and the two free
ports cross the cut.  After its removal, $r\ge2$.

Apply Lemma~\ref{lem:root-regular}.  A common endpoint is possible only for
$k\in\{-1,0,1\}$ and gives, after the factor $A_h^4=O(h^2)$, the local
bound $O(h)$.  Outside the root sector, the $2r$ crossing edges give
$O(h^2\langle k\rangle^{-2r\theta})$.  Summing the finite set of diagrams
proves \eqref{eq:active-covariance-envelope} and
\eqref{eq:bulk-covariance-envelope}.

The fragmented term converges to Brownian covariance by
\eqref{eq:rho-l1} and
\[
 2\sum_{k\in\mathbb Z}\rho_0(k)^2=3.
\]
The local remainder in \eqref{eq:bulk-covariance-envelope} has
$\ell^1$ norm $O(h_N)$.  For the slowest regular term, if $4\theta<1$,
\[
 \frac{h_N^2}{N}S_N(4\theta)
 \le C h_N^2N^{1-4\theta}=O(N^{-2\theta});
\]
the cases $4\theta=1$ and $4\theta>1$ are smaller.  This proves
\eqref{eq:bulk-brownian-covariance}.
\end{proof}

For $2\le r\le q$, set
\begin{equation}
 b_{r,N,t}=\frac{a_{q,r}}{\sqrt N}
 \sum_{i<\lfloor Nt\rfloor}H_{i,r}^{(h_N)},
 \qquad
 \mathcal B_N(t)=\sum_{r=2}^qI_{2r}(b_{r,N,t}).
 \label{eq:cumulative-bulk-kernels}
\end{equation}
For fixed times $t_j$ and coefficients $u_j$, write
$b_{r,N}[u]=\sum_ju_jb_{r,N,t_j}$ and put
\begin{equation}
 \Delta_N[u]=
 \max_{2\le r\le q}\max_{1\le p\le2r-1}
 \|b_{r,N}[u]\otimes_pb_{r,N}[u]\|^2.
 \label{eq:bulk-fourcopy-defect}
\end{equation}

\begin{lemma}[Exact four-block contraction grammar]
\label{lem:four-block-grammar}
Fix $2\le r\le q$ and $1\le p\le2r-1$.  Every labelled summand in
$\|b_{r,N}[u]\otimes_pb_{r,N}[u]\|^2$ is represented by a loopless
$q$-regular pairing multigraph $G$ on eight primitive vertices, two over
each macroblock $1,2,3,4$.  In each macroblock its two primitive vertices
are joined by exactly $q-r$ internal edges.  Let the external macroscopic
quotient $\overline G$ be obtained by deleting these internal edges and
then contracting the two vertices of each block.  There are exactly $p$
edges between blocks $1,2$, exactly
$p$ edges between blocks $3,4$, and every remaining edge joins
$\{1,2\}$ to $\{3,4\}$.  Consequently $\overline G$ is connected, every
macrodegree equals $2r$, and $\overline G$ is Eulerian.  Every primitive
connected component of $G$ meets at least two macroblocks, so $G$ has at
most four primitive components.

If all four primitive components have two vertices, then necessarily
$r=q$; every component consists of $q$ parallel edges joining two distinct
macroblocks, and the four component supports form a connected two-regular
graph on the four macroblocks, hence a four-cycle.  This is the completely
fragmented packet.  Every other labelled summand is estimated
componentwise by Lemma~\ref{lem:root-regular}; no further exceptional
component pattern occurs.
\end{lemma}

\begin{proof}
Each $H_{i,r}$ contains two primitive $q$-kernels joined by $q-r$ ports,
leaving $r$ free ports on each primitive vertex.  The contraction
$\otimes_p$ pairs exactly $p$ ports between the first two macroblocks; the
second copy in the squared norm gives the same $p$ edges between the last
two.  The norm closure pairs every remaining port of the first half with a
port of the second half.  Since $p\ge1$ and $2r-p\ge1$, the edges
$1$--$2$, $3$--$4$, and at least one edge between the two halves are
present; thus $\overline G$ is connected without any centering or
cancellation.  Each block has $2r$ external ports, proving the degree and
Eulerian assertions.  Each primitive vertex has $r\ge2$ external ports,
hence its component leaves its macroblock; eight vertices then give at
most four components.

Suppose all components have size two.  If $q-r>0$, the fixed internal edge
to the other primitive vertex in the same macroblock and any external edge
would put at least three vertices in that component, a contradiction.
Thus $r=q$.  A two-vertex $q$-regular loopless component has exactly $q$
parallel edges.  Each macroblock supplies two primitive vertices, hence
the graph of component supports is two-regular; its connectedness follows
from that of $\overline G$, so it is the four-cycle.  All remaining
patterns fall under the componentwise root/regular expansion.
\end{proof}

\begin{lemma}[Dyadic-chain reduction to partition faces]
\label{lem:dyadic-chain-face-reduction}
Let $Q=(\mathcal V,\mathcal E)$ be a finite loopless multigraph, connected
when $|\mathcal V|\ge2$.  For a partition $\mathcal P$ of
$\mathcal V$, write
\[
 k(\mathcal P)=|\mathcal P|,
 \qquad
 e_Q(\mathcal P)
 =\#\{e\in\mathcal E:\text{ the endpoints of $e$ lie in distinct
 cells of $\mathcal P$}\}.
\]
Fix $\alpha\ge0$, $\theta>0$, and define the face deficit
\begin{equation}
 D_\alpha(\mathcal P)
 =2+\alpha-k(\mathcal P)+\theta e_Q(\mathcal P).
 \label{eq:partition-face-deficit}
\end{equation}
Let $I_N\subset\mathbb Z$ be an interval of cardinality $O(N)$, and let
$(\delta_e)_{e\in\mathcal E}$ be uniformly bounded integer shifts.  Then
\begin{equation}
 N^{-2-\alpha}
 \sum_{\mathbf j\in I_N^{\mathcal V}}
 \prod_{e=\{v,w\}\in\mathcal E}
 \left\langle j_v-j_w+\delta_e\right\rangle^{-\theta}
 \le C(\log(2+N))^{|\mathcal V|-1}
 N^{-\min_{\mathcal P}D_\alpha(\mathcal P)}.
 \label{eq:dyadic-chain-face-bound}
\end{equation}
The constant is uniform over the bounded shifts and over intervals $I_N$
of cardinality $O(N)$.
\end{lemma}

\begin{proof}
The assertion is immediate for bounded $N$.  If $|\mathcal V|=1$, then
$\mathcal E=\varnothing$ and both sides are $O(N^{-1-\alpha})$.  We may
therefore assume $N\ge2$ and $|\mathcal V|\ge2$.

Bounded shifts may first be removed, since
$\langle x+\delta\rangle\asymp\langle x\rangle$ uniformly for
$|\delta|\le C$.  Replace $N$, up to a fixed multiplicative constant, by
$2^J$.

A dyadic clustering sector is described by a strictly increasing chain of
partitions
\[
 \mathcal P_0\prec\mathcal P_1\prec\cdots
 \prec\mathcal P_s=\{\mathcal V\}
\]
and scales
\[
 0=n_0\le n_1\le\cdots\le n_s\le J.
\]
Here two cells of $\mathcal P_{j-1}$ first belong to the same cluster at
distance comparable to $2^{n_j}$.  Such sectors cover
$I_N^{\mathcal V}$ with bounded multiplicity.  Indeed, one may take the
connected components of the graph generated by the relations
$|j_v-j_w|\le2^n$, and then remove repetitions from the resulting chain.

Put
\[
 k_j=k(\mathcal P_j),
 \qquad
 e_j=e_Q(\mathcal P_j).
\]
There are $O(N)$ choices for one global translation.  At the merger
$\mathcal P_{j-1}\prec\mathcal P_j$, the relative locations of the
children contribute at most
\[
 C2^{n_j(k_{j-1}-k_j)}
\]
choices.  Exactly $e_{j-1}-e_j$ graph edges acquire their first nontrivial
scale at that merger, and they contribute
\[
 C2^{-\theta n_j(e_{j-1}-e_j)}.
\]
Consequently the contribution of the sector, after multiplication by
$N^{-2-\alpha}$, is at most $CN^{-d_{\rm ch}}$, where, with
$t_j=n_j/J$ and $F_j=k_j-\theta e_j$, one has
\begin{equation}
 d_{\rm ch}
 =1+\alpha-\sum_{j=1}^s(F_{j-1}-F_j)t_j.
 \label{eq:chain-deficit}
\end{equation}
Since $F_s=1$, summation by parts gives the exact identity
\begin{equation}
 d_{\rm ch}
 =\sum_{j=0}^{s-1}(t_{j+1}-t_j)
 D_\alpha(\mathcal P_j)
 +(1-t_s)D_\alpha(\mathcal P_s).
 \label{eq:chain-deficit-convex}
\end{equation}
The coefficients on the right-hand side are nonnegative and sum to one.
Thus
\[
 d_{\rm ch}\ge\min_{\mathcal P}D_\alpha(\mathcal P).
\]
There are only finitely many partition chains, and a chain has at most
$|\mathcal V|-1$ distinct scales.  The number of scale choices is therefore
$O((J+1)^{|\mathcal V|-1})$, which proves
\eqref{eq:dyadic-chain-face-bound}.
\end{proof}

\begin{lemma}[Root-support capacity on four macroblocks]
\label{lem:root-support-capacity}
Work in the setting of Lemma~\ref{lem:four-block-grammar}.  Select the root
term of Lemma~\ref{lem:root-regular} in exactly $R$ primitive connected
components, and let $\mathcal H_R$ be the hypergraph on the four macroblocks
whose hyperedges are the macroblock supports of the selected components.
Let $c_R$ be the number of connected components of $\mathcal H_R$, isolated
macroblocks included.  Then
\begin{equation}
 c_0=4,
 \qquad c_1\le3,
 \qquad c_2\le2,
 \qquad c_3\le2,
 \qquad c_4=1.
 \label{eq:root-support-capacity}
\end{equation}
After contracting every connected component of $\mathcal H_R$, the regular
macroscopic quotient $Q_R$ is connected and Eulerian.  Consequently, if a
partition of $Q_R$ has $k\ge2$ cells and $e$ crossing microscopic edges,
then $e\ge k$.  When $R=0$, the discrete four-cell partition satisfies
$e\ge8$.
\end{lemma}

\begin{proof}
Every primitive component meets at least two macroblocks, so every root
hyperedge has cardinality at least two.  Each macroblock contains only two
primitive vertices.  The assertions for $R=0$ and $R=1$ are immediate.  If
$R=2$ left three hypergraph components, the two roots would have to repeat
the same proper two-block support.  They would then consume both primitive
vertices over both supporting blocks and disconnect those blocks from the
remaining macroscopic quotient, contradicting Lemma~\ref{lem:four-block-grammar}.
Thus $c_2\le2$.  If $c_3=3$, the unique nontrivial support component would
again contain only two macroblocks, so all three hyperedges would have that
same support.  This would require three distinct primitive vertices over
each of the two blocks, although only two are available.  Hence $c_3\le2$.
If $R=4$, all primitive components have been selected; their support
hypergraph is the connected macroquotient, and therefore $c_4=1$.

Contracting root supports preserves connectedness.  It also preserves even
macroscopic degrees, because the original degree of every block is $2r$
and every contracted internal edge removes two incidences.  Hence $Q_R$ is
Eulerian.  Contracting the cells of any partition of $Q_R$ produces a
connected Eulerian multigraph on $k$ vertices.  Such a graph contains a
closed walk visiting all vertices and therefore has at least $k$ edges,
which proves $e\ge k$.  For $R=0$, the discrete partition cuts every edge of
the four-block quotient; its four degrees are $2r\ge4$, so
$e=\frac12\sum_{a=1}^4 2r=4r\ge8$.
\end{proof}

\begin{proposition}[Uniform four-copy bulk estimate]
\label{prop:bulk-fourcopy}
For every $L,T<\infty$ and every fixed finite linear combination $u$, there
are $C,M<\infty$ such that
\begin{equation}
 \Delta_N[u]
 \le C(\log N)^M\left(N^{-4\theta}+N^{-1/2}\right)
 \label{eq:uniform-fourcopy}
\end{equation}
whenever $h_NN^{1/2-\theta}\le L$.
\end{proposition}

\begin{proof}
We prove the root-sector count that complements the nested-cut estimate.
Fix a diagram from Lemma~\ref{lem:four-block-grammar} and select the root
term of Lemma~\ref{lem:root-regular} in exactly $R$ primitive components.
Since there are eight primitive vertices in total, its physical $h$ factor
is $h^{4-R}$.  Let $c_R$ be the number of connected components of the
resulting root-support hypergraph, isolated macroblocks included.
Lemma~\ref{lem:root-support-capacity} gives
\begin{equation}
 c_0=4,
 \qquad c_1\le3,
 \qquad c_2\le2,
 \qquad c_3\le2,
 \qquad c_4=1,
 \label{eq:root-hypergraph-components}
\end{equation}
and shows that the contracted regular quotient is connected Eulerian, with
$e\ge k$ on every partition into $k\ge2$ cells and $e\ge8$ on the discrete
four-cell face when $R=0$.

The root conditions identify the indices in each connected component of
the root-support hypergraph up to one of finitely many shifts in
$\{-1,0,1\}$.  Hence there are at most $c_R$ free translations.  Contract
these root supports and denote the resulting regular quotient by $Q_R$.

Put $a_*=1/2-\theta$.  Fix a choice of $R$ root components.  If the
corresponding root constraints are inconsistent, the term vanishes.
Otherwise, choose one free integer coordinate in each connected component
of the root-support hypergraph.  After expanding the finitely many regular
weights and deleting bounded loop factors, the resulting term is bounded
by
\begin{equation}
 C N^{-2}h^{4-R}
 \sum_{\mathbf j\in I_N^{\mathcal V_R}}
 \prod_{e=\{v,w\}\in\mathcal E_R}
 \left\langle j_v-j_w+\delta_e\right\rangle^{-\theta},
 \label{eq:root-sector-quotient-sum}
\end{equation}
where
\[
 |\mathcal V_R|=c_R,
 \qquad |\delta_e|\le C,
\]
and $Q_R=(\mathcal V_R,\mathcal E_R)$ is the contracted regular quotient.
As observed above, $Q_R$ is connected and every one of its cuts has
positive even size.

Since $h\le LN^{-a_*}$, set
\[
 \alpha_R=(4-R)a_*.
\]
Lemma~\ref{lem:dyadic-chain-face-reduction} applies to
\eqref{eq:root-sector-quotient-sum}.  For a partition $\mathcal P$ of
$\mathcal V_R$, with
\[
 k=k(\mathcal P),
 \qquad e=e_{Q_R}(\mathcal P),
\]
its face deficit is
\begin{align}
 D_{\alpha_R}(\mathcal P)
 &=2+(4-R)\left(\frac12-\theta\right)-k+\theta e \notag\\
 &=4-\frac R2-k+\theta(e-4+R)
 =:\delta_R(k,e).
 \label{eq:root-sector-deficit}
\end{align}
Moreover, identity \eqref{eq:chain-deficit-convex} shows that the exponent
of every multiscale chain is a convex combination of these face deficits.
It is therefore enough to check the admissible pairs $(k,e)$:
\begin{center}
\begin{tabular}{c|c|c}
$R$ & admissible faces & $\inf\delta_R$\\
\hline
$0$ & $(4,e\ge8),(3,e\ge3),(2,e\ge2),(1,0)$ &
$\min\{4\theta,1-\theta\}$\\
$1$ & $(3,e\ge3),(2,e\ge2),(1,0)$ & $1/2$\\
$2$ & $(2,e\ge2),(1,0)$ & $1$\\
$3$ & $(2,e\ge2),(1,0)$ & $1/2+\theta$\\
$4$ & $(1,0)$ & $1$
\end{tabular}
\end{center}
Indeed the five row minima are respectively
\begin{align*}
&\min\{4\theta,1-\theta,2-2\theta,3-4\theta\}
=\min\{4\theta,1-\theta\},\\
&\min\{1/2,3/2-\theta,5/2-3\theta\}=1/2,\\
&\min\{1,2-2\theta\}=1,\\
&\min\{1/2+\theta,3/2-\theta\}=1/2+\theta,\qquad 1,
\end{align*}
because $0<\theta<1/q\le1/2$.  Since $1-\theta>1/2$, every
nonfragmented bulk diagram gains at least $\min\{4\theta,1/2\}$.

The number of partition chains is finite and each chain contains at most
three dyadic scales, so their summation contributes only a fixed power of
$\log N$.  The only zero-defect active two-port sector corresponds to
$r=1$; it is absent here because $\mathcal B_N$ contains only the orders
$2\le r\le q$.
For the fragmented packet, let $n=\lfloor NT\rfloor$, set
$R_{h,N}=(\rho_h(i-j))_{0\le i,j<n}$, put
$w_i=\sum_j u_j\1_{\{i<\lfloor Nt_j\rfloor\}}$, and let
$D_u=\operatorname{diag}(w_i)$.  Each fragmented four-cycle is bounded by
\[
 \frac C{N^2}\left|\operatorname{Tr}\bigl((D_uR_{h,N})^4\bigr)\right|
 \le\frac{Cn}{N^2}\|D_uR_{h,N}\|_{\rm op}^4
 \le\frac C N\|u\|_1^4\|\rho_h\|_{\ell^1}^4=O(N^{-1})
\]
by \eqref{eq:rho-l1} and the Schur test.  The number of diagrams depends only on the fixed
$q$, so summation proves \eqref{eq:uniform-fourcopy}.
\end{proof}

We record the Hilbert-space implication of
Proposition~\ref{prop:bulk-fourcopy}.

\begin{lemma}[Contraction Cauchy--Schwarz]
\label{lem:contraction-CS-crossover}
For $f\in\mathcal H^{\odot m}$, $g\in\mathcal H^{\odot n}$, and
$0\le\ell\le m\wedge n$,
\begin{equation}
 \|f\otimes_\ell g\|^2
 =\langle f\otimes_{m-\ell}f,
          g\otimes_{n-\ell}g\rangle
 \le\|f\otimes_{m-\ell}f\|
      \|g\otimes_{n-\ell}g\|.
 \label{eq:contraction-CS-crossover}
\end{equation}
\end{lemma}

\begin{proof}
Write the contracted variables as $a,b$ in the squared norm, apply Fubini,
and integrate first the two copies of the free variables of $f$ and then
those of $g$.  This gives the identity; Cauchy--Schwarz gives the
inequality.
\end{proof}

\begin{proposition}[Bulk central limit and asymptotic independence]
\label{prop:bulk-clt-crossover}
If $h_NN^{1/2-\theta}=O(1)$, then
\begin{equation}
 \mathcal B_N\Longrightarrow\sqrt3B
 \label{eq:bulk-functional-fdd}
\end{equation}
in finite-dimensional distributions.  Jointly with the active sector,
\begin{equation}
 (\mathcal A_N,\mathcal B_N)
 \Longrightarrow
 (\lambda\overline C_{q,\theta}I_2(K^{(\gamma)}),\sqrt3B)
 \label{eq:active-bulk-independent}
\end{equation}
along every subsequence on which $\lambda_N\to\lambda<\infty$, and the two
limiting processes are independent.
\end{proposition}

\begin{proof}
For a finite linear combination $B_N[u]$, the derivative/product formula
expands
\[
 \left\langle DB_N[u],-DL^{-1}B_N[u]\right\rangle
 -\E[B_N[u]^2]
\]
into finitely many proper contractions between the kernels $b_{r,N}[u]$.
By orthogonality of the chaos orders and
Proposition~\ref{prop:bulk-covariance},
$\sup_N\|b_{r,N}[u]\|<\infty$ for every $r$.  Same-order proper
contractions have norm at most $\Delta_N[u]^{1/2}$.  If kernels $f_N,g_N$
have respective orders $m<n$ and $1\le\ell<m$,
Lemma~\ref{lem:contraction-CS-crossover} bounds the square
of the cross-contraction by the product of two proper self-contractions;
hence its norm is $O(\Delta_N[u]^{1/2})$.  If $\ell=m$, the smaller tensor
is exhausted, one factor is bounded by $\|f_N\|^2$, and the other
is a proper self-contraction of the larger tensor; hence the norm is
$O(\Delta_N[u]^{1/4})$.  These are all terms in the derivative/product
formula.  Therefore
\begin{equation}
 \left\|
 \left\langle DB_N[u],-DL^{-1}B_N[u]\right\rangle
 -\E[B_N[u]^2]\right\|_2
 \le C\Delta_N[u]^{1/4}\longrightarrow0.
 \label{eq:bulk-stein-bracket}
\end{equation}
Proposition~\ref{prop:bulk-covariance} identifies the limiting variance.
Malliavin integration by parts, applied to the characteristic function,
then gives a centered Gaussian limit.  Cram\'er--Wold proves the Brownian
finite-dimensional distributions.

For an active kernel $a_N$ of order two and a bulk kernel $g_N$ of order
$m\ge4$, Lemma~\ref{lem:contraction-CS-crossover} gives
\begin{align*}
 \|g_N\otimes_1a_N\|^2
 &=\langle g_N\otimes_{m-1}g_N,a_N\otimes_1a_N\rangle,\\
 \|g_N\otimes_2a_N\|^2
 &=\langle g_N\otimes_{m-2}g_N,a_N\otimes a_N\rangle.
\end{align*}
The active kernels are bounded and both bulk self-contractions are proper,
so the mixed Malliavin bracket tends to zero.  To make the independence
step explicit, put
\[
 \Phi_N(s,t)=\E\exp\{is\mathcal A_N[u]+it\mathcal B_N[u]\}.
\]
Malliavin integration by parts gives, uniformly for $(s,t)$ in compact
sets,
\[
 \partial_t\Phi_N(s,t)
 =-t\sigma_N^2\Phi_N(s,t)
 -s\E\!\left[e^{is\mathcal A_N[u]+it\mathcal B_N[u]}
 \langle D\mathcal A_N[u],-DL^{-1}\mathcal B_N[u]\rangle\right]+o(1).
\]
The mixed term tends to zero, $\sigma_N^2$ tends to the Brownian variance,
and $\Phi_N(s,0)$ converges by Proposition~\ref{prop:uniform-active}.
Solving the limiting scalar ODE yields the product characteristic
function.  This proves \eqref{eq:active-bulk-independent}.
\end{proof}

\begin{proposition}[Uniform block bounds]
\label{prop:crossover-blocks}
For every consecutive block $I$ of $m\le NT$ indices,
\begin{align}
 \E|\mathcal A_N(I)|^2
 &\le C\left[
 h_N\frac mN
 +\lambda_N^2\left(\frac mN\right)^{2-2\theta}
 \right],
 \label{eq:active-crossover-block}\\
 \E|\mathcal B_N(I)|^2
 &\le C\frac mN
 +\frac{Ch_N^2}{N}\sum_{r=2}^qS_m(2r\theta).
 \label{eq:bulk-crossover-block}
\end{align}
If $(\lambda_N)$ is bounded, polygonal interpolations of the pair are tight
in $C^\eta([0,T])$ for every $\eta<1/2$.  If $\lambda_N\to\infty$, then
$\lambda_N^{-1}\mathcal A_N$ has polygonal interpolations tight in every
$C^\eta([0,T])$, $\eta<1/2$, its step version has the same $D$-limits, and
$\lambda_N^{-1}\mathcal B_N\to0$ in $D([0,T])$.
\end{proposition}

\begin{proof}
Sum \eqref{eq:active-covariance-envelope} and
\eqref{eq:bulk-covariance-envelope} over the block.  For the active long
tail,
\[
 \frac{h_N^2}{N}S_m(2\theta)
 \le C\lambda_N^2(m/N)^{2-2\theta}.
\]
For the bulk, the fragmented packet is bounded by $Cm/N$ because
$\|\rho_h\|_{\ell^1}=2$, and the slowest residual exponent is $4\theta$.
This proves \eqref{eq:active-crossover-block}--
\eqref{eq:bulk-crossover-block}.

Put $\delta=m/N$ and recall
\[
 S_m(a)\le C
 \begin{cases}
 m^{2-a},&a<1,\\
 m\log(2m),&a=1,\\
 m,&a>1.
 \end{cases}
\]
If $\lambda_N\le L$, then $h_N\le LN^{-1/2+\theta}$.  The $r=2$ residual
in \eqref{eq:bulk-crossover-block} is $O_L(\delta)$: after division by
$\delta$, the three cases $4\theta<1$, $4\theta=1$, and $4\theta>1$ are
bounded respectively by
\[
 N^{-1+2\theta}m^{1-4\theta},\qquad
 N^{-1/2}\log N,\qquad h_N^2.
\]
The higher $r$ are smaller.  Likewise
$h_N\delta+\lambda_N^2\delta^{2-2\theta}\le C_{L,T}\delta$.
Hypercontractivity of the fixed finite sum of chaoses therefore gives, on
the grid,
\[
 \E|\Delta_I\mathcal A_N|^p+
 \E|\Delta_I\mathcal B_N|^p\le C_{p,L,T}\delta^{p/2}.
\]
Linear interpolation inside a cell preserves this estimate.  Choosing
$p$ arbitrarily large in Kolmogorov's criterion gives tightness in every
little H\"older space $c^\eta$, $\eta<1/2$.  Moreover
\[
 \P\!\left(\max_{i\le NT}|\Delta_i\mathcal A_N|
 +\max_{i\le NT}|\Delta_i\mathcal B_N|>\varepsilon\right)
 \le C_{p,\varepsilon}N^{1-p/2}\to0
\]
for $p>2$, so step and polygonal versions have the same limits.

Suppose now that $\lambda_N\to\infty$.  Dividing
\eqref{eq:bulk-crossover-block} by $\lambda_N^2$ gives
\begin{equation}
 \E|\lambda_N^{-1}\mathcal B_N(I)|^2
 \le
 \begin{cases}
 C\lambda_N^{-2}\delta+CN^{-2\theta}\delta^{2-4\theta},&4\theta<1,\\
 C\lambda_N^{-2}\delta+CN^{-1/2}\log N\,\delta,&4\theta=1,\\
 C\lambda_N^{-2}\delta+CN^{-1+2\theta}\delta,&4\theta>1.
 \end{cases}
 \label{eq:supercritical-bulk-blocks}
\end{equation}
In the first case hypercontractivity bounds the $L^p$ increment by
$C_p[\lambda_N^{-1}\delta^{1/2}+N^{-\theta}\delta^{1-2\theta}]$;
the other cases have exponent $1/2$.  Choose
$p\min\{1/2,1-2\theta\}>1$.  The coefficients tend to zero, the
finite-dimensional distributions tend to zero, and the same polygonal and
single-cell argument proves
$\lambda_N^{-1}\mathcal B_N\Rightarrow0$ in $D([0,T])$.

Finally \eqref{eq:active-crossover-block} divided by $\lambda_N^2$ gives
\[
 \E|\lambda_N^{-1}\mathcal A_N(I)|^2
 \le C\frac{h_N}{\lambda_N^2}\delta
 +C\delta^{2-2\theta},
 \qquad
 \frac{h_N}{\lambda_N^2}
 =\frac1{\lambda_NN^{1/2-\theta}}\to0.
\]
For grid increments $\delta\ge N^{-1}$, the first term is eventually
bounded by $\delta^{1+(1/2-\theta)}$.  Hypercontractivity therefore proves
tightness of the active polygonal family in every $c^\eta$, $\eta<1/2$,
and the single-cell argument again
identifies its step version.
\end{proof}

\section*{Statements and declarations}

\paragraph{Author contribution.}
The author conceived the project and is solely responsible for its
conceptualization, formal analysis, proofs, interpretation, and final text.

\paragraph{Funding.}
No funding was received for conducting this study.

\paragraph{Competing interests.}
The author declares no competing interests.

\paragraph{Data availability.}
No datasets were generated or analysed during the current study.

\end{document}